\documentclass[11pt]{amsart}
\usepackage{amssymb}
\usepackage{amssymb, amsmath}
\usepackage{verbatim}
\usepackage{color}
\let\d=\partial
\let\eps=\varepsilon
\let\wt=\widetilde
\let\wh=\widehat

\def\cC{{\mathcal C}}
\def\cD{{\mathcal D}}
\def\cE{{\mathcal E}}

\def\cO{{\mathcal O}}
\def\cP{{\mathcal P}}
\def\cQ{{\mathcal Q}}

\def\N{{\mathbb N}}
\def\R{{\mathbb R}}
\def\T{{\mathbb T}}
\def\Z{{\mathbb Z}}

\newcommand{\e}{\mathrm{e}}

\newcommand{\with}{\quad\hbox{with}\quad}
\newcommand{\andf}{\quad\hbox{and}\quad}

\def\virgp{\raise 2pt\hbox{,}}
\def\cdotpv{\raise 2pt\hbox{;}}

\def\div{ \hbox{\rm div}\,  }
\def\curl{ \hbox{\rm curl}\,  }

\newcommand{\Int}{\displaystyle \int}
\newcommand{\Frac}{\displaystyle \frac}

\newcommand{\ds}{\displaystyle}

\def\dr{\delta\!\rho}

\def\dv{\delta\!v}

\newtheorem{thm}{Theorem}[section]
\newtheorem{lem}{Lemma}[section]
\newtheorem{rmk}{Remark}[section]
\newtheorem{col}{Corollary}[section]
\newtheorem{prop}{Proposition}[section]

\newcommand{\ben}{\begin{eqnarray}}
\newcommand{\een}{\end{eqnarray}}
\newcommand{\beno}{\begin{eqnarray*}}
\newcommand{\eeno}{\end{eqnarray*}}
\numberwithin{equation}{section}

\begin{document}
\title[]{Compressible Navier-Stokes equations with ripped density}
\author[R. Danchin]{Rapha\"{e}l Danchin}
\address[R. Danchin]{Universit\'{e} Paris-Est,  LAMA (UMR 8050), UPEMLV, UPEC, CNRS,
 61 avenue du G\'{e}n\'{e}ral de Gaulle, 94010 Cr\'{e}teil Cedex, France.} \email{raphael.danchin@u-pec.fr}
\author[P.B. Mucha]{Piotr Bogus\l aw Mucha}
\address[P.B. Mucha]{Instytut Matematyki Stosowanej i Mechaniki, Uniwersytet Warszawski, 
ul. Banacha 2,  02-097 Warszawa, Poland.} 
\email{p.mucha@mimuw.edu.pl}

\begin{abstract} 

%{\color{blue}
We are concerned with  the Cauchy problem for 
the two-dimensional compressible Navier-Stokes equations 
supplemented with general  $H^1$ initial velocity and  bounded initial 
density  not necessarily strictly positive: it may be the characteristic function of any set, for instance. 

In the  perfect gas case, we establish global-in-time existence and uniqueness,
 provided the volume (bulk) viscosity coefficient  is large enough.
For more general pressure laws (like e.g. $P=\rho^\gamma$ with $\gamma>1$),
we still get global existence, but uniqueness remains  an open question.
 
   As a by-product of our results, we give a rigorous justification of the convergence to the inhomogeneous 
  incompressible Navier-Stokes equations when the bulk viscosity tends to infinity.
  
    In the three-dimensional case, similar results are proved for short time  without
   restriction on the  viscosity, and for large time if  the initial velocity field
   is small enough.
 \end{abstract}
\maketitle

\section*{Introduction}

 Systems of PDEs coming from classical physics are sources of  never-ending challenges for mathematicians.
%thanks to their rigidness in the statement they are treated in our analysis as a kind of a set of irremovable axioms. 
This is the case of  Euler and Navier-Stokes  systems that are at
the basis of  fluid mechanics. 
In that field, a number of new and sometimes unexpected results flourished in the last decade. 
One can for instance mention the works by C. De Lellis  and L. Sz\'ekelyhidi \cite{DLSz1,DLSz2}
where a technique based on convex integration is used  to  construct  
infinitely many finite energy solutions to the classical incompressible Euler equations. 
Since their energy may be any nonnegative smooth function of
the time variable, those solutions are not physically relevant. For that reason, they are often named  wild
 (or even spam) solutions. Convex integration turned out to be robust enough 
 so as to  be adapted to other PDEs for inviscid flows (e.g. the compressible Euler system  \cite{CHK}) and even  to  models with diffusion like  the classical Navier-Stokes system \cite{BV}. 
% For all those systems however,  the produced solutions  violate the expected energy inequality. 
\smallbreak
In accordance with Laplace determinism principle
(or, in mathematics, with Hada\-mard's definition of well-posedness), 
it is natural to look for conditions on the initial data
ensuring uniqueness, global existence and stability by perturbations. However, even for rather simple physical systems,  the full answer is not often known.
  %just for simple and reduced case,  like mono-dimensional \cite{Kaz}. 
  In this regards, one may mention the celebrated Millenium Problem 
  dedicated to the global  regularity of solutions to the incompressible Navier-Stokes equations in the three dimensional 
case\footnote{See {\tt http://www.claymath.org/millennium-problems}.}.
 So far,   there is no consensus in the community  on whether  the solutions are unique or not, 
regular  or not, for all time. 
  Positive answer is known  for the two space dimensional case, 
  after the work by  O.A. Ladyzhenskaya \cite{Lad} in 1958, 
 that states that weak solutions to the Navier-Stokes equation 
 are unique, stable, and  smooth if the  data are smooth.
 Up  to some small variations (like e.g. viscous flows with  variable density  or coupling 
 with a transport equation through a buoyancy term),  that case is essentially the only 
 one in  classical mathematical fluid mechanics where a complete well-posedness theory 
is available.
\medbreak
In the present paper, we would like to address the global well-posedness issue 
for  the compressible Navier-Stokes system  in the barotropic regime, 
supplemented with  
%we aim at proving the global in time unique solutions to 
general arbitrary large initial data 
with merely bounded (and nonnegative)  density : we have in mind  a ``ripped'' initial density, that
  is a function  that  may have  nontrivial  regions of vacuum, without any extra  regularity assumption.
\smallbreak
The main achievements here  are as follows:
\begin{itemize}
\item  in the two-dimensional case,
for any nonnegative initial density,  just bounded, and  any 
initial  velocity $v_0$ in  $H^1,$  if  the bulk  viscosity $\nu$ is large enough and $\|\div v_0\|_{L^2}=\cO(\nu^{-1/2}),$ 
then  there exists at least one global solution with uniformly bounded density;
% and velocity field with  divergence and vorticity in $L_{r,loc}(\R_+;L_\infty)$ for some $r>1$;

\smallskip 

\item  in the case of a  perfect gas, namely if the pressure is positively proportional to  the density, then
 the above solutions are unique (here  the bulk viscosity need not be large and $\div v_0$ need not be small);
 
\smallskip  
 
%$\clubsuit$ -- as a consequence we obtain local in time well-posedness for the system for arbitrary bulk density in the case of the perfect gas;
\item  we justify  rigorously   the (singular) limit of the compressible Navier-Stokes system to the inhomogeneous incompressible one  as $\nu$
tends to infinity (regardless of the fact that  the density may vanish and have large variations);

\smallskip 

%$\clubsuit$ -- in the case of regular initial density we obtain uniqueness for arbitrary monotone pressure function. [it is obvious??] It is a consequence of new simpler proof of the existence of the 
%constructed solutions;

\item  the above  results remain true  in the  three-dimensional case, provided a suitable smallness condition is prescribed on  the whole initial velocity $v_0.$
\end{itemize}

As our goal is  to  consider as
general densities as possible, we do not strive for 
optimal regularity hypotheses  on the initial velocity, and take it in $H^1,$ for simplicity. 
Although the density of the solution is only a bounded function, 
the corresponding velocity has relatively high regularity. However, we do not reach  the $L_{1,loc}(\R_+;C^{0,1})$ regularity so that the classical methods for showing uniqueness fail.
%, like $L_{r,loc}(\R_+;BMO)$ for some $r>1.$ However it  need not be in $L_{1,loc}(\R_+;C^{0,1}).$
Nevertheless, we establish uniqueness  in the regime of a perfect gas, 
 and  obtain some qualitative results  on  the regions of vacuum: 
their growth  or decrease is controlled in terms of the data and of the time, 
they are  stable and vacuum cannot  appear if the initial density is positive
(or cannot  disappear if the initial density vanishes on some set with positive measure). 
In the case where the initial density is the characteristic function of a  set,  
our results provide us with some information on the regularity of the boundary 
of the support of the density for positive times, even though the flow
is not quite Lipschitz.  

In addition,  our solutions are physical: total mass and momentum are conserved, and 
 the energy balance is fulfilled for all time. As a consequence, in the case 
of  zero energy initial data (which does not mean that the initial velocity is zero 
since it may be anything in the regions of  vacuum), the only possible solution has null velocity
instantaneously.

%As a corollary of our results we distinguish a result concerning the density patches problem. Let initial density be a characteristic function of a set $A$ with smooth boundary
%then the support of the density stay in the set $A(t)$ and the regularity of 
%$\partial A(t)$ stays in $C^\alpha$ with any $\alpha <1$. It is a consequence of the fact that the velocity field is not $C^1$ in space, the merely information we have is the divergence of the flow is 
%ust bounded, and in the case of rough density must not be improved.

%%%%%%%%%%%%%%%%%%%%%%%%%%%%%%%%%%%%%%%%%%%

\section{The results}

We are concerned with the following barotropic compressible 
Navier-Stokes equations  in the unit torus $\T^d$ with  $d=2,3$:
\begin{equation}\label{CNS}
\left\{ \begin{array}{lcr}
  \rho_t + \div (\rho v) =0 & \mbox{ in } & \R_+\times\T^d,\\[1ex]
  (\rho v)_t + \div(\rho v\otimes v) -\mu\Delta v - (\lambda+\mu) \nabla \div v + \nabla P =0 & \mbox{ in } & \R_+ \times \T^d.
 \end{array}\right.
\end{equation}
The pressure  $P$ is a given function of the density. The real numbers $\lambda$ and $\mu$ designate 
the bulk and shear viscosity coefficients, respectively,  and are assumed to satisfy
\begin{equation}\label{eq:posvisc}
\mu>0\quad\hbox{and}\quad \nu:=\lambda+2\mu>0.
\end{equation}
The system is supplemented with the initial data
\begin{equation}
 v|_{t=0} = v_0, \qquad \rho|_{t=0} = \rho_0.
\end{equation}
It is obvious that the 
 total mass and momentum of smooth enough 
 solutions of \eqref{CNS}  are conserved through the evolution, namely, for all $t\geq0,$
 \begin{equation}\label{eq:conservation}
 \int_{\T^d}\rho(t,x)\,dx=\int_{\T^d}\rho_0(x)\quad\hbox{and}\quad
 \int_{\T^d}(\rho v)(t,x)\,dx=\int_{\T^d}(\rho_0v_0)(x)\,dx.
 \end{equation}
 For expository purposes, we shall  always assume  that\footnote{This is   not restrictive,  as  one can rescale the density function and use  the Galilean invariance of the system to have those two conditions satisfied.} 
 \begin{equation}\label{eq:normalization}
  \int_{\T^d}\rho_0(x)\,dx=1\andf\int_{\T^d}(\rho_0v_0)(x)\,dx=0.
  \end{equation}
Next,  if we denote by $e$ the potential energy of the fluid defined, up to an affine function,  by 
the relation $\rho e''=P',$  and introduce the total energy   
$$E(t):= \int_{\T^d}\Bigl(\frac12 \rho(t,x)|v(t,x)|^2+e(\rho(t,x))\Bigr)dx,$$
then    (still for smooth enough solutions)  the following energy balance holds true:
 \begin{equation}\label{eq:energy}
 E(t) +\int_0^t \bigl(\mu\|\nabla\cP v(\tau)\|^2_{2} + \nu \|\div v(\tau)\|_{2}^2\bigr)d\tau =
   E_0:=E(0),\end{equation}
 where $\cP$ denotes the $L_2$-projector onto the set of  solenoidal vector-fields and
 $\|\cdot\|_p,$ the norm  in $L_p(\T^d).$ 
\medbreak
Since the pioneering works by P.-L. Lions in \cite{Lions} and E. Feireisl in \cite{feireisl}
(see also the  paper by D. Bresch and P.-E. Jabin \cite{BJ}
that uses recent achievements of the transport theory), it is well understood 
that in the case of an isentropic pressure law $P(\rho)=\rho^\gamma$ with $\gamma>d/2,$ 
any finite energy initial data  generates a global-in-time weak solution to \eqref{CNS} satisfying
 \begin{equation}\label{eq:IE}
 E(t) +\int_0^t \bigl(\mu\|\nabla\cP v(\tau)\|^2_{2} + \nu \|\div v(\tau)\|_{2}^2\bigr)d\tau\leq E_0\quad\hbox{for all} \quad t\geq0.\end{equation}
However, even in the two-dimensional case, it is not clear that those 
solutions respect the energy balance \eqref{eq:energy} (just inequality
is known), and the regularity and uniqueness issues  are widely open.
{}From the viewpoint of the well-posedness theory, those weak solutions 
are relevant  inasmuch as they  satisfy the so-called  
\emph{weak-strong uniqueness}  principle :  for smooth data, they coincide with the 
 corresponding smooth solution as long as  it exists (see  \cite{Fei2019,Ger}).
%  in sharp contrast with the theory for incompressible homogeneous flows (see \cite{Lad}).   
\smallbreak
Regarding the well-posedness issue, there is a number of results 
in the case of smooth  density \emph{bounded away from zero}
(some of them like \cite{S} being obtained much before the construction of weak solutions). 
The general rule is that  the solutions are known to exist for small time  if the data are large (see e.g. \cite{D01,Itaya,Mu03,S}) and for all time  if    the data are small perturbations of a linearly stable  constant state  (see \cite{D00,MaNi, Mu01}). It has been observed by Y. Cho, H.J. Choe and H. Kim \cite{CCK}
that positivity of density may be somewhat relaxed if 
a suitable compatibility condition involving the initial velocity
and high regularity of the density are guaranteed. 
Let us finally mention that for viscosity coefficients  that depend on the density
in a very specific  way, one can achieve global existence of strong solutions
in dimension two,  even for large data, if $\gamma>3$ (see \cite{Veig}). 
\smallbreak
At the end let us  mention the work by D. Hoff in \cite{Hoff1} devoted to the  construction of
 ``intermediate'' solutions in between the aforementioned weak solutions  and the more regular ones, 
   that may have discontinuous density along some curve ($d=2$) or surface ($d=3$). 
\medbreak
%Our first aim here  is to adapt our work in \cite{DM-1} dedicated to the \emph{incompressible}
%inhomogeneous Navier-Stokes equations, to the compressible Navier-Stokes equations:
We here want to provide the reader with a complete global-in-time existence theory in the case where the initial 
velocity is in $H^1(\T^d)$ and the  initial density is just bounded. 
In the two dimensional case, we shall achieve our goal provided 
 that $\nu^{1/2}\|\div v_0\|_2\leq K$ for some given $K>0$ and that $\nu$ is large enough
 (the assumption on $\div v_0$ comes up naturally 
 when defining a suitable energy functional that controls the $H^1$ regularity). 
 A remarkable feature of our result  is that, even though the density is rough and need not be positive,  one can exhibit some parabolic gain of regularity for the velocity, which  entails that both $\div v$ and $\curl v$ are almost  in $L_{2,loc}(\R_+;L_\infty).$  Although this does not  quite imply that the full gradient of $v$ is  in $L_{1,loc}(\R_+;L_\infty),$
we will  get uniqueness in the case where $P(\rho)=\rho.$
 %despite the fact that System \eqref{CNS}  is quasilinear and partially hyperbolic. 
\medbreak
Let us first state our global existence result in the two-dimensional case.
\begin{thm}\label{thm:global2}  
Assume that the pressure law is $P(\rho)=a\rho^\gamma$ for some $a>0$ and $\gamma\geq1.$
Fix some $K>0,$ and consider any vector field $v_0$ in $H^1(\T^2)$ satisfying $\|\div v_0\|_2\leq K\nu^{-1/2}$
and nonnegative bounded function $\rho_0$ fulfilling  \eqref{eq:normalization}.

There exists  a positive number $\nu_0$ depending only on $K,$ $\gamma,$ $\mu,$ $E_0,$ $\|\nabla v_0\|_2$ and  $\|\rho_0\|_\infty$
such that if $\nu\geq\nu_0$ then System \eqref{CNS} admits a global-in-time solution 
$(\rho,v)$ fulfilling the conservations laws \eqref{eq:conservation}, the 
energy balance \eqref{eq:energy}, 
$$\rho\in L_\infty(\R_+\times\T^2)\cap \cC(\R_+;L_p(\T^2)),\quad p<\infty
\andf\sqrt \rho\, v\in\cC(\R_+;L_2(\T^2)).$$
In addition, we have, denoting $\dot v:=v_t+v\cdot\nabla v$ and $G:=\nu\div v-P,$  
\begin{equation}\label{eq:regT2}
\begin{array}{c}v\in L_\infty(\R_+;H^1(\T^2)),\quad(\nabla^2\cP v,\nabla G,\sqrt\rho\,\dot v)\in L_2(\R_+\times\T^2),\\[1ex]
\sqrt{\rho t}\,\dot v\in L_{\infty,loc}(\R_+;L_2(\T^2)),\quad
\sqrt t\,\nabla\dot v\in L_{2,loc}(\R_+;L_2(\T^2)),\end{array}
\end{equation}
and  both $\div v$ and $\curl v$ are in $L_{r,loc}(\R_+;L_\infty(\T^2))$ for all $r<2$.
\end{thm}
\begin{rmk} For simplicity, we focussed on  the physically relevant case where the
pressure function $P$ is given by $P(\rho)=a\rho^\gamma$ for some $\gamma\geq1$ and $a>0.$
However, the above theorem  remains true whenever: 
\begin{equation}\label{eq:condP}
P\ \hbox{  is a }\  C^1\ \hbox{ nonnegative function on }\  \R_+\ \hbox{ such that }\ \rho\mapsto \rho^{-1}P(\rho)\hbox{ is nondecreasing.}\end{equation}
%{\bf R. V\'erifier dans \'etape 2 les conditions suppl\'ementaires.}
\end{rmk}
In the case of a linear pressure law, our existence result is supplemented with uniqueness.
\begin{thm}\label{thm:uniqueness} Assume  $P(\rho)=a\rho$ for some $a>0.$
Then, for any $T>0,$ any nonnegative $\rho_0$ in $L_\infty(\T^2)$  and $v_0$ in $H^1(\T^2),$
and any viscosity coefficients $(\lambda,\mu),$ there exists at most one  solution to System \eqref{CNS} 
supplemented with  data $(\rho_0,v_0)$ on $[0,T]\times\T^2,$ with the regularity given 
 in Theorem \ref{thm:global2}   (restricted to interval $[0,T]$). 
\end{thm}
Since the norms  of the solution constructed  in 
Theorem~\ref{thm:global2} may be  bounded uniformly  
with respect to $\nu,$  one gets almost for free  the all-time convergence
when $\nu$ tends to $+\infty$  to the following inhomogeneous incompressible Navier-Stokes equations:
\begin{equation}\label{INS}
\left\{ \begin{array}{lcr}
  \rho_t + \div (\rho v) =0 & \mbox{ in } & \R_+\times\T^2,\\[1ex]
  (\rho v)_t + \div(\rho v\otimes v) -\mu\Delta v + \nabla \Pi =0 & \mbox{ in } & \R_+ \times \T^2,\\[1ex]
  \div v=0 & \mbox{ in } & \R_+ \times \T^2.
 \end{array}\right.
\end{equation}
Let us give  the  statement in the case of a fixed initial data (for simplicity):
\begin{thm}\label{thm:asymptotic}
 Fix some data $(\rho_0,v_0)$ in $L_\infty(\T^2)\times H^1(\T^2)$ satisfying
 $\div v_0=0$ and $\rho_0\geq0,$  and denote  by $(\rho^\nu,v^\nu)$ the corresponding
global solution of \eqref{CNS} of Theorem \ref{thm:global2} for $\nu\geq \nu_0.$

Then, for $\nu$ going to $\infty,$
 the whole family $(\rho^\nu,v^\nu)$ converges to the unique global solution 
of System \eqref{INS} with initial data $(\rho_0,v_0)$ given by Theorem 2.1 of \cite{DM-1}, 
and we have
\begin{equation}\label{eq:div}
\div v^\nu=\cO(\nu^{-1/2})\  \hbox{ in }\ L_2(\R_+\times\T^2)\cap L_\infty(\R_+;L_2(\T^2)),\end{equation}
and even $\div v^\nu=\cO(\nu^{\varepsilon-1})$ in $L_{2-\eps,loc}(0,T;L_\infty(\T^2))$ for all 
$\eps\in(0,1)$ and $T>0.$
\end{thm}
\begin{rmk} To the best of our knowledge, Theorem \ref{thm:asymptotic} is the  first example 
of a global-in-time result  of convergence from \eqref{CNS}
to \eqref{INS} \emph{in the truly inhomogeneous framework} (see  our recent work in \cite{DM-TJM}
for an example of almost global convergence).   
\end{rmk}

%\begin{thm}\label{thm:asymptotic} Fix $\mu>0$.
%Take an initial data $(\rho_0,v_0)$ fulfilling the hypotheses of Theorem \ref{thm:global2} with, in addition, $\div v_0=0$
%and denote by $(\rho^\nu,v^\nu)$ the corresponding solution to \eqref{CNS}.  Then 
%$(\rho^\nu,v^\nu)$ tends weakly to the unique  global solution $(\rho,v)$ to \eqref{INS}. \end{thm}

\begin{rmk}\label{rm-3}
The above  results of existence, uniqueness and convergence  are  valid in $\T^3$ either 
locally in time for large data, or globally under a suitable scaling invariant smallness condition
on the velocity (no smallness is required for the density). 
The reader is referred to Appendix \ref{s:d=3} for  more details. 
\end{rmk}

Let us report on the main ideas leading to our results in dimension two.  
Assuming that we are given a solution $(\rho,v)$ to \eqref{CNS}, 
the first step  is to establish  global-in-time  a priori estimates for the $H^1$ norm of $v$ in terms of the data, of the parameters of the system and of  a (given) upper bound of the density. 
The overall strategy  has some similarities with our recent 
work  \cite{DM-1}
dedicated to System \eqref{INS}. However, the compressible situation 
is more complex (as the reader will judge by himself in the next section) since
one cannot expect $\nabla P$ to be in any Lebesgue space.   For that reason, we shall 
consider the   \emph{viscous effective flux} $G$ defined by 
 \begin{equation}\label{eq:G} 
 G:=\nu\,\div v -P\quad\hbox{with }\ \nu:=\lambda+2\mu\end{equation}
 since it has better regularity  than $\div v$ or $P$ taken separately, 
as observed before  by D. Hoff \cite{Hoff1} and P.-L. Lions \cite{Lions} when constructing
intermediate or  weak solutions.
 Rewriting the momentum equation in terms of $G$ and $\cP v$ (the divergence-free  part of $v$) will spare us   making  integrability assumptions on $\nabla\rho,$  in contrast with our recent work in \cite{DM-TJM}.
 
 The second key ingredient of that step is the following  logarithmic interpolation inequality 
\begin{equation}\label{eq:interpo-des}
\biggl(\int_{\T^2} \rho|v|^4\,dx\biggr)^{\frac12}\leq C \|\sqrt\rho v\|_{2}
\|\nabla v\|_{2}\log^{\frac12}\biggl(\e+\|\wt\rho\|_{2}
+\frac{\|\rho\|_{2}\|\nabla v\|_{2}^2}{\|\sqrt \rho v\|_{2}^2}\biggr)\with \wt\rho:=\rho-1,
\end{equation}
 that has been discovered by B. Desjardins \cite{Des-CPDE} and  is an appropriate substitute of the well-known Ladyzhenskaya inequality
$$\|v\|_4^2\leq C\|v\|_2\|\nabla v\|_2$$
since  bounds are available on  $\|\sqrt\rho v\|_{2}$ (through \eqref{eq:energy}), but  not on $\|v\|_2.$

Then, the main idea is to introduce a suitable modified energy functional that
contains informations on  the $H^1$ norm of $v,$ 
 and may be bounded uniformly on $\R_+.$  
%  This follows the arguments employed by B. Desjardins in   \cite{Des-CPDE}. 
Our definition
% of a modified energy functional is slightly   different, which 
enables us, after tracking carefully the dependency of the estimates
  with respect to the viscosity coefficients, to 
  exhibit   global-in-time bounds depending only on the data and on  $\rho^*:=\|\rho\|_\infty,$ 
if $\nu$ is large enough (results in \cite{Des-CPDE} were local). 

\smallbreak
The goal of the second step  is to bound $\rho^*$ in terms of the data.  
As in \cite{Des-CPDE},  we shall rather consider the following quantity 
$$F:=\log\rho -\nu^{-1}(-\Delta)^{-1}\div (\rho v)$$
that may be seen as an approximate damped mode 
associated to \eqref{CNS}.  
The new achievement  here is that, by combining with the first step and 
a  bootstrap argument, one gets  a \emph{global-in-time} control on $\rho^*$ in terms of the data only, provided that $\nu$ is large enough.
\smallbreak
%In order to have a chance to prove uniqueness of the solutions, we need to exhibit more 
%regularity for $\nabla v.$ Ideally, since the system 
%under consideration is partially quasilinear hyperbolic of order one, 
%it would be good to have $\nabla v$   in $L_{r,loc}(\R_+;L_\infty).$
Step 3  aims at  proving  that   $\div v$ and $\curl v$   are  in $L_{r,loc}(\R_+;L_\infty)$
for some $r>1.$  To achieve it,  the general idea
is to  use time weighted estimates to glean some regularity on $v_t,$ then 
to transfer time regularity
to space regularity thanks to  elliptic estimates  and functional embeddings.
However, in contrast with the incompressible case studied in \cite{DM-1}, 
 it is no longer possible to discard the pressure term 
by means of the divergence free property,  and it is actually more
appropriate to work with the \emph{convective derivative $\dot v:=v_t+v\cdot\nabla v$}.
   In the end, we shall get
 bounds on $\sqrt{\rho t}\,\dot v$ in $L_{\infty,loc}(\R_+;L_2)$ and $\sqrt t\nabla\dot v$ in $L_{2,loc}(\R_+;L_2),$ 
 from which we will eventually bound   $\div v$ and $\curl v$  in $L_{r,loc}(\R_+;L_\infty).$
\smallbreak
Steps 1 to 3 were  formal a priori estimates  for smooth solutions. 
To complete the proof of existence, we mollify the initial density
so as to make it strictly positive and  regular. 
Then, one can resort to  classical results to construct a local-in-time smooth solution 
corresponding to those data. The difficulty  is to establish that, indeed, 
the control of norms that has been obtained so far  allows to extend the solution for all time. 
Once it has been done,  the uniform bounds given by steps 1 to 3  allow 
to pass to the limit and to complete the proof of the global existence. 
In fact, since  compared to weak solutions theory, 
 more regularity is available on the velocity,  passing to the limit is much more direct than in \cite{feireisl} or \cite{Lions}.
 Furthermore, as the bounds from steps 1 to 3 have some uniformity with respect to $\nu,$
 similar arguments allow to justify the convergence of \eqref{CNS} to \eqref{INS}, whence Theorem \ref{thm:asymptotic}.
\smallbreak
Since steps 1 to 3  just give that $\div v$ and $\curl v$ are 
in $L_{r,loc}(\R_+;L_\infty)$  for some $r>1,$
 we miss by a little the property that $\nabla v$ is in $L_{1,loc}(\R_+;L_\infty)$ 
and   $v$ need not have a Lipschitz flow. Therefore,  in contrast with what has been done for \eqref{INS}
in \cite{DM-1}  or
for \eqref{CNS} in \cite{DFP},  it is not clear whether recasting the compressible Navier-Stokes equations
in Lagrangian coordinates may help to prove uniqueness. 
However, we know from the previous steps  
that $\nabla\cP v$ and $G$  are in $L_{2,loc}(\R_+;H^1),$ whence
$\nabla v$ is in $L_{2,loc}(\R_+;BMO).$
In the particular case of a \emph{linear} pressure law,  this turns out to 
be enough to control the difference of two solutions   in $L_\infty(0,T;\dot H^{-1})$ 
for the density and $L_2(0,T;L_2)$ for the velocity. 
The proof has some similarities with that of D. Hoff in \cite{Hoff2} but
does not require  Lagrangian coordinates. In fact,  we overcome  that 
$\nabla v  \notin L_{1,loc}(\R_+;L_\infty)$ by combining the information that 
$\nabla v\in L_{2,loc}(\R_+;BMO)$ with a suitable logarithmic interpolation 
inequality from \cite{M10}.
\medbreak
Let us finally point out an interesting  application of Theorem  \ref{thm:global2}
pertaining to the case where the initial density has nontrivial vacuum regions. 
\begin{col}\label{c:patch}
 Let the assumptions of Theorem \ref{thm:global2} be in force, and denote by  $(\rho,v)$ 
  a global solution given by Theorem \ref{thm:global2}.  Let $X$ be the  (generalized) flow of $v,$ defined  by  
  \begin{equation}\label{eq:col1} X(t,y)= y+ \int_0^t  X(\tau,v(\tau,y))\,d\tau,\quad t\geq0,\ y\in\T^2.\end{equation}
Then, the following results hold:
\begin{enumerate}
\item Let  $V_0:=\rho_0^{-1}(\{0\}).$ Then  $\rho_t^{-1}(\{0\})=V_t$
  with   $V_t:=X(t,V_0).$  Furthermore, if $V_0$ is an open set with Lipschitz boundary, 
 then $V_t$ is an open set with  $C^{0,\alpha_t}$ regularity. 
  \item  If  $ \rho_0=1_{A_0}$ and $A_t:=X(t,A_0),$ 
 then $\inf_{x\in A_t} \rho(t,x)>0$ for all $t>0.$ Furthermore, if 
  $A_0$ is a Lipschitz open set, then  $A_t$ has  $C^{0,\alpha_t}$ regularity.  
\end{enumerate}
Above, $\alpha_t>0$ is a continuously decreasing function of $t$
 and is such that $\alpha_0=1.$ 
  \end{col}
\begin{proof}
By using the continuity of Riesz operator, we get
$$\|v\|_{LL}\leq \|v\|_{2}+\|\div v\|_{\infty}+\|\curl v\|_{\infty},$$
 where $LL$ stands for the space of bounded log-Lipschitz functions.
 Hence Theorem \ref{thm:global2}
ensures that $v\in L^1_{loc}(\R_+;LL)$ and, applying  \cite[Th. 3.7]{BCD}
guarantees the existence and uniqueness of a generalized flow fulfilling \eqref{eq:col1}. 
Let us assume that  $\d V_0$ coincides with $\phi_0^{-1}(\{0\})$ for some 
Lipschitz function $\phi_0:\T^2\to\R.$ 
Then, $\d V_t=\phi_t^{-1}(\{0\})$ where $\phi$ solves  the transport equation 
$$
(\d_t+v\cdot\nabla)\phi=0.
$$
Now, \cite[Th. 3.12]{BCD} guarantees that $\phi_t$ has regularity $C^{0,\alpha_t}$ with 
$$\alpha_t:=\exp\biggl(-\int_0^t\|v\|_{LL}\,d\tau\biggr),$$
and one can conclude the proof of the first item.
\medbreak
The second item follows from similar arguments and  from the fact that
$$\rho(t,X(t,y))=\rho_0(y)\exp\biggl(\int_0^t\div v(\tau,X(\tau,y))\,d\tau\biggr)$$
for all $t\geq0$ and $y\in\T^2.$ 
\end{proof}

\medbreak
We end this  part   proposing some  conjecture, 
that probably requires  further developments of the transport theory, 
the basic problem here being  that the velocity field is not Lipschitz, thus preventing us to reformulate 
the equations in the Lagrangian coordinates without any loss of regularity:
\smallbreak \noindent{\it {\bf Conjecture.} The solutions constructed in Theorem \ref{thm:global2} 
(or in Theorem \ref{thm:global3} for the three-dimensional case)
 are unique for arbitrary  strictly increasing convex pressure functions.}

%It seems that nowadays techniques does not fit to the above problem, we are required
%to build a suitable new framework to work with the Lagrangian coordinates in the low
%regularity for this type of parabolic systems.

\medbreak
The rest of the paper unfolds as follows. 
The next section  is dedicated to the proof of regularity estimates for \eqref{CNS}  assuming 
that the solution under consideration is smooth with density bounded away from zero,
 and that $\nu$ is large enough 
(this corresponds to steps $1$ to $3$ above). 
%For better readability,  the most technical parts of that section  are postponed in appendix. 
In  Section \ref{s:existence}, we prove the existence part of our main theorem 
and also justify the convergence of \eqref{CNS} to \eqref{INS} for $\nu$ going to $\infty$, 
while Section \ref{s:uniqueness} is dedicated to uniqueness. 
% Finally, the proof of   the convergence for  $\nu\to+\infty$ is carried out in Section \ref{s:convergence}. 
Some technical results like, in particular, Inequality \eqref{eq:interpo-des} and time weighted estimates, 
and the case $d=3$   are presented in the appendix.

%%%%%%%%%%%%%%%%%%%%%%%%%%%%%%%%%%%%%%%%%%%%%%%%%%%%%

\section{Regularity estimates}\label{s:reg}

The present section is devoted to proving regularity estimates for the velocity 
field of a solution $(\rho,v)$ to \eqref{CNS} in $\R_+\times\T^d.$
 We focus on  $d=2,$ the three
  dimensional case being postponed in appendix. 
We show three results:   a control of the $H^1$ 
norm of the velocity, a pointwise global-in-time bound for the density and, finally,  a new
estimate for the effective viscous flux and the divergence-free part of the velocity. 
This latter estimate is based  on the \emph{shift of integrability  method}
introduced in \cite{DM-1}. 
%All together they close the regularity  bounds determined by Theorem \ref{thm:global2}.

\smallbreak
As a start, we  normalize the potential energy $e$ in such a way that  $e(1)=e'(1)=0,$  setting 
\begin{equation}\label{def:e} e(\rho):=\rho\int_{1}^\rho\frac{P(\varrho)}{\varrho^2}\,d\varrho-{P(1)}(\rho-1).
\end{equation}
Hence,  $\|e\|_1$ is essentially equivalent to $\|\rho-1\|_2^2$
and, in the case  $P(\rho)=\rho^\gamma,$  we have  
$$e(\rho)=\rho\log\rho+1-\rho\ \hbox{ if }\ \gamma=1,\quad\hbox{and }\ 
 e(\rho)= \frac{\rho^\gamma}{\gamma -1} - \frac{\gamma\rho}{\gamma-1} +1\ \hbox{ if }\ \gamma>1.$$
 We shall often use the notations $e$ and $P$
  instead of $e(\rho)$ and $P(\rho).$

\subsection{Sobolev estimates for the velocity}

Here  we  derive a global-in-time $H^1$ energy estimate \emph{that requires only a control on $\sup\rho$}. 
%, almost exact computations, like in the soliton theory.
%The main breakthrough here is that %,by tracking carefully   the dependency  of the estimatesbwith respect to $\nu,$ 
% we achieve  global-in-time bounds for large enough $\nu.$ 
 \smallbreak
Throughout the proof,  we  denote $\wt P:=P-\bar P$ and $\wt G:=G-\bar G$ where $\bar P$ and $\bar G$ stand for the average of $P$ and $G.$ Note that we have
  \begin{equation}\label{eq:wtG}
  \wt G=\nu\,\div v-\wt P.
  \end{equation}  
     \begin{prop} \label{p:H1a}
  Consider a smooth solution $(\rho,v)$ to \eqref{CNS} on $[0,T]\times\T^2$ satisfying 
  \eqref{eq:normalization}.
  Assume that the pressure law fulfills \eqref{eq:condP} 
  and that, for some positive constant  $\rho^*,$   
 \begin{equation}\label{eq:bounded}
 0 \leq \rho(t,x) \leq \rho^*  \mbox{ \ \ for all\ } (t,x) \in  [0,T]\times \T^2. 
\end{equation}
 Let $\dot v:=v_t+v\cdot \nabla v$ be the material derivative of $v,$  and $h:=\rho P'-P.$
 There exist:
 \begin{itemize}
 \item[--] a functional $\cE$ such that
 \begin{equation}\label{eq:equivE} \cE\geq  \frac12\int_{\T^2}\biggl(\rho|v|^2+\mu|\nabla \cP v|^2+\frac1{\nu}\bigl(\wt G^2+\wt P^2)+ 2e\biggr)\,dx,\end{equation}
 \item[--] an absolute positive constant $C$,
 \smallbreak\item[--] a  positive constant $\nu_0$ depending\footnote{Here we find
 $\nu_0=\max\Bigl(\mu,\,C\sqrt{\frac{\rho^*\log(\e+\rho^*)}\mu}\,P(\rho^*),\, \frac{P(\rho^*)}{2},\,
 4\sqrt{\rho^*(1+h(\rho^*))}\Bigr)\cdotp$}
only on the pressure function $P,$ on $\mu$  and  on $\rho^*,$
\end{itemize}
such that  if  $\nu\geq\nu_0$ then  for all $t\in[0,T],$ we have 
\begin{multline}\label{eq:H1}
1+\frac1{\mu E_0}\biggl(\cE(t)+\int_0^t\cD(\tau)\,d\tau\biggr)\hfill\cr\hfill
\leq \biggl(1+\frac{\cE_0}{\mu E_0}\exp\biggl\{C\Bigl( 1 +\frac{(\rho^*)^2}{\mu^4}E_0^2\log(\e+\rho^*)\Bigr)\biggr\}\biggl)^{\exp\bigl\{C      \frac{(\rho^*)^2}{\mu^4} E_0^2\bigr\}},
\end{multline}
  with $E_0$ defined in \eqref{eq:energy} and 
  $$ \cD:= \int_{\T^2}\biggl(\frac14\rho|\dot v|^2+\frac{\mu^2}{4\rho^*}|\nabla^2\cP v|^2+\frac1{8\rho^*}|\nabla G|^2
+\frac1{4\nu}\wt P^2+\Bigl(\frac{\nu\!+\!h}2\Bigr)(\div v)^2+\frac\mu2|\nabla v|^2\biggr)\,dx. $$ 
 \end{prop}
\begin{proof}  As  in the work of B. Desjardins in  \cite{Des-CPDE},  the proof consists in 
introducing  a suitable  `energy' functional that contains $H^1$ information on the velocity, then to combine with the logarithmic 
interpolation inequality \eqref{eq:interpo-des}.  The novelty here is that we succeed
in getting  a \emph{time-independent} control 
on the solution in terms of the data and of $\rho^*.$ 

\subsubsection*{Step 1} The  goal of this step (which is  independent of the dimension $d$) is to bound:
\begin{equation}\label{eq:wtE}
\wt\cE=\frac12\int_{\T^d}\biggl(\mu|\nabla \cP v|^2+\frac1{\nu}\biggl(\wt G^2+\wt P^2+\rho\int_\rho^1\frac{P^2(\tau)}{\tau^2}\,d\tau\biggr)\biggr)dx.
\end{equation}

To achieve it,  we take the $L_2$ scalar product 
 of the momentum equation of \eqref{CNS} with $\dot v,$ and get  
 \begin{multline}\label{eq:e1}
 \int_{\T^d}  \rho |\dot v|^2 \,dx + \frac{1}{2}\frac{d}{dt} \int_{\T^d} \bigl(\mu|\nabla v|^2 + (\lambda\!+\!\mu) (\div v)^2\bigr) \,dx\\ +
 \int_{\T^d} \nabla P \cdot v_t\, dx  = \int_{\T^d}(\rho\dot v)\cdot (v\cdot\nabla v)\,dx.
\end{multline}
To handle the pressure term in the left-hand side, we start from 
\begin{equation}\label{eq:P}
P_t+\div (Pv)+h\,\div v=0.
\end{equation}
Therefore, integrating by parts yields
$$ \int_{\T^d} \nabla P \cdot v_t \, dx= -\frac{d}{dt} \int_{\T^d} P\,\div v\, dx -\int_{\T^d} h\,(\div v)^2\,dx +\int_{\T^d}  P\,v\cdot\nabla \div v\, dx.$$
Since
$$
-(\nu\,\div v)^2=P^2-G^2-2\nu P\div v\andf \nu\nabla\div v=\nabla(P+G),$$
we get after integrating by parts once to avoid the appearance of some $\nabla P$ term, 
 \begin{multline}\label{eq:e1a}
  \int_{\T^d} \nabla P \cdot v_t \, dx
  =  -\frac{d}{dt} \int_{\T^d} P\,\div v\, dx +\frac1{\nu^2}\int_{\T^d} (P^2-G^2)h\,dx+\frac1\nu\int_{\T^d} Pv\cdot\nabla G\,dx\\
-\frac1\nu\int_{\T^d}\Bigl(\frac{P^2}2+2Ph\Bigr)\div v\,dx.\end{multline}
 Observing that, owing to the definition of $G$ and to \eqref{eq:P}, we have 
$$
\bar G=-\bar P\quad\hbox{and}\quad \bar P'=-\int_{\T^d} h\,\div v\,dx,
$$ 
we find that
\begin{eqnarray}
\ds\int_{\T^d}(P^2-G^2)h\,dx&&\!\!\!\!\!\!\!\!=\ds\nu\int_{\T^d} (\wt P-\wt G)\,\div v\,h\,dx+2\nu\bar P\int_{\T^d} h\,\div v\,dx\nonumber\\\label{eq:e1b}
\ds&&\!\!\!\!\!\!\!\!=\ds\nu^2\int_{\T^d} (\div v)^2h\,dx - 2\nu\int_{\T^d} \wt G\,\div v\,h\,dx-\nu\frac d{dt}(\bar P)^2.
\end{eqnarray}
  Let the function $k$ be the unique solution of 
    $$k-\rho k'=-\frac{P^2}2-2Ph \andf  k(1)=P^2(1).$$ 
% whence\footnote{ In the case $P(\rho)=\rho^{\gamma},$ function $k$ takes the  form 
 %$ k(\rho) = \frac{4\gamma -3}{4\gamma -2} \rho^{2\gamma} + \frac{1}{4\gamma -2} \rho.$}
 Then,  we have 
   \begin{equation}\label{eq:e1c}
  -\int_{\T^d}\Bigl(\frac{P^2}2+2Ph\Bigr)\div v\,dx=\int_{\T^d}\bigl(\d_t k+\div(kv)\bigr)dx
  =\frac d{dt}\int_{\T^d} k\,dx.
 \end{equation}
 Hence, plugging \eqref{eq:e1a} and \eqref{eq:e1b} in \eqref{eq:e1c}, we obtain
 $$\displaylines{\quad
 \int_{\T^d}\nabla P\cdot v_t\,dx=\frac d{dt}\int_{\T^d}\biggl(\frac{k-(\bar P)^2}\nu- P\div v\biggr)dx
 -\frac2\nu\int_{\T^d} \wt G\,\div v\,h\,dx\hfill\cr\hfill+\int_{\T^d} (\div v)^2h\,dx+\frac1\nu\int_{\T^d} Pv\cdot\nabla G\,dx.
 \quad}$$
Now, denoting 
\begin{equation}\label{eq:tE}
\check\cE:= \int_{\T^d} \biggl(\frac\mu2|\nabla v|^2 + \frac{\lambda+\mu}2(\div v)^2 +\frac1\nu \bigl(k-(\bar P)^2\bigr) -  \wt P\,\div v\biggr)dx
\end{equation}
and   reverting to  \eqref{eq:e1}, we conclude that 
\begin{multline}\label{eq:e2}
 \frac{d}{dt}\check\cE
 + \int_{\T^d} \rho |\dot v|^2 \,dx+\int_{\T^d}(\div v)^2h\,dx\\
 =\int_{\T^d}(\rho\dot v)\cdot(v\cdot\nabla v)\,dx
 +\frac2{\nu}\int_{\T^d} \wt G\,\div v\,h\,dx-\frac1\nu\int_{\T^d} Pv\cdot\nabla G\,dx.
\end{multline}
We claim that   $\check\cE=\wt\cE.$ Indeed, we have
$$\check\cE=\int_{\T^d} \biggl(\frac\mu2\bigl(|\nabla v|^2-(\div v)^2\bigr)+\frac1{2\nu}\Bigl(
(\nu\div v)^2-2\nu\div v\,\wt P+2(k-(\bar P)^2)\Bigr)\biggr)dx$$
and 
 \begin{equation}\label{eq:k}
 k(\rho)= P^2(\rho)-\frac\rho2\int_{1}^\rho\frac{P^2(\varrho)}{\varrho^2}\,d\varrho.
  \end{equation}
%\begin{equation}\label{eq:wtE}
%\wt\cE=\frac12\int_{\T^d}\biggl(\mu|\nabla \cP v|^2+\frac1{\nu}\biggl(\wt G^2+\wt P^2+\rho\int_\rho^1\frac{P^2(\tau)}{\tau^2}\,d\tau\biggr)\biggr)dx.
%\end{equation}
In order to get a control on the right-hand side of \eqref{eq:e2}, let us rewrite 
the momentum equation  in terms of the viscous effective flux $G=\nu\div v- P$ as follows:
\begin{equation}\label{eq:stokes}
\mu\bigl(\Delta v  -\nabla \div v\bigr) + \nabla G =\rho \dot v.
\end{equation}
{}From it,  we discover that
\begin{equation}\label{eq:e3}
\mu^2\|\Delta\cP v\|_{2}^2+\|\nabla G\|_{2}^2=\| \rho\dot v\|_{2}^2\leq\rho^*\|\sqrt\rho\dot v\|_{2}^2.
\end{equation}
Since we obviously have 
 $$
  \frac2\nu\int_{\T^d}\wt G\,\div v\,h\,dx\leq \frac12\int_{\T^d} (\div v)^2h\,dx+\frac2{\nu^2}\int_{\T^d}\wt G^2h\,dx,
  $$
 equality  \eqref{eq:e2} and the fact that $\check\cE=\wt\cE$  imply  that
$$\displaylines{\quad
 \frac{d}{dt}\wt\cE+\frac14\|\sqrt\rho\,\dot v\|_{2}^2
 + \frac{\mu^2}{2\rho^*}\|\Delta\cP v\|_{2}^2+\frac1{2\rho^*}\|\nabla G\|_{2}^2 + \frac12\int_{\T^d}(\div v)^2h\,dx\hfill\cr\hfill
 \leq \frac2{\nu^2}\int_{\T^d} \wt G^2\,h\,dx-\frac1\nu\int_{\T^d} Pv\cdot\nabla G\,dx + \|\sqrt\rho v\cdot\nabla v\|_{2}^2.\quad}$$
To bound the last term in the right-hand side,  we decompose $\nabla v$ into 
\begin{equation}\label{eq:e2c}\nabla v=\nabla\cP v-\frac1\nu\nabla^2(-\Delta)^{-1}\wt G-\frac1\nu\nabla^2(-\Delta)^{-1}\wt P.\end{equation}
Hence
\begin{multline}\label{eq:e2b}
 \frac{d}{dt}\wt\cE+\frac14\|\sqrt\rho\,\dot v\|_{2}^2
 + \frac{\mu^2}{2\rho^*}\|\Delta\cP v\|_{2}^2+\frac1{2\rho^*}\|\nabla G\|_{2}^2 + \frac12\int_{\T^d}(\div v)^2h\,dx\leq
  \frac2{\nu^2}\int_{\T^d} \wt G^2\,h\,dx\\-\frac1\nu\int_{\T^d} Pv\cdot\nabla G\,dx
+  3\Bigl(\|\sqrt\rho\, v\cdot\nabla\cP v\|_2^2
+\frac1{\nu^2}\|\sqrt\rho\,v\cdot\nabla^2(-\Delta)^{-1}\wt G\|_2^2
+\frac1{\nu^2}\|\sqrt\rho\,v\cdot\nabla^2(-\Delta)^{-1}\wt P\|_2^2\Bigr)\cdotp
\end{multline}

\subsubsection*{Step 2: Bounding the right-hand side of \eqref{eq:e2b} in dimension  $d=2$}

 H\"older and Gagliardo-Nirenberg inequalities yield
$$
\Int_{\T^2} \rho|v\cdot\nabla \cP v|^2\,dx\leq C\sqrt{\rho^*}
\biggl(\int_{\T^2} \rho|v|^4\,dx\biggr)^{\frac12}\|\nabla\cP v\|_2\|\nabla^2\cP v\|_2.$$
Since  the density is not bounded from below,  in order to bound the right-hand side, one has
to   take advantage of    Inequality \eqref{eq:interpo-des}. We get 
 \begin{eqnarray}\label{eq:e14a}
3\!\Int_{\T^2} \!\rho|v\!\cdot\!\nabla \cP v|^2dx &&\!\!\!\!\!\!\!\!\!\leq\! 
 C \sqrt{\rho^*}\,\|\sqrt\rho v\|_{2}\|\nabla v\|_{2}\|\nabla\cP v\|_{2}\|\nabla^2\cP v\|_{2}
\log^{\frac12}\biggl(\e\!+\!\|\wt\rho\|_{2}\!+\!\frac{\|\rho\|_{2}\|\nabla v\|_{2}^2}{\|\sqrt \rho v\|_{2}}\biggr)\nonumber\\&&\!\!\!\!\!\!\!\!\!\leq\!\Frac{\mu^2}{4\rho^*}\|\Delta\cP v\|_{2}^2 
+\! \frac{C(\rho^*)^2}{\mu^2} \|\sqrt\rho v\|_{2}^2\|\nabla v\|_{2}^2\|\nabla \cP v\|_{2}^2
\log\biggl(\e\!+\!\|\wt\rho\|_{2}\!+\!\frac{\|\rho\|_{2}\|\nabla v\|_{2}^2}{\|\sqrt \rho v\|_{2}^2}\biggr)\cdotp
\end{eqnarray}
Arguing similarly and using the fact that $\nabla^2(-\Delta)^{-1}$ maps $L_4(\T^2)$
to itself, we get
\begin{multline}\label{eq:e14b}
\frac{3}{\nu^2}\Int_{\T^2}\rho\, \Bigl|\,v\cdot\left[ \nabla^2(-\Delta)^{-1}\wt G\right]\Bigr|^2\,dx 
\leq\Frac{1}{8\rho^*}\|\nabla G\|_{2}^2 \\+ \frac {C(\rho^*)^2}{\nu^4} \|\sqrt\rho v\|_{2}^2\|\nabla v\|_{2}^2\|\wt G\|_{2}^2\log\biggl(\e\!+\!\|\wt\rho\|_{2}\!+\!\frac{\|\rho\|_{2}\|\nabla v\|_{2}^2}{\|\sqrt \rho v\|_{2}^2}\biggr),
\end{multline}
and also,
\begin{multline}\label{eq:e4c}
\frac{3}{\nu^2}\Int_{\T^2}\rho |v\cdot\nabla^2\Delta^{-1}\wt P|^2\,dx
 %&&\!\!\!\!\!\!\!\!\!\leq C\Frac{\sqrt{\rho^*}}{\nu^2}\biggl(\int_{\T^2} \rho|v|^4\,dx\biggr)^{\frac12}\|\wt P\|_{4}^2\nonumber\\
\\\leq \frac1{4\nu}\|\wt P\|_2^2 +\Frac{C\rho^*}{\nu^3}\|\sqrt \rho\,v\|_2^2\|\nabla v\|_2^2
\log\biggl(\e\!+\!\|\wt\rho\|_{2}\!+\!\frac{\|\rho\|_{2}\|\nabla v\|_{2}^2}{\|\sqrt \rho v\|_{2}^2}\biggr)\|\wt P\|_\infty^2.
\end{multline}
Finally, we have,  thanks to  Inequality \eqref{eq:limitpoincare},
%denoting by $\bar v$ the average of $v$ and $\wt v:=v-\bar v,$
$$\begin{aligned}
-\frac1\nu\int_{\T^2} Pv\cdot\nabla G\,dx&\leq \frac1\nu P^*\|v\|_{2}\|\nabla G\|_{2}\\
&\leq \frac C\nu P^*\log^{\frac12}(\e+\|\wt\rho\|_2)\,\|\nabla v\|_2\|\nabla G\|_2.
\end{aligned}$$
Hence 
\begin{equation}\label{eq:e4e}
-\frac1\nu\int_{\T^2} Pv\cdot\nabla G\,dx\leq\frac1{8\rho^*}\|\nabla G\|_{2}^2
+C\frac{\rho^*}{\nu^2}(P^*)^2\|\nabla v\|_{2}^2\log\bigl(\e+\|\wt\rho\|_{2}\bigr)\cdotp
\end{equation}
Therefore, plugging \eqref{eq:e14a}, \eqref{eq:e14b}, \eqref{eq:e4c} %\eqref{eq:e4d}
 and \eqref{eq:e4e} in \eqref{eq:e2b},  we conclude that
\begin{multline}\label{eq:e8}
 \frac{d}{dt}\wt\cE+\frac14\|\sqrt\rho\,\dot v\|_{2}^2
 + \frac{\mu^2}{4\rho^*}\|\Delta\cP v\|_{2}^2+\frac1{4\rho^*}\|\nabla G\|_{2}^2+\frac12\int_{\T^2}(\div v)^2h\,dx
\\ \leq\frac{C\rho^*}{\nu^2}\log\bigl(\e+\|\wt\rho\|_{2}\bigr)(P^*)^2\|\nabla v\|_{2}^2 +\frac1{4\nu}\|\wt P\|_2^2 +\frac2{\nu^2}\int_{\T^2}\wt G^2\,h\,dx\\
  +C\|\sqrt\rho v\|_{2}^2\|\nabla v\|_2^2\biggl(\frac{\rho^*\|\wt P\|_\infty^2}{\nu^3}+\frac{(\rho^*)^2}{\mu^2}\|\nabla \cP v\|_2^2
+\frac{(\rho^*)^2}{\nu^4}\|\wt G\|_{2}^2\biggr) 
\log\biggl(\e+\|\wt\rho\|_{2}+\frac{\|\rho\|_{2}\|\nabla v\|_{2}^2}{\|\sqrt \rho v\|_{2}^2}\biggr)\cdotp
\end{multline}

\subsubsection*{Step 3: Upgrading the energy functional $\wt\cE$}

In order to handle all the terms of the right-hand side of \eqref{eq:e8}, one has to  
add up to $\wt\cE$ a  suitable multiple of the basic energy $E$ and  of the potential energy  $e$
so as to  glean some time-decay for $\|\wt P\|_{2}.$ 
Indeed,  we have 
 $$ \d_te+\div(ev)+P\div v=0.$$
Hence, integrating on $\T^2$ and remembering that $\nu\,\div v=\wt P+\wt G$ yields
  \begin{equation}\label{eq:e13b}
 \frac d{dt}\int_{\T^2}e\,dx+\frac1\nu\int_{\T^2}|\wt P|^2\,dx=-\frac1\nu\int_{\T^2}\wt P\wt G\,dx.
 \end{equation}
Now, denoting $P^*:=\|P(\rho)\|_\infty,$  we observe that for all $\rho\geq0,$ we have 
$$\begin{aligned}
\rho\int_1^{\rho}\frac{P^2(\tau)}{\tau^2}\,d\tau\leq \rho P(\rho)\int_1^\rho  \frac{P(\tau)}{\tau^2}\,d\tau&=
P(\rho)\bigl(e(\rho) + P(1)(\rho-1)\bigr)\\&\leq P^*\bigl(e(\rho) + P(1)\rho\bigr) - P(\rho) P(1).
\end{aligned}$$
Hence, owing to \eqref{eq:normalization},  
\begin{equation}\label{eq:boundg}
\int_{\T^2} \rho(x)\biggl(\int_1^{\rho(x)}\frac{P^2(\tau)}{\tau^2}\,d\tau\biggr)dx\leq P^* (\|e\|_1+P(1)) - \bar P P(1),
\end{equation}
and thus 
\begin{equation}\label{eq:wtcE}\wt\cE\geq \frac12\int_{\T^2}\biggl(\mu|\nabla \cP v|^2+\frac1{\nu}\Bigl(\wt G^2+\wt P^2\Bigr)\biggr)\,dx
-\frac1{2\nu}\biggl(P^*\|e\|_1 + P(1) \bigl(P^*-\bar P\bigr)\!\biggr)\cdotp
\end{equation}
Consequently, if we  set
 $$\begin{aligned} \cE:&=\wt\cE +E+\|e\|_1 + \frac1{2\nu}\bigl(P^*-P(1)\bigr)P(1)\\
 &=\frac12\!\int_{\T^2}\!\biggl(\rho|v|^2\!+\!\mu|\nabla \cP v|^2\!+\!\frac1{\nu}\biggl(\wt G^2\!+\!
 \wt P^2\!+\!\biggl(\!\rho\int_\rho^1\!\frac{P^2(\tau)}{\tau^2}d\tau\!\biggr)\!+\!(P^*\!-\!P(1))P(1)\!\biggr)
 +4e\!\biggr)dx,\end{aligned}$$ then we have  thanks  to \eqref{eq:wtcE}, 
  \begin{equation}\label{eq:E}\cE\geq \frac12\biggl(\|\sqrt\rho v\|_{2}^2+\mu\|\nabla \cP v\|_{2}^2
  +\frac1{\nu}\bigl(\|\wt G\|_{2}^2+\|\wt P\|_2^2\bigr)\biggr) + \biggl(2-\frac{P^*}{2\nu}\biggr)\|e\|_{1}\biggr)\cdotp\end{equation}
  
  \subsubsection*{Step 4. A global-in-time estimate} 
  In order to control the integral  in the right-hand side of  \eqref{eq:e13b},  one may use that  
   $$
  \frac 1\nu\int_{\T^2}|\wt P|\,|\wt G|\,dx\leq\frac 1{2\nu}\int_{\T^2}\wt P^2\,dx + \frac 1{2\nu}\int_{\T^2}\wt G^2\,dx.
  $$
  Then, Poincar\'e inequality  implies that 
  $$
  \frac 1\nu\int_{\T^2}\Bigl(\frac2\nu h+1\Bigr)\wt G^2\,dx\leq \frac{4\rho^*}\nu\biggl(\frac2\nu\|h\|_\infty+1\biggr)\frac{\|\nabla G\|_2^2}{4\rho^*}\cdotp$$
  Using also the fact that  $\|\wt\rho\|_2^2=\|\rho\|_2^2-1\leq(\rho^*)^2-1,$  we get from \eqref{eq:e8} that
  $$\displaylines{
  \frac d{dt}\cE+\frac14\|\sqrt\rho\,\dot v\|_{2}^2+\frac{\mu^2}{4\rho^*}\|\nabla^2\cP v\|_{2}^2
  +\frac1{4\rho^*}\biggl(1- \frac{4\rho^*}\nu\biggl(\frac2\nu\|h\|_\infty+1\biggr)\!\biggr)\|\nabla G\|_{2}^2
+\frac{1}{4\nu}\|\wt P\|_2^2\hfill\cr\hfill+\!\int_{\T^2}\! \Bigr(\nu+\frac h2\Bigl)(\div v)^2dx+\mu\|\nabla\cP v\|_{2}^2-C\frac{\rho^*\log(\e+\rho^*)}{\nu^2}(P^*)^2\|\nabla v\|_{2}^2\hfill\cr\hfill\leq 
C\|\sqrt\rho v\|_{2}^2\|\nabla v\|_2^2\biggl(\frac{\rho^*\|\wt P\|_\infty^2}{\nu^3}+\frac{(\rho^*)^2}{\mu^2}\|\nabla \cP v\|_2^2
+\frac{(\rho^*)^2}{\nu^4}\|\wt G\|_{2}^2\biggr) 
\log\biggl(\e+ \rho^*+\frac{\rho^*\|\nabla v\|_{2}^2}{\|\sqrt \rho v\|_{2}^2}\biggr)\cdotp}
 $$ 
Now, since
\begin{equation}\label{eq:elliptic}
\mu\|\nabla v\|^2_{2} + (\lambda+\mu) \|\div v \|_{2}^2=\mu\|\nabla\cP v\|_{2}^2
+\nu\|\div v\|_{2}^2,
\end{equation}
we have  if  $\nu\geq\mu,$  
$$ \nu\|\div v\|_2^2+\mu\|\nabla\cP v\|_2^2\geq\mu\|\nabla v\|_2^2.$$
Therefore, because  for all $A\geq0,$ 
$$\log(\e+\rho^*+\rho^*A)\leq \log(\e+\rho^*)+\log(1+A)\leq\log(\e+\rho^*)+A,$$  if one assumes  that 
 \begin{equation}\label{eq:condnu1}
 1\geq 2C \frac{\rho^*\log(\e+\rho^*)(P^*)^2}{\mu\nu^2}\andf   
  \frac{8\rho^*}\nu\biggl(\frac2\nu\|h\|_\infty+1\biggr)\leq1,
  \end{equation}
then the above  inequalities imply that 
   \begin{multline}\label{eq:e9}
  \frac d{dt}\cE+\cD\leq C\rho^*\|\sqrt\rho v\|_{2}^2\|\nabla v\|_2^2\frac{\|\wt P\|_\infty^2}{\nu^3}\biggl(\log(\e+\rho^*)+\frac{\|\nabla v\|_{2}^2}{\|\sqrt \rho v\|_{2}^2}\biggr)\\+ C(\rho^*)^2\|\sqrt\rho v\|_{2}^2\|\nabla v\|_2^2\biggl(\frac{1}{\mu^2}\|\nabla \cP v\|_2^2
+\frac{1}{\nu^4}\|\wt G\|_{2}^2\biggr)\biggl(\log(\e+\rho^*)+\log\biggl(1+\frac{\|\nabla v\|_{2}^2}{\|\sqrt \rho v\|_{2}^2}\biggr)\biggr)
\end{multline} 
with 
$$\cD:= \frac14\|\sqrt\rho\,\dot v\|_{2}^2+\frac{\mu^2}{4\rho^*}\|\nabla^2\cP v\|_{2}^2+\frac1{8\rho^*}\|\nabla G\|_{2}^2
+\frac{1}{4\nu}\|\wt P\|_2^2+\frac12\int_{\T^2}(\div v)^2(\nu+h)\,dx+\frac{\mu}2\|\nabla v\|_{2}^2.$$
So, finally, if one assumes  that 
\begin{equation}\label{eq:largenu}
\nu\geq\mu,\quad \nu^2\geq {2C}\mu^{-1}\rho^*\log(\e\!+\!\rho^*) (P^*)^2,\quad \nu \geq 8\rho^*(2\nu^{-1}\|h\|_\infty+1) 
\andf \nu\geq P^*/2,
\end{equation}
the last condition ensuring that the coefficient of the last term in \eqref{eq:E} is greater than $1,$
then  \eqref{eq:equivE} holds true, and thus
\begin{equation}\label{eq:star}
\|\sqrt\rho v\|_{2}^2\leq2\cE,\qquad \|\nabla\cP v\|_2^2\leq 2\cE/\mu\andf  \|\wt P\|_2^2+\|\wt G\|_2^2\leq2\nu\cE.\end{equation}
Thanks to that, Inequality \eqref{eq:e9} combined with the energy balance \eqref{eq:energy} 
and the fact that the map $r\mapsto r\log^{1/2}(a+b/r)$ is nondecreasing on $\R_+$ if $a\geq1$ and $b\geq0,$ imply that
$$\displaylines{\quad
\frac d{dt}\cE+ \cD \leq C\frac{(\rho^*)^2}{\mu^3}E_0\|\nabla v\|_2^2\,\cE \, \log\biggl(1+\frac\cE{\mu E_0}\biggr)\hfill\cr\hfill
+C\biggl(\frac{(\rho^*)^2}{\mu^3}E_0\log(\e+\rho^*)+\frac{\rho^*\|\wt P\|_\infty^2}{\mu\nu^3}
+\log(\e+\rho^*)\frac{\rho^*\|\wt P\|_\infty^2}{\nu^3}\biggr)\|\nabla v\|_2^2\,\cE.}
$$
Note that Condition \eqref{eq:largenu} entails that  
$$\frac{\rho^*\|\wt P\|_\infty^2}{\mu\nu^3}
+\log(\e+\rho^*)\frac{\rho^*\|\wt P\|_\infty^2}{\nu^3}\leq1.$$
Therefore applying  Lemma \ref{l:osgood} with 
$$
A:=1,\quad B:=\frac{1}{\mu E_0},\quad
f:=C\frac{(\rho^*)^2}{\mu^3} E_0 \|\nabla v\|_2^2
\andf  g:=C\Bigl( 1 +\frac{(\rho^*)^2}{\mu^3}E_0\log(\e+\rho^*)\Bigr)\|\nabla v\|_2^2,
$$
we get 
$$\displaylines{
1+\frac1{\mu E_0}\biggl(\cE(t)+\int_0^t\cD(\tau)\,d\tau\biggr)\hfill\cr\hfill
\leq \biggl(1+\frac{\cE_0}{\mu E_0}\exp\biggl\{C\Bigl( 1 +\frac{(\rho^*)^2}{\mu^3}E_0\log(\e+\rho^*)\Bigr)\int_0^t\|\nabla v\|_2^2\,d\tau\biggr\}\biggl)^{\exp\bigl\{C      \frac{(\rho^*)^2}{\mu^3} E_0\int_0^t\|\nabla v\|_2^2\,d\tau\bigr\}},}
$$
which, in light of  the basic energy balance \eqref{eq:energy}, yields  \eqref{eq:H1}.
\end{proof}
 \begin{rmk} One has some freedom
 in the definition of $\cE,$   and  lots of possibilities  for bounding  the right-hand side of \eqref{eq:e8}. 
 As a consequence, for small $\nu,$   one can  get  a global, but time dependent 
 control on $\cE.$ 
 We chose not to treat that case here since the condition that $\nu$ is large 
 will be needed 
 in the next step, so as to remove the a priori assumption that $\rho$ is bounded. 
  \end{rmk}
   \begin{rmk}   Relation  \eqref{eq:wtG} and Inequality  
   %$$   (\nu\, \div v)^2\leq 2\wt G^2+2\wt P^2.$$
   \eqref{eq:equivE}  imply  that 
   $$\nu\|\div v\|_{2}^2\leq 4\cE.$$\end{rmk}
  
  %{\color{blue}In order to understand better the dependence from $\nu$ in the estimate given by   
%Proposition \ref{p:H1a} we restate (\ref{eq:H1}) in terms of $\|\rho_0\|_\infty$
%and $\|v_0\|_{H^1}$ remembering that we have the energy estimate (\ref{eq:IE}).
%\begin{col}\label{col-2}
 %Assume conditions of Proposition \ref{p:H1a}, then bound (\ref{eq:H1}) implies
 %\begin{multline}\label{ineq:2}
  %\sup_{t\leq T} \big( \|\nabla \cP v\|^2_2+\nu \|\div v\|^2_2\big) +
  %\int_0^T \int_{\T^2} \big(\rho |\dot v|^2 + |\nabla^2 \cP v|^2 + |\nabla G|^2)\,dx\,dt\\ \leq 
  %(1+\nu)C_0(\|\rho_0\|_\infty,\|v_0\|_{H^1}),
% \end{multline}
%where the RHS is independent from $T$.
%\end{col}}

 %%%%%%%%%%%%%%%%%%%%%%%

\subsection{An upper bound for the density}

 Here, we prove that, for large enough $\nu,$  if the initial data fulfill the assumptions of the previous section,  then we have a global-in-time control on the supremum of $\rho.$ 
 As in the previous subsection, we assume that we are given a smooth 
 solution with strictly positive density,   keeping in mind that the  result below 
 will be only applied to the family constructed in Section \ref{s:existence}. 
  For simplicity, we assume that  $P(\rho)=\rho^\gamma$ for 
 some $\gamma\geq1.$
  
  \begin{prop}\label{p:boundrho2}    
  Consider a  smooth solution 
 $(\rho,v)$ of  \eqref{CNS}  on  $[0,T]\times\T^2$ for some  $T<\infty$
 pertaining to  smooth  initial data $(\rho_0,v_0)$ such that $\rho_0>0$ 
 and  $\nu^{1/2}\|\div v_0\|_2\leq K.$ 
 
 There exists $\nu_0$ depending  on $K,$ $\gamma,$ $\mu,$
 $\|\rho_0\|_\infty,$  $E_0$ and $\|\nabla v_0\|_2,$ \underline{but independent of $T$} such that 
 if $\nu\geq\nu_0,$ then
  \begin{equation}\label{eq:rhomax}
     \sup_{t\in [0,T]} \|\rho(t)\|_\infty \leq 2\,\e^{\frac{\gamma-1}\gamma E_0}
  \|\rho_0\|_\infty.
  \end{equation} 
  \end{prop}
\begin{proof} 
Throughout the proof, we denote slightly abusively the right-hand side of \eqref{eq:rhomax}
by $\rho^*.$
 We start from the observation that as $\rho$ is smooth and positive (by assumption), 
 we may write, owing to \eqref{eq:wtG},  
  $$
 \d_t\log\rho+ v\cdot\nabla\log\rho=-\frac1\nu\bigl(\wt P+\wt G).
 $$
 Remember that the definition of $\wt G$ ensures that
 $$
 \Delta\wt G=\d_t(\div(\rho v))+\div(\div(\rho v\otimes v)).
 $$
 Therefore, following \cite{Des-CPDE} and introducing
 $$
 F:=\log\rho-\nu^{-1}(-\Delta)^{-1}\div(\rho v),
 $$
 we discover that (with the summation convention 
 over repeated indices), 
 \begin{equation}\label{eq:F}
 \d_tF+v\cdot\nabla F+\frac1\nu\wt P=-\frac1\nu[v^j,(-\Delta)^{-1}\d_i\d_j]\rho v^i.
 \end{equation}
 Since we have
 $$
 P(\rho)\geq \gamma\log\rho +1\quad\hbox{for all }\ \rho>0,
 $$ setting $F^+:=\max(0,F)$ yields 
   \begin{equation}\label{ineq:F}
 \d_tF^++v\cdot\nabla F^++\frac{\gamma}{\nu} F^+\leq \frac1\nu|[v^j,(-\Delta)^{-1}\d_i\d_j]\rho v^i|
 +\frac{\gamma}{\nu^2}|(-\Delta)^{-1}\div(\rho v)| + \frac{1}{ \nu}(\bar P-1).
 \end{equation}
 As $P(\rho)=\rho^\gamma,$   
   we have  $\bar P-1=(\gamma-1)\|e\|_1,$ so that
  the last term  may be bounded by $(\gamma-1)E_0.$ 
  Since the field $v$ is assumed to be smooth, applying the maximal principle for the transport equation yields:
  \begin{multline}\label{eq:F+}
 \|F^+(t)\|_\infty\leq  \e^{-\frac\gamma\nu t}\|F^+(0)\|_\infty +\frac1\nu\int_0^t\e^{-\frac{\gamma}{\nu}(t-\tau)}
 \|[v^j,(-\Delta)^{-1}\d_i\d_j]\rho v^i(\tau)\|_\infty\,d\tau\\
 +\frac{\gamma}{\nu^2}\int_0^t\e^{-\frac{\gamma}{\nu}(t-\tau)}\|(-\Delta)^{-1}\div(\rho v)(\tau)\|_\infty\,d\tau
+ \frac{\gamma-1}{ \gamma}\bigl(1-\e^{-\frac\gamma\nu t}\bigr)E_0.\end{multline}
Using that the average of $\rho v$ is zero,  Sobolev embedding and
the properties of continuity of Riesz operator imply that 
\begin{equation}\label{eq:d=2}
\|(-\Delta)^{-1}\div(\rho v)\|_\infty\lesssim \|(-\Delta)^{-1}\nabla\div(\rho v)\|_4
\lesssim\|\rho v\|_4.\end{equation}
Then, we use again \eqref{eq:interpo-des} and get
$$
\|(-\Delta)^{-1}\div(\rho v)\|_\infty\lesssim(\rho^*)^{\frac34}
 \|\sqrt\rho v\|_{2}^{\frac12}
\|\nabla v\|_{2}^{\frac12}\log^{\frac14}\biggl(\e+\rho^*
+\frac{\rho^*\|\nabla v\|_{2}^2}{\|\sqrt \rho v\|_{2}^2}\biggr),
$$
whence, thanks to the energy balance \eqref{eq:energy} and the definition of 
$\cE$ (assuming that $\nu\geq\mu$),
\begin{equation}\label{eq:F1}
\|(-\Delta)^{-1}\div(\rho v)\|_\infty\leq C(\rho^*)^{\frac34} E_0^{\frac14}
\|\nabla v\|_{2}^{\frac12}\log^{\frac14}\biggl(\e+ \rho^*+\frac{\rho^*\cE}{\mu\,E_0}\biggr)\cdotp
\end{equation}
Since $[v^j,(-\Delta)^{-1}\d_i\d_j]\rho v^i = [\wt v^j,\Delta^{-1}\d_i\d_j]\rho v^i$ with $\wt v=v-\bar v,$  
the second term in the r.h.s. of \eqref{eq:F+} may be bounded
by means of  Sobolev embedding and of the following Coifman, Lions, Meyer and Semmes
inequality (from \cite{CLMS}) as follows:
\begin{equation}\label{eq:CLMS}\|[v^j,(-\Delta)^{-1}\d_i\d_j]\rho v^i\|_\infty\lesssim
\|[\wt v^j,(-\Delta)^{-1}\d_i\d_j]\rho v^i\|_{W^{1,3}}
\lesssim\|\nabla v\|_{12}\|\rho v\|_4.
\end{equation}
To handle $\nabla v,$ we  use that
$$\begin{aligned}\|\nabla v\|_{12}&\lesssim \|\nabla\cP v\|_{12}+\nu^{-1}\bigl(\|\wt G\|_{12}+\|\wt P\|_{12}\bigr)\\
&\lesssim \|\nabla^2\cP v\|_2+\nu^{-1}\bigl(\|\nabla G\|_2+\|\wt P\|_\infty\bigr)\cdotp
 \end{aligned}$$
Hence, using once more \eqref{eq:interpo-des},
$$\displaylines{\quad
\|[v^j,(-\Delta)^{-1}\d_i\d_j]\rho v^i\|_\infty\lesssim (\rho^*)^{\frac34}
\bigl(\|\nabla^2\cP v\|_2+\nu^{-1} \|\nabla G\|_2
+\nu^{-1} \|\wt P\|_\infty\bigr) 
\hfill\cr\hfill\times \|\sqrt\rho v\|_{2}^{\frac12}
\|\nabla v\|_{2}^{\frac12}\log^{\frac14}\biggl(\e+\rho^*
+\frac{\rho^*\|\nabla v\|_{2}^2}{\|\sqrt \rho v\|_{2}^2}\biggr),\quad}
$$
whence, using the energy conservation \eqref{eq:energy} and the definition of 
$\cE$ and $\cD,$
\begin{multline}\label{eq:F2}
\|[v^j,(-\Delta)^{-1}\d_i\d_j]\rho v^i\|_\infty\lesssim \bigl((\rho^*)^{\frac54}
 \mu^{-1}\cD^{\frac12}+ (\rho^*)^{\frac34} \nu^{-1} \|\wt P\|_\infty\bigr)
 \\ \times  E_0^{\frac14}\|\nabla v\|_{2}^{\frac12}\log^{\frac14}\biggl(\e+ \rho^*
+\frac{\rho^*\cE}{\mu\,E_0}\biggr)\cdotp
\end{multline}
Plugging \eqref{eq:F1} and \eqref{eq:F2} in \eqref{eq:F+} and performing obvious
simplifications, we end up with 
\begin{multline}\label{eq:F3}
\|F^+(t)\|_\infty\leq\|F^+(0)\|_\infty  + \frac{\gamma-1}\gamma E_0\\
+C(\rho^*)^{\frac34}\frac{E_0^{\frac14}}\nu
\int_0^t\e^{-\frac\gamma\nu(t-\tau)}\biggl( \bigl(\sqrt{\rho^*}\mu^{-1}\cD^{\frac12}+\nu^{-1}(\|\wt P\|_\infty+\gamma) \bigr)\|\nabla v\|_{2}^{\frac12}\log^{\frac14}
\biggl(\e+\rho^*+\frac{\rho^*\cE}{\mu\,E_0}\biggr)\biggr)d\tau.
\end{multline}

%{\color{blue}
 Let us consider the largest sub-interval $[0,T_0]$ of $[0,T]$ on which 
\eqref{eq:rhomax} is fulfilled. 
Then, Inequality \eqref{eq:H1} tells us that there exist $\nu_0$  
 depending only on   $\|\rho_0\|_\infty,$ $\mu$ and $\gamma,$ 
and $C_0>0$ (depending also on $K,$ $E_0,$ $\|\rho_0\|_\infty$, $\|\nabla v_0\|_{2}$) so that 
we have for all $t\in[0,T_0],$ if $\nu\geq\nu_0,$
\begin{equation}\label{eq:DE}
\cE(t)+\int_0^t\cD(\tau)\,d\tau\leq C_0.
\end{equation}
Inequality  \eqref{eq:F3} thus becomes for a possibly larger $C_0,$ 
%(taking a larger $C_0$ as the case may be):
\begin{equation}\label{eq:F4}
\|F^+(t)\|_\infty\leq\|F^+(0)\|_\infty +  
\frac{\gamma-1}\gamma E_0 +\frac{C_0}{\nu}
\int_0^t\e^{-\frac{\gamma}{\nu}(t-\tau)}
\biggl(\frac{\cD^{\frac12}(\tau)}\mu+\frac1\nu\biggr)\|\nabla v(\tau)\|_{2}^{\frac12}\,d\tau.
 \end{equation}
 From H\"older  inequality, we have for all $t\in[0,T],$ 
$$\begin{aligned}
\int_0^t\e^{-\frac{\gamma}{\nu}(t-\tau)}
\cD^{\frac12}\|\nabla v\|_{2}^{\frac12}\,d\tau
&\leq C\biggl(\frac\nu\gamma\biggr)^{\frac14} \biggl(\int_0^T\cD(\tau)\,d\tau\biggr)^{\frac12}
\biggl(\int_0^T\|\nabla v(\tau)\|_{L^2}^2\,d\tau\biggr)^{\frac14}\\
\andf
\int_0^t\e^{-\frac{\gamma}{\nu}(t-\tau)}
\|\nabla v(\tau)\|_{2}^{\frac12}\,d\tau&\leq C\biggl(\frac\nu\gamma\biggr)^{\frac34}\biggl(\int_0^T\|\nabla v(\tau)\|_{L^2}^2\,d\tau\biggr)^{\frac14}.
\end{aligned}$$
  As the  integrals in the right-hand side may be bounded in terms of the data
  according to the basic energy balance  \eqref{eq:energy} and to \eqref{eq:DE}, 
 we eventually get (changing once again $C_0$ if needed)  if $\nu\geq\nu_0$:
 $$ \|F^+(t)\|_\infty\leq\|F^+(0)\|_\infty + C_0\nu^{-\frac14}+ \frac{\gamma-1}\gamma E_0.$$
%{\bf R. There is some freedom in \eqref{eq:CLMS}. Exponent $3/4$ is not optimal. 
%One can get almost $1$ but maybe not $1.$ Who cares !! we have it :) 
%in other words, we are not able to improve it to get something better...}
Of course, owing to the definition of $F^+$ and to \eqref{eq:H1} and \eqref{eq:F1}, we have
$$
\log\rho\leq F^++\nu^{-1}\|\Delta^{-1}\div(\rho v)\|_\infty\leq F^+
+\nu^{-1}  C_0.
$$
Hence one can eventually conclude that
\begin{equation}\label{eq:boundrho1}
\log\rho^*\leq \log\rho_0^* +C_0\nu^{-\frac14} + \frac{\gamma-1}\gamma E_0.
\end{equation}
Now, if $\nu$ is so large as to satisfy also 
$$
 C_0\nu^{-\frac14}  <\log2,
$$%}
then \eqref{eq:boundrho1} implies (\ref{eq:rhomax}) with a strict  inequality. 
So  $T_0=T.$%, and since $T$ is arbitrary, we obtain \eqref{eq:rhomax} on $[0,\infty).$
\end{proof}

%%%%%%%%%%%%%%%%%%%%%%

\subsection{Weighted estimates}

That section is devoted to the proof of the following result, that is based on the 
estimates that have been established  so far. 
For better readability, we  postpone the most technical parts  to the appendix. 
\begin{prop}\label{p:weight} Define $\nu_0$ as in Proposition \ref{p:boundrho2}.
 Let $T\geq 1$. Then, smooth solutions to \eqref{CNS}  on $[0,T]\times\T^2$ 
  fulfill, if $\nu\geq\nu_0$:
 \begin{equation}\label{eq:weight} \sup_{t\in[0,T]} \int_{\T^2} \rho |\dot v|^2 t\,dx +
  \int_{0}^{T} \!\!\!\int_{\T^2} (\mu|\nabla\cP \dot v|^2 + \nu | \div \dot v|^2) t \,dx \,dt 
  \leq C_0T\exp\Bigl(\frac{C_0T}\nu\Bigr),\end{equation}
   where  $C_0$ depends  on $\rho^*,$ $\mu,$ $\cE_0$ and on the pressure function, \emph{but is independent of $\nu$ and~$T$.}   \end{prop}
\begin{proof} Here it will be convenient to  use the two notations $\dot f$ 
and $\frac D{Dt}f$ to designate the convective derivative of $f,$
and  we shall denote $A:B= \sum_{i,j} A_{ij}B_{ij}$
if $A$ and $B$ are two $d\times d$ matrices. Finally, if $v$ is a vector field on $\T^d$
then $(Dv)_{ij}:=\d_jv^i$ and $(\nabla v)_{ij}:=\d_iv^j$ for $1\leq i,j\leq d.$ 

\smallbreak
The general principle is  to rewrite the momentum equation as:
\begin{equation}\label{mom}
 \rho \dot v - \mu\Delta v -  (\nu-\mu)\nabla \div v + \nabla P =0,
\end{equation}
then to take the material derivative  and test it  by $t\,\dot v.$ We get
\begin{equation}\label{mom-d}
\int_{\T^d}\biggl(\frac D{Dt}(\rho\dot v)-\mu\frac D{Dt}\Delta v-(\nu-\mu)\frac D{Dt}\nabla\div v
+\frac D{Dt}\nabla P\biggr)\cdot (t\,\dot v)\,dx=0.
\end{equation}
The rest of the proof consists in describing    each term of \eqref{mom-d}. 
To this end, we shall repeatedly use the fact that
for all $\nu\geq\nu_0$ (where $\nu_0$ is given by Proposition \ref{p:boundrho2}), we have
\begin{equation}\label{w22}
 \|\nabla v\|_{L_4(0,T\times \T^2)} \leq C_0\cdotp
\end{equation}
Indeed, recall the decomposition
\begin{equation}\label{eq:decompo}
 v = \cP v - \frac{1}{\nu} \nabla (-\Delta)^{-1} (\wt G + \wt P).
\end{equation}
Proposition \ref{p:H1a} and Sobolev  embeddings  imply that 
 \begin{equation}\label{w22aa}
\|\nabla \cP v\|_{L_4(0,T \times \T^2)}\lesssim
\|\nabla\cP v\|_{L_\infty(0,T;L_2)}^{1/2}\|\nabla^2\cP v\|_{L_\infty(0,T;L_2)}^{1/2}\leq C_0.\end{equation}
 Furthermore,  we have
 \begin{equation}\label{w22a}
 \| \wt G\|_{L_4(0,T\times \T^2)} \lesssim\|\wt G\|_{L_\infty(0,T;L_2)}^{1/2}\|\nabla G\|_{L_2(0,T;L_2)}^{1/2}\leq\nu^{1/4}C_0,\end{equation}
 and 
  \begin{equation}\label{w22b}
  \| \wt P\|_{L_4(0,T\times \T^2)}\leq\|\wt P\|_{L_2(0,T;L_2)}^{1/2} \|\wt P\|_{L_\infty(0,T;L_\infty)}^{1/2}\leq\nu^{1/4} C_0.\end{equation}

\subsubsection*{Step 1}  
 Obvious computations give  (in any dimension):
\begin{equation}\label{w1a}
 \int_{\T^d} \frac D{Dt}(\rho\dot v) \cdot\dot v \, t \,dx = \frac 12 \int_{\T^d} \biggl(\frac{D}{Dt} ( \rho |\dot v|^2 t) +\dot \rho |\dot v|^2 t -  \rho |\dot v|^2\biggr) dx.
\end{equation}
Integrating by parts, we see that
\begin{equation}\label{w11}
 \int_{\T^d} \frac{D}{Dt} ( \rho|\dot v|^2 t)\,dx = \frac{d}{dt} \int_{\T^d} \rho|\dot v|^2 t\, dx -
 \int_{\T^d} \div v (\rho|\dot v|^2 t)\, dx.
\end{equation}
Thanks to the mass conservation equation, we have \begin{equation}
\int_{\T^d} \dot \rho|\dot v|^2 t\, dx = -\int_{\T^d} \rho \,\div v |\dot v|^2 t\, dx,
\end{equation}
whence 
$$\ \int_{\T^d} \frac D{Dt}(\rho\dot v) \cdot\dot v \, t \,dx  = 
\frac 12 \frac{d}{dt} \int_{\T^d} \rho |\dot v|^2 t dx
- \frac12 \int_{\T^d} \rho|\dot v|^2 \,dx -  \int_{\T^d} \rho \,\div v |\dot v|^2 t\, dx.$$
If $d=2$ then one can bound the last term using that
 $$\begin{aligned}
 \int_{\T^2}\rho\,\div v|\dot v|^2t\,dx&= \nu^{-1}\int_{\T^2} (\wt P+\wt G)\rho|\dot v|^2t\,dx\\
 &\leq \nu^{-1}\|\wt P\|_\infty\|\sqrt{\rho t}\,\dot v\|_2^2 + \rho^*\nu^{-1}\|\wt G\|_2\|\sqrt t\,\dot v\|_4^2\\
 &\leq C_0\nu^{-1}\|\sqrt{\rho t}\,\dot v\|_2^2+ C_0\nu^{-1/2}\|\sqrt t\,\dot v\|_4^2.\end{aligned}
 $$
Since $\int_{\T^2}\rho\dot v\,dx=0,$ one 
can take advantage of the Poincar\'e inequality \eqref{eq:poincarep} with $p=2$ and get:
\begin{equation}\label{eq:dotv4}
\|\sqrt{t}\,\dot v\|_{4}^2\leq C \|\sqrt{t}\,\dot v\|_{2}\|\sqrt{t}\,\nabla\dot v\|_{2}
\leq C(1+\|\wt\rho\|_2)\|\sqrt{t}\,\nabla\dot v\|_{2}^2\leq C\rho^*\|\sqrt{t}\,\nabla\dot v\|_{2}^2.
\end{equation}
%Hence we have
%\begin{equation}\label{eq:dotv4b}
%\|\sqrt t\,\dot v\|_{L_2(0,T;L_4)}\leq C\sqrt{\rho^*} \|\sqrt t\,\nabla\dot v\|_{L_2(0,T;L_2)}.\end{equation}
Hence, 
$$ \int_{\T^2}\rho\,\div v|\dot v|^2t\,dx\leq C_0\biggl(\nu^{-1}\|\sqrt{\rho t}\,\dot v\|_2^2
+\nu^{-1/2}\|\sqrt t\,\nabla\dot v\|_2^2\biggr),
$$
and thus
\begin{multline}\label{eq:step1}
 \int_{\T^d} \frac D{Dt}(\rho\dot v) \cdot\dot v \, t \,dx  \geq 
\frac 12 \frac{d}{dt} \int_{\T^2} \rho |\dot v|^2 t dx  - \frac12 \int_{\T^2} \rho|\dot v|^2 \,dx
\\-C_0\biggl(\nu^{-1}\int_{\T^2} \rho t\,|\dot v|^2\,dx+\nu^{-1/2}\int_{\T^2} t|\nabla\dot v|^2\,dx\biggr)\cdotp
\end{multline}
\subsubsection*{Step 2} To handle the second term of \eqref{mom-d}, we use that 
\begin{equation}\label{w2}
 -\frac{D}{Dt} \Delta v = - \div \frac{D}{Dt} \nabla v -\nabla v\cdot\nabla^2v\with
 (\nabla v\cdot\nabla^2v)^i:= \sum_{j,k} \d_k v^j\,\d_j\d_kv^i.
\end{equation}
%As we shall not use the special structure of the last term, 
%The same convention is applied
%to all the products of first and second derivatives of in the computations below. 
Hence, testing  \eqref{w2}  by $t\dot v$ and integrating by parts yields for $d=2,3,$
$$-\int_{\T^d}\frac D{Dt} \Delta v\cdot t\dot v\,dx= \int_{\T^d} \frac{D}{Dt} \nabla v : \nabla \dot v \, t \,dx - \int_{\T^d} (\nabla v \cdot \nabla^2 v) \cdot \dot v \, t \, dx.
$$
Since
$$
 \frac{D}{Dt} \nabla v = \nabla \dot v -  \nabla v\cdot \nabla v,
$$
we get  
\begin{equation}\label{w21}
-\int_{\T^d}\frac D{Dt} \Delta v\cdot t\dot v\,dx=  \int_{\T^d} |\nabla \dot v|^2t \,dx - \int_{\T^d} (\nabla v \cdot \nabla v): \nabla \dot v \, t \, dx
 - \int_{\T^d} (\nabla v\cdot \nabla^2 v) \cdot \dot v t \,dx.
\end{equation}
The first term is the main one. The other two terms 
are denoted by $I_1$ and $I_2,$ respectively.  
Bounding $I_1$ is easy : using H\"older inequality   yields
 $$
|I_1|=\biggl| \int_{\T^d} (\nabla v \cdot \nabla v):\nabla \dot v \, t \,dx\biggr|\leq \|t^{1/4}\nabla v\|_{4}^2 \|\sqrt t\,\nabla \dot v\|_2.$$
Therefore, we have according to \eqref{w22}, 
 \begin{equation}\label{WI1}
\biggl| \int_0^T I_1(t)\,dt\biggr|
 \leq C_0\sqrt T\,  \|\sqrt t\,\nabla \dot v \|_{L_2(0,T\times\T^2)}.\end{equation}
Bounding $I_2$ is much more involved. We eventually get (see the details in appendix):
\begin{multline}\label{eq:I2}
\biggl|\int_0^T I_2\,dt\biggr|\leq 
 \biggl( T^{1/4} \|\sqrt{\rho t}\,  \dot v \|_{L_\infty(0,T;L_2)}^{1/2}\|\sqrt t\,\dot v\|_{L_2(0,T;L_4)}\\+
\sqrt T\bigl(\nu^{-3/4}\|\sqrt t\,\nabla\dot v\|_{L_2(0,T;L_2)}
+\nu^{-2}\|\sqrt t\,\dot v\|_{L_2(0,T;L_4)}\bigr)\biggr)\cdotp
\end{multline}
Plugging \eqref{WI1} and \eqref{eq:I2} in \eqref{w21}
and using \eqref{eq:dotv4} yields
\begin{multline}\label{eq:step2-2}
-\mu\int_0^T\!\!\! \int_{\T^2} \biggl(\frac{D}{Dt} \Delta v\biggr) \cdot \dot v \, t \, dx \,dt \geq 
 \mu\int_0^T \!\!\!\int_{\T^2} |\nabla \dot v|^2t \, dx \,dt\\
 -C_0T^{1/4}\bigl(T^{1/4} +  \|\sqrt{\rho t}\,\dot v\|^{1/2}_{L_\infty(0,T;L_2)}\bigr)  \|\sqrt t\,\nabla\dot v\|_{L_2(0,T\times\T^2)}.
 \end{multline}

\subsubsection*{Step 3}

In order to bound  the third term of equation \eqref{mom-d}, we use the relation
\begin{equation}\label{w3}
 -\frac{D}{Dt}\nabla \div v = - \nabla \frac{D}{Dt} \div v + \nabla v\cdot \nabla \div v.
\end{equation}
We have to keep in mind that the right-hand side  involves only the potential part 
 $\cQ v$ of the velocity, since $\div v = \div \cQ v$. This enables us to write that
 $$
 \nabla\frac D{Dt}\div v=\nabla\div\dot v-\nabla({\rm tr}(\nabla v\cdot\nabla\cQ v)).
 $$
 Hence, testing \eqref{w3} with $\dot v\,t$ and integrating by parts, we find that
\begin{multline}\label{w31}
-\int_{\T^d}\!\frac{D}{Dt}\nabla \div v\cdot \dot v\,t\,dx=
\int_{\T^d} (\div\dot v)^2 \,t \, dx - K_1+K_2\\
\with K_1:=
 \int_{\T^d} {\rm Tr}(\nabla v\cdot \nabla \cQ v) \, \div \dot v \, t \, dx\andf
K_2:=\int_{\T^d} (\nabla v \cdot\nabla \div v)\cdot  \dot v \, t \, dx.
\end{multline}
Since  $\nu \cQ v = -\nabla (-\Delta)^{-1} (\wt G+ \wt P),$
using \eqref{w22}, \eqref{w22a} and \eqref{w22b},  we find that, if $d=2,$ 
\begin{eqnarray}\label{eq:K1}
\nu\left|\int_0^T K_1  \,dt\right| &\!\!\!\leq\!\!\!& C\sqrt T\, \|\nabla v\|_{L_4(0,T;L_4)}\|\wt P 
+ \wt G\|_{L_4(0,T;L_4)} \|\sqrt t\,\div \dot v\|_{L_2(0,T;L_2)}\nonumber \\
 &\!\!\!\leq\!\!\!& C_0\nu^{1/4}\sqrt T\, \|\sqrt t\,\div \dot v \|_{L_2(0,T;L_2)}.
\end{eqnarray}
Bounding $K_2$ will be performed in the appendix. In the end, we get
\begin{equation}\label{eq:K2}
\nu \biggl|\int_0^T K_2 \,dt\biggr| \leq C_0T^{1/4}\Bigl(\|\sqrt{\rho t}\, \dot v\|_{L_\infty(0,T;L_2)}^{1/2} 
 \|\sqrt t\,\dot v\|_{L_2(0,T;L_4)}
\!+\!T^{1/4}\|\sqrt t\,\nabla\dot v\|_{L_2(0,T;L_2)}\Bigr)\cdotp
\end{equation}
Thanks to \eqref{eq:dotv4}, the conclusion of this step is that if $\nu$ is large enough then
\begin{multline}\label{eq:step3}
-(\nu\!-\!\mu)\int_0^T \!\!\!\int_{\T^2} \frac{D}{Dt} \nabla \div v\cdot\dot v\,t\, dx\, dt \geq  (\nu\!-\!\mu)\! \int_0^T\!\!\! \int_{\T^2} (\div\dot v)^2 t \,dx\,dt \\
- C_0T^{1/4}\Bigl((\nu T)^{1/4}\,\|\sqrt t\, \div \dot v \|_{L_2(0,T;L_2)}+(\|\sqrt{\rho t}\, \dot v\|_{L_\infty(0,T;L_2)}^{1/2}+T^{1/4})\|\sqrt t\,\nabla\dot v\|_{L_2(0,T;L_2)}\Bigr)\cdotp
\end{multline}

\subsubsection*{Step 4}

The last term under  consideration in \eqref{mom} is
\begin{equation}\label{w4}
 \frac{D}{Dt} \nabla P = \nabla \frac{D}{Dt} P - \nabla v\cdot \nabla P.
\end{equation}
Here the analysis is simple: since $\dot P=-h \,\div v$,  we  have
$$ \begin{aligned}\int_{\T^d}\frac{D}{Dt} \nabla P\cdot \dot v\,t\,dx= L_1+L_2 \with 
L_1&:=\int_{\T^d} h \,\div v \, \div \dot v \, t \, dx \\ \andf
L_2&:=- \int_{\T^d} \d_iv^j\,\d_jP\,\dot v^i
%\biggl(\nabla \Bigl(\cP v - \frac{1}{\nu} \nabla(-\Delta)^{-1}(\wt G+\wt P)\Bigr) \cdot \nabla \wt P \biggr)\cdot  \dot v
 \, t \, dx.\end{aligned}
$$
On the one hand, we obviously have
\begin{equation}\label{w36}
| L_1 | \leq \frac\nu4 \int_{\T^d} (\div\dot v)^2 \, t \,  dx + T \nu^{-1}\|h \|_\infty^2 \int_{\T^d} (\div v)^2 \, dx
\quad\hbox{ for all }\ t\in[0,T]. 
\end{equation}
On the other hand, integrating by parts a couple of times and using $\div v=\nu^{-1}(\wt P+\wt G)$ yields 
$$\begin{aligned}
L_2&=\int_{\T^d} \wt P\,\nabla\div v\cdot\dot v\,t\,dx+\int_{\T^d}\wt P\, \nabla v:D\dot v\,t\,dx\\
&=\frac1\nu \int_{\T^d} \wt P\,\nabla P\cdot\dot v\,t\,dx+\frac1\nu \int_{\T^d} \wt P\, \nabla G\cdot\dot v\,t\,dx
+\int_{\T^d}\wt P\, \nabla v:D\dot v\,t\,dx\\
&=-\frac1{2\nu} \int_{\T^d} \wt P^2\,\div\dot v\,t\,dx+\frac1\nu \int_{\T^d} \wt P\, \nabla G\cdot\dot v\,t\,dx
+\int_{\T^d}\wt P\, \nabla v:D\dot v\,t\,dx.\end{aligned}$$
Hence we have, if $d=2,$ 
$$\displaylines{
\biggl|\int_0^TL_2(t)\,dt\biggr|\leq \frac{1}{2\nu}\|\wt P\|_{L_4(0,T;L_4)}^2\|t\,\div\dot v\|_{L_2(0,T;L_2)}
\hfill\cr\hfill+\frac{1}\nu \|\wt P\|_{L_\infty(0,T;L_4)}\|\nabla G\|_{L_2(0,T;L_2)}\| t\,\dot v\|_{L_2(0,T;L_4)}\hfill\cr\hfill
+\|\wt P\|_{L_\infty(0,T;L_\infty)}\|\nabla v\|_{L_2(0,T;L_2)}\|t\,\nabla\dot v\|_{L_2(0,T;L_2)},}
 $$
whence, thanks to  \eqref{w22b} and \eqref{eq:dotv4}, 
$$\biggl|\int_0^TL_2(t)\,dt\biggr|\leq C_0\sqrt T\,
\|\sqrt t\,\nabla\dot v\|_{L_2(0,T;L_2)}.
$$
So this step gives
\begin{multline}\label{eq:step4}
\int_{\T^2} \frac{D}{Dt} \nabla P\cdot \dot v\,t\,dx \geq - \frac\nu4 \int_0^T \!\!\!\int_{\T^2} (\div\dot v)^2 \, t \,dx\,dt \\
- \|h\|_\infty\nu^{-1}T\|\div v\|_{L_2(0,T;L_2)}^2-C_0\sqrt T\,\|\sqrt t\,\nabla\dot v\|_{L_2(0,T;L_2)}.\end{multline}

\subsubsection*{Step 5}

Plugging inequalities \eqref{eq:step1}, \eqref{eq:step2-2}, \eqref{eq:step3} and \eqref{eq:step4} 
in \eqref{mom-d} (after integrating on $[0,T]$) and using 
the fact that for a smooth solution, we have $\sqrt{\rho t} \, \dot v |_{t=0}=0,$
 we discover that for large enough $\nu,$ 
 $$\displaylines{
 \|\sqrt{\rho t} \, \dot v\|_{L_\infty(0,T;L_2)}^2 + 2\mu
 \int_0^T \|\sqrt t\, \nabla\cP \dot v\|_2^2 \,dt +\frac{3\nu}2\int_0^T \|\sqrt t\,\div \dot v\|_2^2\,dt
\hfill\cr\hfill \leq   C_0\biggl(\!\nu^{-1}\!\!\int_0^T\|\sqrt{\rho t}\dot v\|_2^2+\nu^{-1/2}\int_0^T\|\sqrt t\nabla\dot v\|_2^2\,dt\biggr)
 +\|\sqrt\rho\,\dot v\|_{L_2(0,T;L_2)}^2 + {\|h\|_\infty T}\nu^{-1}\|\div v\|_{L_2(0,T;L_2)}^2  \hfill\cr\hfill
  +C_0T^{1/4}\Bigl((\nu T)^{1/4}\|\sqrt{t}\,\div\dot v\|_{L_2(0,T;L_2)}
+\bigl(T^{1/4} +  \|\sqrt{\rho t}\,\dot v\|^{1/2}_{L_\infty(0,T;L_2)}\bigr)\Bigr)  \|\sqrt t\,\nabla\dot v\|_{L_2(0,T\times\T^2)}\cdotp
}$$
Taking advantage of inequality  \eqref{eq:H1}, we have
$$
\|\sqrt\rho\,\dot v\|_{L_2(0,T;L_2)}^2+\nu\|\div v\|_{L_2(0,T;L_2)}^2\leq C_0.
$$
Furthermore, Young inequality implies that 
$$\begin{aligned}
C_0\sqrt T\|\sqrt t\,\nabla\dot v\|_{L_2(0,T;L_2)}&\leq \frac\mu2\|\sqrt t\,\nabla\dot v\|_{L_2(0,T;L_2)}^2
+C_0T\\
C_0\nu^{1/4}\sqrt T\|\sqrt t\,\div\dot v\|_{L_2(0,T;L_2)}&\leq \frac\nu2\|\sqrt t\,\div\dot v\|_{L_2(0,T;L_2)}^2
+C_0T\nu^{-1/2}\andf\\
C_0T^{1/4} \|\sqrt{\rho t}\,\dot v\|_{L_\infty(0,T;L_2)}^{1/2}\|\sqrt t\,\nabla\dot v\|_{L_2(0,T;L_2)}
&\leq C_0T \!+\!\frac12 \|\sqrt{\rho t}\,\dot v\|_{L_\infty(0,T;L_2)}^2\!+\!\frac\mu2\|\sqrt t\,\nabla\dot v\|_{L_2(0,T;L_2)}^2.\end{aligned}
$$
In the end, we thus have  if $\nu$ is large enough and $T\geq1,$
%$$\displaylines{ \|\sqrt{\rho t} \, \dot v\|_{L_\infty(0,T;L_2)}^2 + \mu
 %\int_0^T \|\sqrt t\, \nabla\cP \dot v\|_2^2 \,dt +\nu\int_0^T \|\sqrt t\,\div \dot v\|_2^2\,dt
 %\leq       C_0\nu^{-1}\int_0^T\|\sqrt{\rho t}\dot v\|_2^2  \hfill\cr\hfill
% +C_0T^{1/2}\|\sqrt t\,\nabla\dot v\|_{L_2(0,T;L_2)}
 %+C_0\bigl(T^{1/2}\nu^{-1} +T^{1/4} \|\sqrt{\rho t}\,\dot v\|_{L_\infty(0,T;L_2)}^{1/2}\bigr)\|\sqrt t\,\dot v\|_{L_2(0,T;L_4)}+ C_0T.}$$
%At this point, one may use \eqref{eq:dotv4} to bound  $\sqrt t\,\dot v$ in $L_2(0,T;L_4).$ 
% Using repeatedly Young inequality, we arrive at 
$$\displaylines{
X^2(t)+\frac12\int_0^tY^2(\tau)\,d\tau\leq C_0T+C_0\nu^{-1}\int_0^TX^2\,d\tau\with\cr 
X(t):=\|\sqrt{\rho s}\,\dot v\|_{L_\infty(0,t;L_2)} \andf Y(t):= \bigl(\mu\|\sqrt t\, \nabla\cP \dot v\|_{L_2(0,T;L_2)}^2  
+ \nu \|\sqrt t\,\div \dot v\|_{L_2(0,T;L_2)}^2\bigr)^{\frac12}.}$$
Then, applying Gronwall inequality completes the proof of the proposition.
\end{proof}
\medbreak 
The following consequence  of those weighted estimate  will be fundamental  
in  the proof of uniqueness. 
\begin{col}\label{c:c}  Let $(\rho,v)$ be a smooth solution of \eqref{CNS} on $[0,T]\times\T^2$
and assume that $\nu\geq\nu_0.$ Then,  we have  for all $\varepsilon\in(0,1),$
\begin{equation}\label{c1x}
 \|\div v\|_{L_{2-\varepsilon}(0,T;L_\infty(\T^2))} \leq C_{0,T}\, \nu^{-1+\varepsilon}\andf 
   \|\nabla \cP v\|_{L_{2-\varepsilon}(0,T;L_\infty(\T^2))} \leq C_{0,T}, \end{equation}
 for some  $C_{0,T}>0$ depending on $\varepsilon,$  $T,$ $\mu$ and on $\cE_0,$ 
 but not on $\nu.$
 \end{col}
\begin{proof} Remember  that
\begin{equation}\label{w20}
 \mu(\nabla\div v-\Delta v) + \nabla G = \rho \dot v.
\end{equation}
Hence, we have 
$$\nabla(\sqrt t\, G)=\cQ(\sqrt t\,\rho\dot v)\andf -\mu\Delta(\sqrt t\,\cP v)=\cP(\sqrt t\,\rho\dot v).$$
Then, combining the fact that $\cP$ and $\cQ$ map $L_p(\T^2)$ to itself for all $1<p<\infty$ with routine 
interpolation inequalities yields  for all $t\in[0,T]$ and $p\in[2,\infty),$
$$\|\nabla(\sqrt t\, G)\|_p+\|\nabla^2(\sqrt t\,\cP v)\|_p\leq C_{\rho,T} \|\sqrt{\rho t}\,\dot v\|_2^{\frac2p}
\|\nabla(\sqrt t\,\dot v)\|_2^{1-\frac2p},$$
which already gives, after time integration and use of Inequality \eqref{eq:weight}, that 
$$  \int_0^T\Bigl( \|\nabla(\sqrt t\,G)\|^{\frac{2p}{p-2}}_p+
    \|\nabla^2(\sqrt t\,\cP v)\|^{\frac{2p}{p-2}}_p\Bigr) dt
  \leq C_{0,T}\quad\hbox{if }\ \nu\geq\nu_0.  $$
Next, in order to bound $\wt G,$ one can argue again by interpolation
writing that 
$$\|\wt G\|_\infty\lesssim\bigl(\nu^{-\frac12}\|\wt G\|_2\bigr)^\theta\bigl(\sqrt t\|\nabla G\|_p\bigr)^{1-\theta}
t^{\frac{\theta-1}2}\nu^{\frac\theta2}\with \theta:=\frac{p-2}{2p-2}\cdotp$$
Since the previous computations ensure that 
$$t\mapsto \bigl(\sqrt t\|\nabla G\|_p\bigr)^{1-\theta}\ \hbox{ is in }\ L_{\frac{4p-4}{p-2},loc}(\R_+)$$
and, because $t\mapsto t^{\frac{\theta-1}2}$ is in $L_{q,loc}(\R_+)$ for all $q<\frac{4p-4}p$
and, finally, $t\mapsto \|\wt G(t)\|_2$ is bounded according to Proposition \ref{p:H1a}, 
one can conclude that for all $\varepsilon\in(0,1),$ 
$$
\int_0^t\|\wt G\|_\infty^{2-\varepsilon}\,dt\leq C_{0,T}\nu^{\varepsilon}.
$$
Bounding $\nabla\cP v$ in $L_\infty(\T^2)$ follows from  similar arguments. 
Finally, as $\div v=\nu^{-1}(\wt G+\wt P)$ and $P$ is bounded, one gets the desired inequality for $\div v.$
\end{proof}
%\begin{rmk}\label{r:div} The proof of the above corollary 
%actually supplies a control on $\div v$ and $\nabla\cP v$ in 
%$L_r(0,T;L_\infty)$ for some  $r>1.$ This will be useful for uniqueness. 
%\end{rmk}

%%%%%%%%%%%%%%%%%%%%%%%%%%%%%%%%%%%%%%%%%%%%%%%%%%%%%%%

\section{The proof of   existence in dimensions $2$ and $3$}\label{s:existence}

This section is mainly devoted to the construction of solutions fulfilling Theorem \ref{thm:global2}
(or the corresponding statement in dimension $3,$ see the appendix). 
   The main  two difficulties we have to face is that the initial density has no regularity whatsoever
and is not positive. To fit in the classical literature devoted to the compressible Navier-Stokes equations, 
one   has to mollify the  initial data  and  to make the density 
strictly positive. Although this procedure does not disturb  the  a priori estimates we proved 
hitherto, the state-of-the-art on the topics just  ensures  the existence of a smooth solution corresponding to the 
regularized data on    some   \emph{finite} time interval. 
As a first, we thus have to justify that, indeed, the estimates we proved so far
ensure that   the approximate smooth solution is global, if  $\nu$ is large enough.
Then, resorting to  rather classical compactness 
arguments will enable us to conclude 
the  proof of Theorem \ref{thm:global2}.
\smallbreak
At the end of the section, we justify the convergence from \eqref{CNS} to \eqref{INS},
namely we prove Theorem \ref{thm:asymptotic}.
The passing to the limit therein is very similar to Step 4 of the proof of existence.

\subsubsection*{Step 1.} The original initial data are:
\begin{equation}\label{e1}
 \rho_0 \in L_\infty(\T^d) \mbox{ and } v_0 \in H^1(\T^d).
\end{equation}
First, we want to change the initial density in such a way that it is bounded away from zero and
still has total mass equal to one. To this end, we introduce  for any $\delta\in(0,1),$ 
\begin{equation}\label{e2}
 \tilde\rho^\delta_0=\max\{ \rho_0,\delta\}
% \end{equation}
\mbox{ \ \ and  then \ \ }
% \begin{equation}
 \check\rho_0^\delta= \min\{ \xi_\delta, \tilde \rho_0^\delta\},
\end{equation}
where $\xi_\delta\geq1$ is fixed so that 
\begin{equation}
 \int_{\T^d} \check\rho^\delta_0\,dx =1.
\end{equation}
Clearly,  we have $\xi_\delta\to\rho_0^*:=\|\rho_0\|_\infty$ when $\delta\to0,$  and thus 
\begin{equation}
 \delta \leq \check\rho_0^\delta \leq \rho^*_0\andf \check\rho_0^\delta\to\rho_0\ \hbox{ pointwise}.
\end{equation}
Then, we  smooth out both $\check\rho_0^\delta$ and $v_0^\delta$ as follows:
\begin{equation}\label{e4}
 \rho^\delta_0 = \pi_\delta \ast \check\rho_0^\delta \mbox{ \ and \ }
 v_0^\delta = \pi_\delta \ast v_0,
\end{equation}
where $(\pi_\delta)_{\delta>0}$ is a family of positive mollifiers. 
\medbreak
Let us emphasize that the total mass of $\rho_0^\delta$ is still equal to one, and that
$\rho_0^\delta\geq\delta.$

\subsubsection*{Step 2.}
 We solve \eqref{CNS} with data $(\rho_0^\delta,v_0^\delta)$
 according to the classical literature.  For example, one may use the following result  (see \cite{D10,Mu03,S}): 
\begin{thm}\label{thm:lwp}
 Let $\rho_0\in W^1_p(\T^d)$ and  $v_0 \in W^{2-2/p}_p(\T^d)$ for some  $p>d,$ with $d=2,3.$
 Assume that $\rho_0>0.$ Then there exists
 $T_*>0$ depending only on the norms of the data, and on $\inf_{\T^d}\rho_0$ such that  \eqref{CNS} supplemented with data
 $\rho_0$ and $v_0$ has a  unique solution $(\rho,v)$ on the time interval $[0,T_*],$  satisfying\footnote{Recall that  $W^{1,2}_p(0,T_*\times\T^d)$ designates the set of
 functions $v:[0,T_*)\times\T^d\to\R^d$ such that $v\in W^{1}_p(0,T_*;L_p(\T^d))\cap L_p(0,T_*;W^2_p(\T^d)),$ and $W^{2-2/p}_p(\T^d),$ the corresponding trace space on $t=0$ (that may 
 be identified to the Besov space $B^{2-2/p}_{p,p}(\T^d)$).}
 \begin{equation}\label{e5}
  v \in W^{1,2}_p(0,T_*\times\T^d )\andf 
  \rho \in \cC([0,T_*];W^1_p(\T^d)).
 \end{equation}
\end{thm}
Let  $(\rho^\delta,v^\delta)$ be the maximal solution pertaining to data $(\rho_0^\delta,v_0^\delta)$ provided by the above statement, and let $T^\delta$ be  the  largest time  so that $(\rho^\delta,v^\delta)$ fulfills \eqref{e5} for all $T<T^\delta$. 
Since $(\rho^\delta,v^\delta)$ is  smooth and with no vacuum, it satisfies  all the formal estimates 
we proved so far,  with the same constants \emph{independent of $\delta.$} 
 In particular, $\div v^\delta$ is in $L_1(0,T;L_\infty)$ for all $T<T^\delta,$
 which implies that $\rho^\delta$ is bounded from below and above,
 according to:
\begin{equation}\label{e3}
 \delta \exp\biggl\{-\int_0^T \|\div v^\delta\|_\infty\,dt\biggr\} \leq \rho^\delta(t,x) \leq 
 \rho^*_0 \exp\biggl\{ \int_0^T \|\div v^\delta\|_\infty\,dt\biggr\}\cdotp
\end{equation}

\subsubsection*{Step 3.}  Our goal here is to prove that the solution $(\rho^\delta,v^\delta)$ is actually global. 
To achieve it, we  argue by contradiction, assuming that $T^\delta$ is finite.
\smallbreak
The classical estimates for the continuity equations imply that for all $T<T^\delta$
(dropping exponents $\delta$ on $(\rho^\delta,v^\delta),$ for better readability):
\begin{equation}\label{e6}
 \|\nabla \rho(T)\|_p \leq    \|\nabla \rho_0^\delta\|_p 
+ C\int_0^T\big( \|\nabla v\|_\infty \|\nabla \rho\|_p + \|\nabla \div v\|_p\big)dt.
\end{equation}
Observe that  the previous sections ensure that we have  
$\sqrt{\rho t}\,\dot v\in L_\infty(0,T^\delta;L_2)$ and $\sqrt t\,\nabla\dot v\in L_2(0,T^\delta;L_2).$
 Combining with straightforward interpolation arguments and H\"older inequality, we deduce that 
 \begin{eqnarray}\label{e10}
 \rho\dot v\in L_a(0,T^\delta;L_p(\T^2))&\quad\hbox{for all } \ 2<p<\infty\andf a<p'&\quad\!\!\!\hbox{if }\ d=2,\\\label{e11}
 \rho\dot v\in L_a(0,T^\delta;L_p(\T^3))&\quad\hbox{for all }\ p\in]2,6[\andf \frac1a=\frac54-\frac{3}{2p}&\quad\!\!\!\hbox{if }\ d=3.
 \end{eqnarray}
 Remembering that  $\Delta v + \nu \nabla \div v - \nabla P=-\rho\dot v,$  we thus get
\begin{equation}\label{e7}
\Delta v + \nu \nabla \div v - \nabla P\in  L_{a}(0,T;L_p), 
\end{equation}
whence, using $L_p$ estimates for the Riesz operator and the fact that $P=\rho^\gamma$
with $\rho$  bounded,  one may conclude that, uniformly with respect to $\delta,$ we have for all $t<T^\delta,$
\begin{equation}\label{e8}
 \|\nabla^2 v(t) \|_p  \leq C \|\nabla \rho(t) \|_p + h(t)\mbox{ \ \ with \ \ }
  h \in L_{a}(0,T^\delta).
\end{equation}
Hence we have for all $T<T^\delta,$
\begin{equation}\label{e9}
 \|\nabla \rho(T)\|_p \leq \left(\|\nabla \rho^\delta_0\|_p+\int_0^T h(t)\,dt\right)
 \exp\biggl\{\int_0^T C(1+\|\nabla v\|_\infty)\, dt\biggr\}\cdotp
\end{equation}

In order to close the estimates, we have to bound  $\|\nabla v\|_\infty$. 
Since $p>d$ and $\nabla v$ is bounded in $L_1(0,T;BMO)$ 
(recall Corollary \ref{c:c}),  one may start from  the following well known logarithmic inequality:
$$ \|\nabla v\|_\infty \leq C\|\nabla v\|_{BMO} \log\biggl(\e+\frac{\|\nabla v\|_{W^1_p}}{\|\nabla v\|_{BMO}}\biggr),$$
which, in light of \eqref{e8}, implies that 
$$
 \|\nabla v\|_\infty \leq C\|\nabla v\|_{BMO} \log\biggl(\e+\frac{h+\|\nabla \rho\|_{p}}{\|\nabla v\|_{BMO}}\biggr)\cdotp
$$
Hence, plugging that inequality in \eqref{e6}, we discover that for all $T<T^\delta,$
$$
 \|\nabla \rho(T)\|_p \leq    \|\nabla \rho_0^\delta\|_p 
+\int_0^Th\,dt +C\int_0^T\biggl(1+\|\nabla v\|_{BMO}\log\biggl(\e+\frac{h+\|\nabla \rho\|_{p}}{\|\nabla v\|_{BMO}}\biggr)\biggr)
 \|\nabla \rho\|_p\,dt.
$$
Since 
$$
\|\nabla v\|_{BMO}\log\biggl(\e+\frac{h}{\|\nabla v\|_{BMO}}\biggr)\leq C\max(h,\|\nabla v\|_{BMO})
$$
and
$$
\|\nabla v\|_{BMO}\log\biggl(\e+\frac{\|\nabla \rho\|_{p}}{\|\nabla v\|_{BMO}}\biggr)\leq C(1+\|\nabla v\|_{BMO})\log(\e+\|\nabla \rho\|_p),
$$
we get
$$
\displaylines{ \|\nabla \rho(T)\|_p \leq    \|\nabla \rho_0^\delta\|_p 
+\int_0^Th\bigl(1+C\|\nabla \rho\|_p\bigr)dt
+C\int_0^T\bigl(1+\|\nabla v\|_{BMO}\bigr)\log\bigl(\e+\|\nabla \rho\|_{p}\bigr) \|\nabla \rho\|_p\,dt.}
$$
From this and Osgood lemma,  one can conclude (as $T^\delta$ is finite) that  $\nabla \rho$ and 
$\nabla u$ belong to  $L_\infty(0,T^\delta;L_p)$ and $L_1(0,T^\delta;L_\infty),$ respectively. 

Putting  together with \eqref{e8}, this  leads  to 
\begin{equation}\label{e13}
 \rho_t =-\div (v \rho) \in L^{a}(0,T^\delta;L_p).
\end{equation}
Hence, by Sobolev embedding,  one can conclude  that there exists $\alpha >0$ such that
\begin{equation}\label{e14} \rho\in  C^{\alpha}([0,T^\delta)\times\T^d).\end{equation}
Now, one can  go back to the momentum equation of \eqref{CNS}, written in the form
\begin{equation}\label{e15}
 \rho v_t - \mu \Delta v - \nu \nabla \div v = -\nabla P - \rho v \cdot \nabla v.
\end{equation}
Thanks to \eqref{e3} and  \eqref{e14}, one may apply Theorem 2.2. of \cite{D10} and  get
\begin{equation}\label{e16}
 \|v\|_{W^{1,2}_p(0,T^\delta\times\T^d)} \leq C_\delta\left( \|\nabla P\|_{L_p(0,T^\delta\times\T^d)}
  + \|v\cdot \nabla v\|_{L_p(0,T^\delta\times\T^d)}\right)\cdotp
\end{equation}
For general $p>2$ if $d=2,$ or $2<p<6$ if $d=3,$ we do not know how to prove 
directly that $v\cdot\nabla v$ is in $L_p(0,T^\delta\times\T^d),$ 
and we shall need several steps. 

More precisely, if $d=2,$ then one may use the fact that for all $p<q<\infty,$ 
$$
\|\nabla v\|_{p^*}\leq C \|\nabla^2v\|_{p}^{1/2}\|v\|_q^{1/2}\with \frac1{p^*}=\frac12\biggl(\frac1p+\frac1q\biggr)
$$
which, combined with the fact that $v\in L_\infty(0,T^\delta;H^1(\T^2))$ (from Proposition  \ref{p:H1a}) and thus $v\in L_\infty(0,T^\delta;L_r(\T^2))$ for all $r<\infty,$
and \eqref{e8} implies that $v\cdot\nabla v\in L_{2a}(0,T^\delta;L_p(\T^2)).$
Hence the right-hand side of  \eqref{e15} belongs to  $L_{2a}(0,T^\delta;L_p(\T^2)),$  
and Theorem 2.2. of \cite{D10} implies that
$$
\d_t v\in  L_{2a}(0,T^\delta;L_p(\T^2))\andf \nabla^2 v \in L_{2a}(0,T^\delta;L_p(\T^2)).
$$
Starting from that new information and arguing as above entails that   the right-hand side of  \eqref{e15} belongs to  $L_{4a}(0,T^\delta;L_p(\T^2)),$
and so on. After a finite number of steps, we eventually reach $v\in W^{1,2}_p(0,T^\delta\times\T^2).$
\medbreak
% In the 2D case we have from Proposition  \ref{p:H1a} and Sobolev embedding,
%\begin{equation}\label{e17}
 %v \in L_\infty(0,T;H^1(\T^2)) \subset L_\infty(0,T;L_{q}(\T^2))\quad\hbox{for all }\ q<\infty.
%\end{equation}
%To estimate the gradient, one may use the  interpolation estimate:
%\begin{equation}\label{e18}
 %\|\nabla v\|_{{2p}} \leq  C\|\nabla^2 v\|_{p}^{1/p}\|v\|_p^{1-1/p}. 
%\end{equation}
For the 3D case we note that the information that  $v\in L_\infty(0,T^\delta;H^1(\T^2))$ implies that 
\begin{equation}\label{e19}
 v \in L_\infty(0,T;L_6(\T^3)).
\end{equation}
Hence, to bound $v\cdot\nabla v$ in $L_p(\T^3),$ we need  to have  $\nabla v$ in $L_k$ with $k$ such that $\frac{1}{6} + \frac{1}{k} = \frac{1}{p}$
(remember that $2<p<6$ in the 3D case).  By interpolation and the definition of $a$ in \eqref{e11}, we have
\begin{equation}\label{e20}
 \|\nabla v\|_k \leq C\|\nabla^2 v\|_p^{1-a/4} \|v\|_6^{a/4}.
\end{equation}
Hence 
$$
 \|\nabla v\|_k^{\frac{4a}{4-a}} \leq C\|\nabla^2 v\|_p^{a} \|v\|_6^{\frac{4a^2}{16-4a}}, 
 $$
and $v\cdot\nabla v$ is thus in $L^{4a/(4-a)}(0,T^\delta; L_p(\T^3))$ which, in view of Theorem 2.2. of \cite{D10} 
yields 
$$
\d_t v\in  L_{4a/(4-a)}(0,T^\delta;L_p(\T^3))\andf \nabla^2 v \in L_{4a/(4-a)}(0,T^\delta;L_p(\T^3)).
$$
Again, after a finite number of steps, we achieve 
%To understand the above inequality we note that 
%\begin{equation}\label{21} W^{1/2}_p(\T^3) \subset L_k(\T^3), \mbox{ \ since  \ \ } \frac{3}{1/2} \biggl(\frac 1p - \frac 1k\biggr) =1,
%\end{equation}and by definition $\frac 1p -\frac 1k = \frac 16\cdotp$\medbreak
%Hence, reverting to  \eqref{e16},  one may conclude that 
\begin{equation}\label{e22}
 \|v\|_{W^{1,2}_p(0,T^\delta\times\T^d)}<\infty.
\end{equation}
Now,  thanks to the trace theorem and the estimates that we proved for $\rho,$ one may conclude that, if $T^\delta$ is finite, then 
$$\sup_{T<T^\delta} \bigl(\|v(T)\|_{W^{2-2/p}_p}+\|\rho(T)\|_{W^1_p}\bigr)<\infty\andf
\inf_{T<T^\delta} \rho(T)>0.$$
Thanks to that information, one may solve System \eqref{CNS} supplemented with initial data $(\rho(T),v(T))$ whenever
$T<T^\delta,$ and the existence time $T_*$ provided by Theorem \ref{thm:lwp} is independent of $T.$ 
In that way, taking $T=T^\delta-T_*/2,$ we get a continuation of the solution beyond $T^\delta,$ thus contradicting the definition of $T^\delta.$

Hence $T^\delta=+\infty.$ In other words, the solution $(\rho^\delta,v^\delta)$ is global and all 
the estimates of the previous sections are true on $\R_+.$ Furthermore, 
it is clear that they  are uniform with respect to $\delta.$

\subsubsection*{Step 4.}

The previous step ensures uniform boundedness of $(\rho^\delta,v^\delta)$
in the desired existence space. 
The last step is to prove the  convergence of a subsequence. 
Since we have more  regularity  than  
in the classical weak solutions  theory, one can 
pass to the limit by following the steps therein. However, this would   give  some 
restriction on the pressure laws that one can consider
(typically, if $P=a\rho^\gamma,$ then one has to assume that $\gamma>d/2$). 
In our case, the higher  regularity of the velocity  will enable us to pass to the limit  for
rather general pressure laws,  and by means of a much more elementary method.  
\smallbreak
To start with, let us observe that, up to extraction, we have 
\begin{equation}\label{eee}
v^\delta\to v\ \hbox{ in } \ L_2(0,T\times\T^d)\ \hbox{ for all }\ T>0.\end{equation}
Indeed, since $(v^\delta)$ and $(\sqrt t\, v_t^\delta)$ are bounded in 
$L_2(0,T;L_2),$ Lemma 3.2 of \cite{DM-1} implies that 
$(v^\delta)$ is bounded in $H^{\frac12-\alpha}(0,T;L_2(\T^d))$ for all $\alpha>0,$ 
which, combined with the fact that $(v^\delta)$ is also bounded in $L_2(0,T;H^1(\T^d))$
implies that 
$$(v^\delta)\quad\hbox{is bounded in }\  H^{\frac14}(0,T\times \T^d).$$ 
This  entails \eqref{eee} by standard compact Sobolev embedding. 
\medbreak
This is still  not enough to pass to the limit in the pressure term of the  momentum equation. 
To achieve it, we shall exhibit some   strong convergence property for 
the effective viscous flux $G^\delta$.

 {}From  \eqref{e7} and uniform estimates given by the previous sections, one knows that  
 \begin{equation}\label{eeee}
 (G^\delta)\ \hbox{ is bounded in }\  L_{\infty}(0,T;L_2)\cap  L_2(0,T;H^1)\ \hbox{ for all finite }\ T>0,
 \end{equation}
  which already yields weak convergence.
  \smallbreak
   To get strong convergence, one can take advantage of uniform estimates for $(G^\delta_t)$:
    from the previous step, Sobolev embeddings  and  the relation 
$$P^\delta_t=-\div(P^\delta\,v^\delta)-h^\delta\div v^\delta,$$
we gather that $(P^\delta_t)$ is bounded in $L_2(0,T;W^{-1}_p)$ for all finite $T>0$
and $p<\infty$ (or just $p\leq 6$ if $d=3$). 
Furthermore, we also know that  $\sqrt t\,\div \dot v^\delta$ is bounded  in $L_2(0,T;L_2).$ 
Since $\div ( v^\delta \cdot \nabla v^\delta)$ is bounded  in $L_2(0,T;W^{-1}_p)$ (again, use the previous step), 
 one may conclude that   
 $$ \sqrt t\, G^\delta_t\quad\hbox{is bounded in }\ L_2(0,T;W^{-1}_p).$$
 By suitable modification of Lemma 3.2 of [7], we deduce that 
 $$(G^\delta)\ \hbox{ is bounded in }\  H^{\frac12-\alpha}(0,T;W^{-1}_p)\ \hbox{ for all }\ \alpha>0,$$
 and interpolating with \eqref{eeee} allows to get that 
 $(G^\delta)$  is bounded in  $H^{\beta}(0,T\times\T^d)$ for some small enough $\beta>0.$
 So, finally, up to extraction, we have
 \begin{equation}\label{eeeee}
G^\delta\to G\ \hbox{ in } \ L_2(0,T\times\T^d)\ \hbox{ for all }\ T>0.\end{equation}
  
% Hence by the Lions-Aubin theorem we get the strong convergence of $G$ in the $L_p$ spaces (up to a sub-sequence, point-wise a.e.).
%Note  that  it even implies the  strong convergence of the velocity, but
%at this moment we cannot have it for the gradient of the velocity because we are tied by
%the density involved in the effective viscous flux.

We are now in a good position 
to prove the strong convergence of the density. 
After suitable relabelling, the  previous considerations ensure that there exists 
a  sub-sequence $(\rho^n,v^n)_{n\in\N}$ of $(\rho^\delta,v^\delta)$ 
such that, for all $T>0,$
\begin{equation}\label{e25}
\rho^n \rightharpoonup^\ast \rho \mbox{  \  \ in \ }L_\infty(0,T\times\T^d)\andf v^n\to  v
\mbox{  \  \ in \ }L_2(0,T\times\T^d).
\end{equation}
Since for all $n\in\N,$ we have 
\begin{equation}\label{e26}
 \rho^n_t +\div(\rho^n v^n)=0,
\end{equation}
 the limit $(\rho,v)$  fulfills 
\begin{equation}\label{e27}
 \rho_t +\div(\rho v )=0.
\end{equation}
At this point, let us emphasize that, since $\div v\in L_1(0,T;L_\infty)$ 
(another consequence of the uniform estimates provided by the previous step) and 
$\rho\in L_\infty(0,T\times\T^d),$ one can assert that  
$\rho$ is actually a  \emph{renormalized} solution of \eqref{e27} 
(apply Theorem II.2 of \cite{DPL}), and thus fulfills 
\begin{equation}\label{e29}
 (\rho \log \rho)_t +\div(\rho \log \rho\: v)+\rho \,\div v =0.
\end{equation}
Of course, since $(\rho^n,v^n)$ is smooth, we also have
\begin{equation}\label{e28}
 (\rho^n \log \rho^n)_t +\div(\rho^n \log \rho^n \:v^n)+\rho^n \div v^n =0.
\end{equation}
Then, remembering the definition of $G^n,$ we get
\begin{equation}\label{e30}
 (\rho^n \log \rho^n)_t +\div(\rho^n \log \rho^n\: v^n)+ \nu^{-1} \rho^n P(\rho^n) + 
 \nu^{-1} \rho^n G^n=0
\end{equation}
and the limit version
\begin{equation}\label{e31}
 (\rho \log \rho)_t +\div(\rho \log \rho\: v)+\nu^{-1} \rho P(\rho)
 + \nu^{-1} \rho G =0.
\end{equation}
Denote by $\overline{\rho \log \rho}$ and $\overline{\rho P(\rho)}$
the weak limits of $\rho^n \log \rho^n$ and $\rho^n P(\rho^n)$, respectively. 
Since functions $z\mapsto z\log z$ and $z\mapsto zP(z)$ are convex, we 
know that \begin{equation}\label{e32}
 \overline{\rho \log \rho} \geq \rho \log \rho\andf \overline{\rho P(\rho)}\geq \rho P(\rho).
\end{equation}
Furthermore, integrating \eqref{e30} and \eqref{e31} on $[0,T]\times\T^d,$ we find that
$$\displaylines{
\nu\biggl(\int_{\T^d}(\rho^n\log\rho^n-\rho\log\rho)(T)\,dx
-\int_{\T^d}(\rho^n_0\log\rho^n_0-\rho_0\log\rho_0)\,dx\biggr)\hfill\cr\hfill+\int_0^T\!\!\!
\int_{\T^d}(\rho^nP^n-\rho P)\,dx\,dt +\int_0^T\!\!\!\int_{\T^d}(\rho^n G^n-\rho G)\,dx\,dt=0.}
$$
By construction, the term pertaining to the initial data tends to zero. Furthermore, 
since $(G^n)$ converges strongly to $G,$ the last term  also tends to zero. This leads us to 
$$
\nu\int_{\T^d}(\overline{\rho\log\rho}-\rho\log\rho)(T)\,dx
+\int_0^T\!\!\!
\int_{\T^d}(\overline{\rho P}-\rho P)\,dx\,dt=0.
$$ 
Combining with \eqref{e32}, one may now 
 conclude that 
\begin{equation}\label{e34}
 \overline{\rho \log \rho} = \rho \log \rho.
\end{equation}
Since the function $z\mapsto z \log z$ is strictly convex we find by  standard arguments 
that $(\rho^n)$ converges strongly and pointwise to $\rho.$  Hence one can pass to the limit
in all the nonlinear terms (in particular in the  pressure one) of the momentum equation, 
and  conclude that $(\rho,v)$ is indeed a solution to \eqref{CNS}. 

Besides, classical arguments that may be found in \cite{DPL} ensure that
$\rho\in\cC(\R_+;L_p)$ for all $p<\infty,$ and that strong convergence holds true 
in the corresponding space. 
Thanks to that information, since \eqref{eq:energy} is fulfilled with data $(\rho_0^n,v_0^n)$ by 
the sequence $(\rho^n,v^n)_{n\in\N},$ one may pass to the limit
and see that $(\rho,v)$ satisfies \eqref{eq:energy} as well.
Finally, since the internal energy $e$ is continuous with respect to time
(a consequence of the strong convergence of $\rho$), one may reproduce the argument
that has been used in \cite{DM-1} so as to prove 
that   $\sqrt\rho\, v\in \cC(\R_+;L_2).$ 
This completes the proof of our existence theorems in dimensions $2$ and $3.$ \qed
\bigbreak
{\bf Proof of Theorem \ref{thm:asymptotic}.}
We end this section with a fast justification of the convergence of solutions to \eqref{CNS}
to those of \eqref{INS} when $\nu$ goes to $\infty,$ leading to Theorem \ref{thm:asymptotic}.  
As the proof goes along the lines of that of Theorem \ref{thm:global2}, we just indicate the main steps.
The starting point is the estimate provided by Proposition \ref{p:H1a} which ensures
in particular \eqref{eq:div}, that $(\nabla G^\nu)$ is bounded
in $L_2(\R_+\times\T^2)$  and that $(v^\nu)$ is bounded in $L_\infty(\R_+;H^1),$
while Proposition \ref{p:boundrho2} guarantees that $(\rho^\nu)$ is bounded in $L_\infty(\R_+\times\T^2).$
Hence, there exists $(\rho,v)\in L_\infty(\R_+\times\T^2)\times L_\infty(\R_+;H^1)$
and a subsequence $(\rho^n,v^n)$ of $(\rho^\nu,v^\nu)$ such that
$$\rho^n\rightharpoonup^*\rho\ \hbox{ in }\ L_\infty(\R_+\times\T^2)\andf v^n\rightharpoonup v
\ \hbox{ in }\ L_\infty(\R_+;H^1).$$
As in the proof of existence, in order to get some compactness, one may 
look at time weighted estimates. More specifically, we know from Proposition \ref{p:weight} that 
if $\nu\geq\nu_0$ then
 $$ \sup_{t\in[0,T]} \int_{\T^2} \rho |\dot v^\nu|^2 t\,dx +
  \int_{0}^{T} \!\!\!\int_{\T^2} (\mu|\nabla\cP \dot v^\nu|^2 + \nu | \div \dot v^\nu|^2) t \,dx \,dt 
  \leq C_0T\,e^{\frac{C_0T}\nu},$$
and this ensures  that  $(v^\nu)$ is bounded  in, say, $H^{1/4}(0,T\times\T^2)$ for all $T>0.$
Hence, we actually have (extracting one more  subsequence as the case may be), 
$$v^n\to v\quad\hbox{in}\quad L_{2,loc}(\R_+;L_2(\T^2)).$$ 
Next, arguing exactly as in the proof of existence, we get that, for all finite $T>0$ and $p<\infty,$
$$(G^\nu)\ \hbox{ is bounded in }\ L_\infty(0,T;L_2)\cap L_2(0,T;H^1)\andf
(\sqrt t\,G^\nu_t)\ \hbox{ is bounded in }\ L_2(0,T;W^{-1}_p)$$
from which we deduce that
$(G^\nu)$ is bounded in $H^{\frac12-\alpha}(0,T;W_p^{-1})$ for all $\alpha>0$ and, eventually 
$$G^n\to G\quad\hbox{in}\quad L^2_{loc}(\R_+;L_2(\T^2)).$$ 
Putting together all those results of convergence, one gets
$$
\d_t\rho+\div(\rho v)=0\andf
\d_t(\rho v)+\div(\rho v\otimes v)-\mu\Delta\cP v+\nabla G=0.
$$
Since we know in addition (from \eqref{eq:div}) that $\div v=0,$ one can conclude that $(\rho,v,\nabla G)$
satisfies \eqref{INS}.  Finally, from the uniform bounds that are available for $(\rho^\nu,v^\nu),$
one may check that $(\rho,v,\nabla G)$ has the regularity of the solution constructed in Theorem 2.1
of \cite{DM-1}, which is unique. Hence the whole family $(\rho^\nu,v^\nu)$ converges to $(\rho,v).$
\qed

%%%%%%%%%%%%%%%%%%%%%%%%%%%%%%%%%%%%%%%%%%%%%%%%%%%%%%

\section{The proof of uniqueness}\label{s:uniqueness}

 Here we show the uniqueness of the solutions we  constructed in the paper, 
 both in dimensions $2$ and $3.$ The main difficulty we have to face
 is that having $\div v$ and $\nabla\cP v$ in $L_1(0,T;L_\infty)$ (see Corollary \ref{c:c}) \emph{does not} ensure that
 $\nabla v$ is in $L_1(0,T;L_\infty)$ so that, in contrast with our recent work \cite{DFP}, 
 one cannot  reformulate  System \eqref{CNS} in Lagrangian coordinates to prove uniqueness. 
 However, we do have $\nabla v$ is in $L_{r}(0,T;BMO)$ for some $r>1,$  which will turn out 
 to be enough to prove uniqueness \emph{provided that the pressure law is linear}. 
 Actually, we encounter  the same difficulty as   in  D. Hoff's 
 paper  \cite{Hoff2}:    we need, at some point,    to  bound  the difference of the pressures in $H^{-1}$
 by the norm of the difference of the densities  in $H^{-1},$ which is impossible if $P$ is nonlinear.   
 
 %Here is the main statement of this section.
 \begin{prop}\label{p:uniqueness} Assume that  $P(\rho)=a\rho$ for some $a>0,$ 
 and consider two finite energy solutions $(\rho,v)$ and $(\bar\rho,\bar v)$ of \eqref{CNS}
on $[0,T_0]\times\T^d$ ($d=2,3$) with bounded density, satisfying 
\eqref{eq:conservation} and  emanating from the same initial data. 
If, in addition, $v$ and $\bar v$ are in $L_\infty(0,T_0;H^1),$
$\sqrt t\,\nabla \dot v$ and $\sqrt t\,\nabla\dot{\bar v}$ are in $L_2(0,T_0\times\T^d),$
$\sqrt{\rho t}\,\dot v,\,\sqrt{\bar\rho t}\,\dot{\bar v}$ belong to  $L_\infty(0,T_0;L_2),$
\begin{equation}
\label{r-0}
\nabla\bar v \in L_2(0,T_0;L_3(\T^d))
\andf \int_0^{T_0}(1+|\log t|)\|\nabla\bar v(t)\|_{BMO}\,dt <\infty,\end{equation}
then $(\bar\rho,\bar v)\equiv(\rho,v)$ on $[0,T_0]\times\T^d.$
\end{prop}
\begin{proof} The general scheme of the proof is the same in dimensions $2$ or $3.$
Assume that $a=1$ for notational simplicity and 
consider two solutions $(\rho,v)$ and $(\bar \rho,\bar v)$ to \eqref{CNS}
corresponding to  the same initial data $(\rho_0,v_0)$.
The system for the difference
$$
 \dr: = \rho -\bar \rho \andf \dv:= v -\bar v
$$
reads 
\begin{equation}\label{r0}
\left\{ \begin{array}{l}
\displaystyle  \dr_t + \div(\dr\,\bar v + \rho \dv)=0, \\[7pt]
\displaystyle  \rho \dv_t + \rho v \cdot \nabla \dv - \mu\Delta \dv - (\lambda+\mu) \nabla \div \dv +  \nabla \dr= \dr \,\dot{\bar v} + \rho \dv \cdot\nabla\bar v.
 \end{array}\right.
\end{equation}

In order to show that $\dr\equiv0$ and $\dv\equiv0,$   we shall  perform  estimates in
 $L_\infty(0,T;\dot H^{-1})$ for $\dr,$
and  in $L_2(0,T\times\T^d)$ for $\sqrt\rho\,\dv.$ 
To this end, we set $\phi:=-(-\Delta)^{-1}\dr$ (which makes sense, since $\int_{\T^d}\dr\,dx=0$) so that 
\begin{equation}\label{r000}
 \|\nabla \phi\|_{2} = \|\dr\|_{\dot H^{-1}}=\|\dr\|_{H^{-1}}.
\end{equation}
Now, testing the first equation of \eqref{r0} by $\phi$ yields
$$
\frac12 \frac{d}{dt} \|\nabla \phi\|_2^2 \leq \bigg\vert\int_{\T^d}  (\bar v\cdot\nabla\phi \,\dr 
 + \rho\, \dv\cdot\nabla \phi)\,dx\bigg\vert\cdotp 
$$
The last term is  bounded as follows:
$$
 \biggl|\int_{\T^d} \rho \,\dv\cdot \nabla \phi\, dx\biggr|\leq 
 \sqrt{\rho^*}\, \|\sqrt\rho\,\dv\|_2\|\nabla \phi\|_2.
$$
Regarding  the first one, observe that (with the usual summation convention)
$$
 \int_{\T^d}  \bar v \cdot \nabla \phi \,\dr\,dx =
 \int_{\T^d}\bar v^j\d_j\phi\,\Delta\phi\,dx
 =- \int_{\T^d} \d_k\bar v^j\,\d_j\phi\,\d_k\phi\,dx
 +\frac12\int_{\T^d}\div\bar v\,|\nabla\phi|^2\,dx.
 $$
 Hence, we have 
$$
 \biggl |\int_{\T^d}  \bar v \cdot \nabla \phi \,\dr \, dx \biggr| 
 \leq C\|\nabla\bar v\|_{\rm BMO} \|\nabla \phi\otimes\nabla\phi\|_{\mathcal{H}^1}.
$$
Now, in light of the following inequality (see e.g. \cite[Thm. D]{M10}), 
\begin{equation}\label{eq:log}
 \|f\|_{\mathcal{H}^1} \leq C\|f\|_{1} ( |\log \|f\|_1| + \log (\e+\|f\|_\infty)),
\end{equation}
we discover that
\begin{equation}\label{r00}
 \biggl |\int_{\T^d}  \bar v \cdot \nabla \phi \,\dr \, dx \biggr|  \leq C\|\nabla\bar v\|_{\rm BMO}\|\nabla \phi\|_{2}^2 \bigl(|\log \|\nabla\phi\|_2^2| + \log(\e+\|\nabla\phi\|_\infty^2)\bigr)\cdotp
\end{equation}
Since the densities are bounded by $\rho^*$,  we have
$$
 \|\nabla \phi(t)\|_{\infty}\leq C\rho^* \mbox{ \ \ for all \  } t\in[0,T_0].
$$
Hence Inequality \eqref{r00} implies that for some constant $C$ depending only on $\rho^*,$
\begin{equation}\label{r-3}
 \frac12\frac{d}{dt} \|\nabla \phi\|_2^2 \leq C\Bigl( \|\sqrt\rho\,\dv\|_2 + \|\nabla\bar v\|_{\rm BMO}\|\nabla \phi\|_{2}(1+|\log\|\nabla\phi\|_2|)\Bigr)\|\nabla\phi\|_2\cdotp
\end{equation}
Since $\nabla\phi(0)=0,$   this gives after integration that for all $t\in[0,T_0],$ 
\begin{equation}\label{r-4}
 \|\nabla \phi(t)\|_2 \leq C\biggl(\int_0^t\|\sqrt\rho\,\dv\|_2\,d\tau +
 \int_0^t \|\nabla\bar v\|_{\rm BMO}\|\nabla \phi\|_{2}(1+|\log\|\nabla\phi\|_2|)\,d\tau\biggr)\cdotp
\end{equation}
%{\bf R. And here we have a problem : I think that it implies the below
%inequality with a smaller power than $1/2$ if applying Osgood lemma. }
%$$\begin{aligned}
 %\|\nabla \phi(t)\|_2& \leq C\log(\e+{\rho^*}) \int_0^t e^{ C\log(\e+{\rho^*})\int_s^t\|\nabla \bar v\|_{\rm BMO}\, d\tau}  \|(\sqrt\rho\,\dv)(s)\|_2\, ds \\
 %&\leq C_{0,T}\, t^{1/2} \left( \int_0^t \|(\sqrt\rho\,\dv)(s)\|_2^2\,ds \right)^{1/2},
%\end{aligned}$$
%with $C_{0,T}$ depending only on $\rho^*,$ $T$ and $\|\nabla\bar v\|_{L_1(0,T;{\rm BMO})}.$
%\smallbreak
%This leads to the following relation between $\dr$  and $\dv$:
Hence, using \eqref{r000} and denoting $Z(t):=  \sup_{\tau\leq t} \tau^{-1/2}  \|\dr(\tau)\|_{\dot H^{-1}},$ we
get after using Cauchy-Schwarz inequality, for all $T\in[0,T_0],$ 
\begin{equation}\label{r-10}
Z(T) \leq C\biggl(\sup_{t\in[0,T]} \int_0^t \|\nabla\bar v\|_{\rm BMO}\,  Z(1+|\log\tau|+|\log Z|)\,d\tau
+\|\sqrt\rho\,\dv\|_{L_2(0,T\times\T^d)}\biggr)\cdotp
\end{equation}

In order to control the difference of the velocities, we introduce the solution $w$ to the 
following backward parabolic system: 
\begin{equation}\label{r1}
\left\{\begin{array}{l} \rho w_t + \rho v \cdot \nabla w + \mu\Delta w + (\lambda\!+\!\mu) \nabla \div w = -\rho \dv,\\[2pt]
w|_{t=T}=0.\end{array}\right.
\end{equation}
Solving the above system is not part of the classical 
theory for linear parabolic systems, as the coefficients are rough and may vanish. 
However, if $\rho$ and $v$ are regular with $\rho$ bounded away from zero, this is
 well known, and the case we are interested may be achieved by a regularizing process
of $\rho$ and $v,$ after using  Inequality \eqref{r5} below for the corresponding regular solutions. 
\medbreak
Now, testing the equation by $w,$ we find that 
\begin{equation}
 \sup_{t \in (0,T)} \int_{\T^d} \rho |w|^2\, dx  
 +\int_0^T\!\!\!\int_{\T^d}\bigl(\mu|\nabla\cP w|^2+\nu(\div w)^2\bigr)\,dx\,dt \leq  \int_0^T\!\!\!\int_{\T^d} \rho |\dv|^2 \,dx\,dt.
\end{equation}

Next, we test \eqref{r1} by $w_t$  and take advantage of the usual elliptic estimates given by 
$$
\mu\Delta w+(\lambda+\mu)\nabla\div w=-\rho\dot w-\rho\dv
$$
that ensure that 
$$
\mu^2\|\nabla^2\cP w\|_2^2+\nu^2\|\nabla\div w\|_2^2=\|\rho\dot w+\rho\dv\|_2^2
\leq\rho^*\|\sqrt\rho\,(\dot w+\dv)\|_2^2
$$
in order to get 
\begin{multline}\label{r2}
 \sup_{t\in (0,T)} \int_{\T^d}\bigl( \mu|\nabla\cP w(t)|^2 \!+\! \nu(\div w(t))^2\bigr)\,dx\\ + \int_0^T\!\!\! \int_{\T^d} \biggl(\rho |w_t|^2
 \!+\!\frac{\mu^2}{6\rho^*}|\nabla^2\cP w|^2\!+\!\frac{\nu^2}{6\rho^*}|\nabla\div w|^2\biggr)dx\,dt 
\leq \frac32\int_0^T\!\!\!\int_{\T^d} \bigl(\rho |\dv|^2 +\rho |v\cdot \nabla w|^2\bigr)\,dx\,dt.
\end{multline}
If $d=2$ then we bound the last term as follows:
$$
\begin{aligned}
 \int_{\T^2} \rho |v\cdot \nabla w|^2 dx  &\leq \sqrt{\rho^*} \|\rho^{1/4} v\|^2_4  \|\nabla w\|^2_4  \\
 &\leq C\sqrt{\rho^*} \|\rho^{1/4} v\|^2_4  \|\nabla w\|_2  \| \nabla^2 w\|_2 \\ 
 &\leq C\frac{(\rho^*)^2}{\mu^2} \|\rho^{1/4} v\|^4_4  \|\nabla w\|_2^2  + \frac{\mu^2}{12\rho^*} \|\nabla^2 w\|^2_2.
\end{aligned}
$$
If $d=3,$ then we rather write that
$$
\begin{aligned}
 \int_{\T^3} \rho |v\cdot \nabla w|^2 dx  &\leq (\rho^*)^{3/4} \|\sqrt \rho\, v\|^{1/2}_2 \|\nabla v\|_2^{3/2} \|\nabla w\|^{1/2}_2  \|\nabla^2 w\|^{3/2}_2 \\
 &\leq C(\rho^*)^3\|\sqrt\rho\, v\|^2_2 \|\nabla v\|_2^6 \|\nabla w\|_2^2  + \frac{\mu^2}{12\rho^*} \|\nabla^2 w\|^2_2.
\end{aligned}
$$
Hence, using  the properties of regularity of $v$,  plugging the above inequality in \eqref{r2}, then resorting to Gronwall inequality, we get 
\begin{multline}\label{r5}
 \sup_{t\in (0,T)} \int_{\T^d}\bigl( \rho|w|^2 + \mu |\nabla\cP w|^2 + \nu (\div w)^2\bigr)dx \\
 +\int_0^T \!\!\!\int_{\T^d}\bigl( \mu|\nabla\cP w|^2 +\nu(\div w)^2 +
\mu^2 |\nabla^2\cP w|^2 + \nu^2 |\nabla \div w|^2 + \rho|w_t|^2\bigr) \,dx\,dt 
 \\\leq C_T \int_0^T\!\!\! \int_{\T^d} \rho |\dv|^2 \,dx\,dt,
\end{multline}
with $C_T$ depending only on the norms of the two solutions on $[0,T].$
\medbreak
Let us next test  \eqref{r0} by $w.$ We get
\begin{equation}\label{r4}
 \int_0^T \!\!\!\int_{\T^d} \rho |\dv|^2 \,dx\,dt -   \int_0^T \!\!\!\int_{\T^d} \dr\, \div w \,dx\,dt 
 \leq \int_0^T\!\!\! \int_{\T^d}\bigl( \dr\,\dot{\bar v}\cdot w + \rho (\dv\cdot \nabla\bar v)\cdot  w\bigr) \,dx\,dt.
\end{equation}

One can bound the first term of the right-hand side as follows:
% (using the  fact that $\|\dr\|_{\dot H^{-1}}=\|\dr\|_{H^{-1}}$ since $\int_{\T^2}\dr\,dx=0$):
$$\begin{aligned}
\biggl| \int_0^T \!\!\!\int_{\T^d} \dr\,\dot{\bar v}\cdot w\,dx\,dt \biggr|&=
\biggl| \int_0^T \!\!\!\int_{\T^d} t^{-1/2} \dr\: t^{1/2}\dot{\bar v}\cdot w\,dx\,dt \biggr|\\
&\leq  \| t^{-1/2} \dr\|_{L_\infty(0,T;\dot H^{-1})} \| t^{1/2}\nabla(\dot{\bar v}\cdot w)\|_{L_1(0,T;L_2)}\\
&\leq \| t^{-1/2}\dr \|_{L_\infty(0,T;\dot H^{-1})} \bigl(\|\sqrt{t} \nabla \dot{\bar v}\|_{L_2(0,T;L_2)}\|w\|_{L_2(0,T;L_\infty)} \\
&\hspace{5cm}+\|\sqrt{t}\, \dot{\bar v}\|_{L_{2}(0,T;L_{6})} \|\nabla w \|_{L_2(0,T;L_{3})}\bigr)\cdotp
\end{aligned}$$
%where $(q_1,q_2)\in(2,+\infty)^2$ are chosen so that  $\frac12=\frac1{q_1}+\frac{1}{q_2}\cdotp$
%\medbreak
For bounding the last term of \eqref{r4}, one can  just use the fact that
$$
 \int_0^T \!\!\!\int_{\T^d} \rho (\dv \cdot \nabla\bar v)\cdot w \,dx\,dt \leq \sqrt{ \rho^*} \|\sqrt {\rho}\, \dv\|_{L_2(0,T;L_2)}
 \|\nabla\bar v\|_{L_2(0,T;L_3)} \|w\|_{L_\infty(0,T;L_6)}.
 %\\&\leq C  \|\nabla \bar v\|_{L_4(0,T;L_4)} T^{1/4} \int_0^T\!\!\! \int_{\T^2} \rho |\dv|^2 \,dx\,dt, \end{aligned}
 $$
Finally, we note that 
$$
 \int_0^T\!\!\! \int_{\T^d} \dr \,\div w \,dx \,dt \leq  T \|t^{-1/2}\dr\|_{L_\infty (0,T;\dot H^{-1} )} \|\nabla\div w\|_{L_2(0,T;L_2)}.
%&\leq C_0(T) T^{1/2}\nu^{-1} \|\dr\|_{L_\infty (0,T;H^{-1} )} \|\sqrt{\rho} \dv\|_{L_2(0,T.L_2)}.
$$
Plugging the above three inequalities in \eqref{r4}, we get
\begin{multline}\label{r-9}
\int_0^T\!\!\! \int_{\T^d} \rho |\dv|^2 \,dx\,dt \leq
\|t^{-1/2}\dr\|_{L_\infty(0,T;\dot H^{-1})}\Bigl( T \|\nabla\div w\|_{L_2(0,T;L_2)} \\
+ C\|\sqrt{t} \nabla \dot{\bar v}\|_{L_2(0,T;L_2)}\|w\|_{L_2(0,T;L_\infty)} 
+C\|\sqrt{t} \dot{\bar v}\|_{L_{2}(0,T;L_{6})} \|\nabla w \|_{L_2(0,T;L_{3})}\Bigr)
\\
+C\|\sqrt\rho\,\dv\|_{L_2(0,T;L_2)}\|\nabla\bar v\|_{L_2(0,T;L_3)}\|w\|_{L_\infty(0,T;L_6)}.
\end{multline}

Observe  that our assumptions on $\bar v$ guarantee that we have
\begin{equation}
 \|\sqrt{t} \nabla \dot{\bar v}\|_{L_2(0,T;L_2)}+ \|\sqrt{t}\, \dot{\bar v}\|_{L_{2}(0,T;L_{6})} + \|\nabla\bar v\|_{L_2(0,T;L_3)} \leq C_{T}.
\end{equation}
%Indeed, the first and last terms just above are bounded by means of \eqref{eq:weight} and \eqref{w22}, respectively, and Sobolev embedding.
%For the second one,   this is $\dot H^1(\T^3)\hookrightarrow L_6(\T^3)$ and, if $d=2,$ the consequence of 
%$$ \|\sqrt{t} \,\dot{\bar v}\|_{L_{2}(0,T;L_{6})} \lesssim  \|\sqrt{t}\, \dot{\bar v}\|_{L_{2}(0,T;L_4)}^{2/3}
 %\|\sqrt{t}\nabla\dot{\bar v}\|_{L_{2}(0,T;L_2)}^{1/3}, $$ 
%then \eqref{eq:weight} and \eqref{eq:dotv4}.\medbreak
Next, we  have to bound the terms containing $w$ in \eqref{r-9} by means of the data.  
Since  $\int_{\T^d} \rho w\,dx$ need not be zero,  Poincar\'e inequality \eqref{eq:poincarep} becomes
$$\|w\|_2\leq \biggl|\int_{\T^d}\rho w\,dx\biggr| + \rho^*\|\nabla w\|_2.$$
To bound the mean value of $\rho w,$ we note that integrating \eqref{r1} on $[t,T]\times\T^d$ readily gives
$$
 \int_{\T^d} (\rho w)(t,x) \, dx = \int_t^T\!\!\!\int_{\T^d} (\rho\,\dv)(\tau,x) \, dx\,d\tau.
 $$
 Therefore we have 
$$ \left| \int_{\T^d} (\rho w)(t) \, dx \right| \leq \sqrt{\rho^*}\, T^{1/2} \|\sqrt\rho\,\dv\|_{L_2(0,T;L_2)}\quad\hbox{for all }\ t\in[0,T],
$$
whence
\begin{equation}\label{r-11}
 \|w(t)\|_2 \leq C_{\rho^*}\bigl(\|\nabla w(t)\|_2 + T^{1/2} \|\sqrt\rho\,\dv\|_{L_2(0,T;L_2)}) \mbox{ \ \ for all \ } t\in [0,T].
\end{equation}
Then, combining with \eqref{r5}, we end up with 
$$
 \|w\|_{L_2(0,T;H^1)} \leq C_{0,T} T^{1/2} \|\sqrt\rho\,\dv\|_{L_2(0,T;L_2)}\andf
  \|\nabla w\|_{L_2(0,T;\dot H^1)}\leq C_{0,T}  \|\sqrt\rho\,\dv\|_{L_2(0,T;L_2)}.
$$
By interpolation and Sobolev embedding, it  follows that  for small enough $\varepsilon$ if $d=2$
(and $\varepsilon=1/4$ if $d=3$), we have
\begin{equation}
 \|w\|_{L_2(0,T;L_\infty)} \leq C_\varepsilon T^{1/2-\varepsilon}  \|\sqrt\rho\,\dv\|_{L_2(0,T;L_2)}.
\end{equation}
Likewise,  we have
$$\begin{aligned}
\|\nabla w\|_{L_2(0,T;L_{3}(\T^2))}& \lesssim \|\nabla w\|_{L_2(0,T;L_2(\T^2))}^{2/3}\|\nabla^2w\|_{L_2(0,T;L_2(\T^2))}^{1/3}\\
\andf \|\nabla w\|_{L_2(0,T;L_{3}(\T^3))}&\lesssim \|\nabla w\|_{L_2(0,T;L_2(\T^3))}^{1/2}\|\nabla^2w\|_{L_2(0,T;L_2(\T^3))}^{1/2},\end{aligned}$$
whence 
\begin{equation}
 \|\nabla w\|_{L_2(0,T;L_{3})} \leq C T^{\alpha}\|\sqrt\rho\,\dv\|_{L_2(0,T;L_2)}\with 
 \alpha=1/3\hbox{ if } d=2,\quad
 \alpha=1/4\hbox{ if } d=3.
\end{equation}
Finally,  using once more that  $H^1(\T^d)\hookrightarrow L_6(\T^d)$ for $d=2,3,$ we get
after plugging all the above inequalities  in  \eqref{r-9}, for all $T\in[0,T_0],$ 
$$
\|\sqrt\rho \dv\|^2_{L_2(0,T;L_2)}\leq  CT^{1/3}\Bigl(\|t^{-1/2}\dr\|_{L^\infty(0,T;\dot H^{-1})}
 \|\sqrt\rho \dv\|_{L_2(0,T;L_2)}  + \|\sqrt\rho \dv\|_{L_2(0,T;L_2)}^2\Bigr)\cdotp
 $$
 Clearly, the above inequality implies that, if $T$ is small enough then 
 \begin{equation}\label{r-12}
 \|\sqrt\rho \dv\|_{L_2(0,T;L_2)}\leq   CT^{1/3} Z(T).
 \end{equation}
 Plugging that inequality in \eqref{r-10} and assuming that $T$ is small enough, we obtain 
 $$
 Z(t)\leq C_T\int_0^t\bigl(1+\|\nabla\bar v(\tau)\|_{BMO}\,(1+|\log\tau|+|\log Z(\tau)|)\bigr)
Z(\tau) \,d\tau
 \quad\hbox{for all } \ t\in[0,T].
 $$
 Then, Osgood lemma (see e.g. \cite[Lem. 3.4]{BCD}) implies that $Z\equiv0$ on $[0,T],$ and thus, 
 owing to \eqref{r-12}, that $\sqrt \rho\,\dv\equiv0$ on $[0,T].$ 
 %Then, one can combine with Inequality \eqref{r-10} and 
 %conclude that for small enough $T,$ we  have 
 %\begin{equation}
 %\int_0^T\!\!\! \int_{\T^d} \rho |\dv|^2 \,dx\,dt =0.
%\end{equation}
%Reverting to \eqref{r-10} enables us to conclude to  $\dr \equiv 0,$ too,  on $[0,T].$ 
\smallbreak
Now, since $\sqrt\rho\,\dv$ and $\dr$ are zero, the second equation of \eqref{r0} becomes
$$
 \rho \dv_t + \rho v \cdot \nabla \dv - \mu\Delta \dv - (\lambda+\mu) \nabla \div \dv =0,
$$
which implies that 
$$
\frac12 \|\sqrt\rho \dv\|^2_{L_\infty(0,T;L_2)}+\int_0^T\bigl(\mu\|\nabla\cP\dv\|_2^2+\nu\|\div\dv\|_2^2\bigr)\,dx=0.
$$
Since $\int_{\T^d}\rho\,\dv\,dx=0,$  one gets (in light of Inequality \eqref{eq:poincarep})
that $\dv=0$  on $[0,T],$ which completes the proof of uniqueness.
\medbreak
To complete the proof of Theorem \ref{thm:uniqueness}, it suffices to observe
that Inequality \eqref{c1x} implies  Assumption \eqref{r-0} in Proposition \ref{p:uniqueness}. 
\end{proof}
%\medbreak
%We end the section with  a rapid study of the large bulk viscosity asymptotics, 
%which turns out to be  an easy variation  on Step 4 of the proof of existence. 
%Here is the statement we want to prove.
%\begin{thm}\label{thm:limit} Let the assumptions of Theorem \ref{thm:global2} be in force, uniformly with respect to $\nu,$
%and denote by $(\rho^\nu,v^\nu)$ the solution constructed therein. 
%Then there exists $(\rho,v)$ so that $(\rho^\nu,v^\nu)$ converges strongly to $(\rho,v)$
%in $L_\infty(0,T;L_p)\times L_2(0,T;L_2)$ for all $p<\infty$ and $T>0,$ where
%$(\rho,v)$ stands for the unique solution of \eqref{INS} (for some appropriate $\nabla\Pi$) supplemented
%with initial data $(\rho_0,v_0).$\end{thm}

%%%%%%%%%%%%%%%%%%%%%%%%%%%%%%%%%%%%

\begin{appendix}
\section{Some  inequalities}
The following Osgood type lemma has  been used in Section \ref{s:reg}.
\begin{lem}\label{l:osgood}
Let $f$ and $g$ be two locally integrable nonnegative functions on $\R_+,$ and 
assume that the a.e. differentiable function $X:\R_+\to\R_+$ satisfies
$$
X'\leq f\,X\log(A+BX)+g\,X\quad\hbox{for some }\ A\geq1\ \hbox{ and }\ B\geq0.
$$
Then we have for all $t\geq0,$
$$
A+BX(t)\leq \Bigl(A+Be^{\int_0^tg\,d\tau}X(0)\Bigr)^{\exp\int_0^tf\,d\tau}\cdotp
$$
\end{lem}
\begin{proof}
It suffices to prove the inequality on $[0,T]$ for all $T\geq0.$
Setting $Y(t):=e^{-\int_0^tg\,d\tau}X(t),$ then 
$Z(t):=C_TY(t)$ with $C_T:=\exp\int_0^Tg\,d\tau,$ we have for all $t\in[0,T],$ 
$$(A+BZ)'=BZ'\leq  BZ \log(A+BZ)\, f\leq \bigl(A\!+\!BZ\bigr)\log(A\!+\!BZ)\,f.$$
Therefore,  integrating once, 
$$
\log(\log(A+BZ(t)))\leq \log(\log(A+BZ(0)))+\int_0^tf\,d\tau\quad\hbox{for all }\ t\in[0,T].
$$
Then considering  $t=T$ and taking $\exp$ twice gives 
$$
A+BZ(T)\leq(A+BZ(0))^{\exp\int_0^Tf\,d\tau}.
$$
Reverting to the original function $X$ gives exactly what we want at $t=T.$
\end{proof}

We also used  the following  Poincar\'e inequality.
\begin{lem} Let $\rho$ be in $L_{p'}(\T^d)$ with $\frac1p+\frac1{p'}=1$
and $2\leq p\leq\frac{2d}{d-2}$ if $d\geq3$ ($2\leq p<+\infty$ if $d=2$). 
Assume that
\begin{equation}\label{eq:momnul}
\int_{\T^d}\rho b\,dx=0\andf M:=\int_{\T^d}\rho\,dx>0.\end{equation}
There exists a constant $C_p$ depending on $p$ and on $d$ (and with $C_2=1$),
such that 
\begin{equation}\label{eq:poincarep}
\|b\|_{2}\leq\biggl(1+\frac{C_p}M\|\rho-c\|_{{p'}}\biggr)\|\nabla b\|_{2}\quad\hbox{for any real number }\ c.
\end{equation}
Furthermore, in dimension $d=2,$ we have
\begin{equation}\label{eq:limitpoincare}
\|b\|_{2}\leq C\log^{\frac12}\biggl(\e+\frac{\|\rho-c\|_{2}}{M}\biggr)\|\nabla b\|_{2}.
 \end{equation}
\end{lem}
\begin{proof}
Let  $\bar b$ be the average of $b$ and $\wt b:=b-\bar b.$ Then we have by Poincar\'e inequality, 
\begin{equation}\label{eq:zz}
\|b\|_{2}\leq|\bar b|+\|\wt b\|_{2}\leq|\bar b|+\|\nabla b\|_{2}.
\end{equation}
Now,  hypothesis \eqref{eq:momnul} implies that  for all real number $c,$ we have
\begin{equation}\label{eq:average}
-M\bar b=\int_{\T^d}(\rho-c)\wt b\,dx.
\end{equation}
Therefore, by Sobolev inequality,
\begin{equation}\label{eq:averageb}
M|\bar b|\leq\|\rho-c\|_{p'}\|\wt b\|_{p}\leq C_p\|\rho-c\|_{{p'}}\|\nabla b\|_{2}
\end{equation}
and, clearly, $C_2=1.$ This gives \eqref{eq:poincarep}.
\medbreak
To handle  the endpoint case $d=2$ and $p=+\infty,$
 decompose $\wt b$ into Fourier series:
$$
\wt b(x)=\sum_{k\in\Z^2\setminus\{(0,0)\}} \wh b_k\, e^{2i\pi k\cdot x},
$$
and set for any integer $n,$ 
$$
\wt b_n(x):=\sum_{1\leq|k|\leq n}  \wh b_k\, e^{2i\pi k\cdot x}.
$$
By Cauchy-Schwarz inequality, it is easy to prove that
\begin{equation}\label{eq:BMO}
\|\wt b_n\|_\infty\leq C\sqrt{\log n}\,\|\nabla b\|_2.
\end{equation}
 Because the average of $\wt b_n$ is $0,$   one may write, thanks to \eqref{eq:average} that for all $c\in\R,$
 $$
-M\bar b=\int_{\T^2}(\rho -c)\wt b\,dx=\int_{\T^2} \rho\, \wt b_n\,dx+\int_{\T^2}(\rho-c)(\wt b-\wt b_n)\,dx.
$$
Therefore, using H\"older and Poincar\'e inequality, and also \eqref{eq:BMO},
$$\begin{aligned}
M|\bar b|&\leq\|\rho\|_{1}\|\wt b_n\|_{\infty}
+\|\rho-c\|_{2}\|\wt b-\wt b_n\|_{2}\\
&\leq C\Bigl(\sqrt{\log n}\,M+n^{-1}\|\rho-c\|_{2}\Bigr)\|\nabla b\|_{2}.
\end{aligned}$$
Then taking $n\approx\|\rho-c\|_{2}/M$ gives
\begin{equation}\label{eq:averagec}
|\bar b|\leq C \log^{\frac12}\biggl(\e+\frac{\|\rho-c\|_{2}}{M}\biggr)\|\nabla b\|_{2},
\end{equation}
which, combined with  \eqref{eq:zz} yields \eqref{eq:limitpoincare}.
\end{proof}

We used the following version of Desjardins' estimate in \cite{Des-CPDE}.
\begin{lem}\label{DLest}
 Let $\rho \in L_\infty(\T^2)$  with $\rho\geq0,$ and $u \in H^1(\T^2).$ Then, 
 there exists a  universal constant $C$ such that for all real number $c,$ we have
  \begin{equation}\label{eq:DLest-bis}
\biggl(\int_{\T^2} \rho u^4\,dx\biggr)^{\frac12}\leq C \|\sqrt\rho u\|_{2}
\|\nabla u\|_{2}\log^{\frac12}\biggl(\e+\frac{\|\rho-c\|_{2}}{M}
+\frac{\|\rho\|_{2}\|\nabla u\|_{2}^2}{\|\sqrt \rho u\|_{2}^2}\biggr)\cdotp\end{equation}
\end{lem}
\begin{proof} Let  $\wt u:=u-\bar u$ and  fix some $n\in\N.$ Then, keeping the same notation as in 
the above lemma and  using H\"older inequality, 
$$\begin{aligned}
\biggl(\int_{\T^2}\rho u^4\,dx\biggr)^{\frac12} &=\biggl(\int_{\T^2} \bigl(\bar u+\wt u_n+(\wt u-\wt u_n)\bigr)^2\,\rho u^2\,dx\biggr)^{\frac12}\\&\leq
|\bar u|\|\sqrt\rho u\|_{2}+\|\sqrt\rho u\|_{2}\|\wt u_n\|_{\infty}
+\|\rho\|_{2}^{\frac14}\|\wt u-\wt u_n\|_{8}\biggl(\int_{\T^2}\rho u^4\,dx\biggr)^{\frac14}\cdotp
\end{aligned}
$$
We thus have, using Young inequality and embedding $\dot H^{\frac34}(\T^2)\hookrightarrow L_8(\T^2),$
\begin{equation}\label{eq:des1}
\biggl(\int_{\T^2}\rho u^4\,dx\biggr)^{\frac12} \leq 2\|\sqrt\rho u\|_{2}\bigl(|\bar u|+\|\wt u_n\|_{\infty}\bigr)
+C\|\rho\|_2^{\frac12}\|\wt u-\wt u_n\|_{\dot H^{\frac34}}^2.
\end{equation}
Hence, taking advantage of \eqref{eq:BMO} and of 
\begin{equation}\label{eq:des3}
\|\wt u-\wt u_n\|_{\dot H^{\frac34}}\leq n^{-1/4}\|\nabla u\|_{2},\end{equation}
then plugging  \eqref{eq:des3} in \eqref{eq:des1}, we get
$$
\biggl(\int_{\T^2}\rho u^4\,dx\biggr)^{\frac12} \lesssim \|\sqrt\rho u\|_{2}\,|\bar u|
+\bigl(\sqrt{\log n} \|\sqrt\rho u\|_{2}+ n^{-\frac12}\|\rho\|_{2}^{\frac12}\|\nabla u\|_{2}\bigr)\|\nabla u\|_{2}.
$$
Taking $\displaystyle n\approx\frac{\|\rho\|_{2}\|\nabla u\|_{2}^2}{\|\sqrt \rho u\|_{2}^2}$
and using \eqref{eq:averagec} to bound $|\bar u|$ yields
the desired inequality. 
\end{proof}

%%%%%%%%%%%%%%%%%%%%%%%%%%%%%%%%%%%%%%%%

\section{End of the proof of time weighted estimates  in the 2D case}

We  here provide the reader with the proofs of Inequalities \eqref{eq:I2} and \eqref{eq:K2}.

\subsection*{Proof of \eqref{eq:I2}} We use \eqref{eq:decompo} to bound $I_2$  as follows: 
\begin{equation}\label{I2} I_2 = \int_{\T^2} \nabla v \cdot \nabla^2 \left(\cP v - \frac{1}{\nu} \nabla (-\Delta)^{-1}(\wt G +\wt P)\right) \cdot \dot v t \, dx=:I_{21}+I_{22}+I_{23}.\end{equation}
From \eqref{w20}, we know that 
\begin{equation}\label{w23}
\mu^2 \| \sqrt t\,\nabla^2 \cP v(t) \|_2^2 +\| \sqrt t\,\nabla G(t)\|_2^2  \leq \rho^*\,\|\sqrt{\rho t}\,\dot v\|_2^2.
\end{equation}
Since we have
$$
 \frac{\mu}{(\rho^*)^{1/4}}\| \sqrt t\,\nabla^2 \cP v\|_{L_4(0,T;L_2)} 
 \leq \biggl(\frac{\mu T^{1/2}}{\sqrt{\rho^*}}\|\nabla^2 \cP v \|_{L_2(0,T;L_2)}\biggr)^{1/2}\! \biggl(\mu\| \sqrt t\,\nabla^2 \cP v \|_{L_\infty(0,T;L_2)}\biggr)^{1/2}
 $$
 and a similar inequality for $\nabla G,$ combining \eqref{w23} and   Proposition \ref{p:H1a} yields
\begin{equation}\label{eq:vL4}
 \|\sqrt t\,\nabla^2 \cP v \|_{L_4(0,T;L_2)} + \|\sqrt t\,\nabla G\|_{L_4(0,T;L_2)} 
 \leq C_0T^{1/4}  \|\sqrt{t\rho}\, \dot v\|_{L_\infty(0,T;L_2)}^{1/2}.
\end{equation}
Therefore, putting together with \eqref{w22}, we gather that 
\begin{eqnarray}\label{w24}
\biggl| \int_0^T I_{21}\, dt\biggr|&\!\!\!\leq\!\!\!&\|\nabla v\|_{L_4(0,T\times\T^2)}
\|\sqrt t\,\nabla^2\cP v\|_{L_4(0,T;L_2)} \|\sqrt t\,\dot v\|_{L_2(0,T;L_4)}\nonumber\\ 
&\!\!\!\leq\!\!\!& C_0 T^{1/4} \|\sqrt{\rho t}\,  \dot v \|_{L_\infty(0,T;L_2)}^{1/2}\|\sqrt t\,\dot v\|_{L_2(0,T;L_4)}.
\end{eqnarray}
Term $I_{22}$ is almost the same: taking into account \eqref{eq:vL4}, we obtain
\begin{eqnarray}\label{w25}
\biggl|\int_0^T  I_{22} \,dt\biggr|&\!\!\!\leq\!\!\!&\nu^{-1}\|\nabla v\|_{L_4(0,T\times\T^2)}
\|\sqrt t\,\nabla G\|_{L_4(0,T;L_2)} \|\sqrt t\,\dot v\|_{L_2(0,T;L_4)}\nonumber\\ 
 &\!\!\!\leq\!\!\!& C_0\nu^{-1} T^{1/4} \|\sqrt{\rho t}\,  \dot v \|_{L_\infty(0,T;L_2)}^{1/2}\|\sqrt t\,\dot v\|_{L_2(0,T;L_4)}.
\end{eqnarray}
To handle  $I_{23},$ we integrate by parts several times and get (with the summation convention 
for repeated indices and the notation $\psi:=(-\Delta)^{-1}\wt P$):
$$
\begin{aligned}
I_{23}&=-\frac1\nu\int_{\T^2} \d_kv^j\,\d^3_{ijk}\psi\,\dot v^i t\,dx\\
&=\frac1\nu\int_{\T^2}\d_k\div v\,\d^2_{ik}\psi\,\dot v^i\,t\,dx
+\frac1\nu\int_{\T^2}\d_kv^j\,\d^2_{ik}\psi\,\d_j\dot v^i\,t\,dx\\
&=\frac1{\nu^2}\int_{\T^2}\d_k\wt P\,\d^2_{ik}\psi\,\dot v^i\,t\,dx+
\frac1{\nu^2}\int_{\T^2}\d_k\wt G\,\d^2_{ik}\psi\,\dot v^i\,t\,dx
+\frac1\nu\int_{\T^2}\d_kv^j\,\d^2_{ik}\psi\,\d_j\dot v^i\,t\,dx\\
&=-\frac1{\nu^2}\int_{\T^2}\wt P\,\d^2_{ikk}\psi\,\dot v^i\,t\,dx-
\frac1{\nu^2}\int_{\T^2}\wt P\,\d^2_{ik}\psi\,\d_k\dot v^i\,t\,dx\\
&\hspace{5cm}+\frac1{\nu^2}\int_{\T^2}\d_k\wt G\,\d^2_{ik}\psi\,\dot v^i\,t\,dx
+\frac1\nu\int_{\T^2}\d_kv^j\,\d^2_{ik}\psi\,\d_j\dot v^i\,t\,dx.
\end{aligned}
$$
As  $\psi:=(-\Delta)^{-1}\wt P,$  integrating by parts
one more time in the first term of the right-hand side just above, we conclude that
\begin{multline}\label{w26}
I_{23}= -\frac1{2\nu^2}\int_{\T^2}\wt P^2\,\div\dot v\,t\,dx-
\frac1{\nu^2}\int_{\T^2}\wt P\,\d^2_{ik}\psi\,\d_k\dot v^i\,t\,dx\\
+\frac1{\nu^2}\int_{\T^2}\d_k\wt G\,\d^2_{ik}\psi\,\dot v^i\,t\,dx
+\frac1\nu\int_{\T^2}\d_kv^j\,\d^2_{ik}\psi\,\d_j\dot v^i\,t\,dx.
\end{multline}
Hence,  using H\"older inequality
and the continuity of $\nabla^2(-\Delta)^{-1}$ on $L_4(\T^2),$ we get
$$\displaylines{
\biggl|\int_0^T I_{23}(t)\,dt\biggr|\lesssim\frac{\sqrt T}{\nu^2}
\biggl(\|\wt P\|_{L_4(0,T\times\T^2)}^2\|\sqrt t\,\nabla\dot v\|_{L_2(0,T\times\T^2)}
\hfill\cr\hfill+\|\nabla G\|_{L_2(0,T\times\T^2)}\|\wt P\|_{L_\infty(0,T;L_4)}\|\sqrt t\,\dot v\|_{L_2(0,T;L_4)}
\hfill\cr\hfill+\nu\|\nabla v\|_{L_2(0,T\times\T^2)}\|\wt P\|_{L_\infty(0,T\times\T^2)}\|\sqrt t\,\nabla\dot v\|_{L_2(0,T\times\T^2)}\biggr)\cdotp}
$$
Hence, thanks to \eqref{eq:energy} and  \eqref{w22b}, 
one can conclude that
\begin{equation}\label{eq:I23}
\biggl|\int_0^T I_{23}(t)\,dt\biggr|\leq  C_0\sqrt T\bigl(\nu^{-1}\|\sqrt t\,\nabla\dot v\|_{L_2(0,T;L_2)}
+\nu^{-2}\|\sqrt t\,\dot v\|_{L_2(0,T;L_4)}\bigr)\cdotp
\end{equation}

\subsection*{Proof of \eqref{eq:K2}}

We use  the decomposition
$K_2=K_{2,1}+K_{2,2}+K_{2,3}$ with 
$$
\displaylines{
K_{2,1}:=\int_{\T^2}(\nabla\cP v\cdot\nabla\div v)\cdot\dot vt\,dx,\qquad
K_{2,2}:=-\frac1\nu\int_{\T^2}(\nabla^2(-\Delta)^{-1}\wt G\cdot\nabla\div v)\cdot\dot v t\,dx\cr\andf
K_{2,3}:=-\frac1\nu\int_{\T^2}(\nabla^2(-\Delta)^{-1}\wt P\cdot\nabla\div v)\cdot\dot v t\,dx.}$$
In order to handle $K_{2,1},$ we integrate by parts (note that $\div \cP v=0$) 
and use the fact that $\nu\,\div v=\wt P+\wt G.$ We get, with the usual summation convention
$$K_{2,1}=
%-\frac1\nu\int_{\T^2}\d_i(\cP v)^j\,\wt P\,\d_j\dot v^i\,t\,dx-\frac1\nu\int_{\T^2}\d_i(\cP v)^j\,\wt G\,\d_j\dot v^i\,t\,dx\\&=
-\frac{\sqrt t}\nu\int_{\T^2}\d_i(\cP v)^j\,\wt P\,\d_j\dot v^i\,\sqrt t\,dx
+\frac1\nu\int_{\T^2}\d_i(\cP v)^j\,\sqrt t\d_j G\,\dot v^i\,\sqrt t\,dx.
%\end{aligned}
$$
Therefore,
\begin{eqnarray}\label{eq:K21}
\nu\biggl|\int_0^T K_{2,1}\,dt\biggr|&\!\!\!\leq\!\!\!& \sqrt T\|\nabla\cP v\|_{L_2(0,T;L_2)}\|\wt P\|_{L_\infty(0,T;L_\infty)}
\|\sqrt t\,\nabla\dot v\|_{L_2(0,T;L_2)}\nonumber\\
&&\qquad\qquad+\|\nabla\cP v\|_{L_4(0,T;L_4)}\|\sqrt t\, \nabla G\|_{L_4(0,T;L_2)}\|\sqrt t\,\dot v\|_{L_2(0,T;L_4)}\nonumber\\
&\!\!\!\leq\!\!\!& C_0\bigl(\sqrt T\|\sqrt t\,\nabla\dot v\|_{L_2(0,T;L_2)}+T^{1/4}\|\sqrt{\rho t}\, \dot v\|_{L_\infty(0,T;L_2)}^{1/2} 
 \|\sqrt t\,\dot v\|_{L_2(0,T;L_4)}\bigr)\cdotp
\end{eqnarray}

Next, integrating by parts in $K_{2,2}$  and using $\nu\,\div v=\wt P+\wt G$ gives
$$\begin{aligned}
\nu K_{2,2}&=-\int_{\T^2}\sqrt t\nabla G\cdot \sqrt t\dot v\: \div v\,dx
+\sqrt t \int_{\T^2}\nabla^2(-\Delta)^{-1}\wt G\cdot\sqrt t\nabla\dot v\,\div v\,dx\\
&=-\int_{\T^2}\sqrt t\nabla G\cdot \sqrt t\dot v\: \div v\,dx
+\frac{\sqrt t}\nu \int_{\T^2}\nabla^2(-\Delta)^{-1}\wt G\cdot\sqrt t\nabla\dot v\,\wt G\,dx\\&\hspace{7cm}
+\frac{\sqrt t}\nu \int_{\T^2}\nabla^2(-\Delta)^{-1}\wt G\cdot\sqrt t\nabla\dot v\,\wt P\,dx\end{aligned}
$$
from which we get
\begin{eqnarray}\label{eq:K22}
\nu\biggl|\int_0^T K_{2,2}\,dt\biggr|&\!\!\!\leq\!\!\!& \|\div v\|_{L_4(0,T;L_4)}\|\sqrt t\,\dot v\|_{L_2(0,T;L_4)}
\|\sqrt t\,\nabla G\|_{L_4(0,T;L_2)}\nonumber\\
&&+\nu^{-1}\sqrt T\|\wt G\|_{L_4(0,T;L_4)}\|\sqrt t\nabla\dot v\|_{L_2(0,T;L_2)}\bigl(
\|\wt G\|_{L_4(0,T;L_4)}+\|\wt P\|_{L_4(0,T;L_4)}\bigr)\nonumber\\
&\!\!\!\leq\!\!\!&C_0T^{1/4}\|\sqrt t\,\dot v\|_{L_2(0,T;L_4)}\|\sqrt{\rho t}\, \dot v\|_{L_\infty(0,T;L_2)}^{1/2} 
+\nu^{-1/2}\sqrt T\|\sqrt t\,\nabla\dot v\|_{L_2(0,T;L_2)}.\end{eqnarray}

Finally, using again the notation $\psi:=(-\Delta)^{-1}\wt P,$ we have
$$
\begin{aligned}
\nu^2K_{2,3}&=-\int_{\T^2}\d_i\d_j\psi\,\d_j\wt P\,t\dot v^i\,dx
-\int_{\T^2}\d_i\d_j\psi\,\d_jG\,t\dot v^i\,dx\\
&=\int_{\T^2}\d_i\d^2_j\psi\,\wt P\,t\dot v^i\,dx+\int_{\T^2}\d_i\d_j\psi\,\wt P\,t\d_j\dot v^i\,dx
-\int_{\T^2}\d_i\d_j\psi\,\d_jG\,t\dot v^i\,dx\\
&=\frac1{2}\int_{\T^2}\wt P^2\,t\div\dot v\,dx+\int_{\T^2}\d_i\d_j\psi\,\wt P\,t\d_j\dot v^i\,dx
-\int_{\T^2}\d_i\d_j\psi\,\d_jG\,t\dot v^i\,dx.
\end{aligned}
$$
Therefore,
\begin{eqnarray}\label{eq:K23}
\nu^2\biggl|\int_0^T K_{2,3}\,dt\biggr|&\!\!\!\lesssim\!\!\!& \sqrt T\|\wt P\|_{L_4(0,T;L_4)}^2\|\sqrt t\nabla\dot v\|_{L_2(0,T;L_2)}\nonumber\\
&&+\|\wt P\|_{L_4(0,T;L_4)}\|\sqrt t\,\nabla G\|_{L_4(0,T;L_2)}\|\sqrt t\,\dot v\|_{L_2(0,T;L_4)}\nonumber\\
&\!\!\!\leq\!\!\!& C_0\sqrt{\nu T}\|\sqrt t\nabla\dot v\|_{L_2(0,T;L_2)}
+C_0(\nu T)^{1/4}\|\sqrt{\rho t}\, \dot v\|_{L_\infty(0,T;L_2)}^{1/2} 
 \|\sqrt t\,\dot v\|_{L_2(0,T;L_4)}.
\end{eqnarray}
Plugging  \eqref{eq:H1}, \eqref{eq:vL4}, \eqref{w22}, \eqref{w22a} and \eqref{w22b}
in \eqref{eq:K21}, \eqref{eq:K22}  and  \eqref{eq:K23-3} yields \eqref{eq:K2}. 
%$$\nu \biggl|\int_0^T K_2 \,dt\biggr| \leq C_0T^{1/4}\Bigl(\|\sqrt{\rho t}\, \dot v\|_{L_\infty(0,T;L_2)}^{1/2} 
 %\|\sqrt t\,\dot v\|_{L_2(0,T;L_4)}\!+\!T^{1/4}\|\sqrt t\,\nabla\dot v\|_{L_2(0,T;L_2)}\Bigr)\cdotp$$

%%%%%%%%%%%%%%%%%%%%%%%%%%%%%%%%%%%%%%

\section{The three-dimensional case}\label{s:d=3}

This section is devoted to extending our existence result to the three-dimensional torus. 
For expository purpose, we focus on the global-in-time  issue for small data, 
although a similar statement may be proved locally in time for large data.
\begin{thm}\label{thm:global3} Let $v_0$ be in $H^1(\T^3)$ and $\rho_0$ be a bounded
nonnegative function on $\T^3.$  
Assume that $P(\rho)=a\rho^\gamma$ for some $\gamma\geq1$ and $a>0.$ 
There exists $\nu_0>0$ depending only on $\mu,$ $\gamma,$ $a$ and on the norms of the data, and $c_0>0$ such that  if \begin{equation}\label{eq:smalld3}
\mu\|\nabla \cP v_0\|_2^2+\frac1\nu\|\wt P_0\|_2^2
+\nu\|\div v_0\|_2^2\leq  c_0 \frac{\mu^5}{(\rho^*)^3E_0},
\end{equation}
 then System \eqref{CNS} has 
a  unique global solution with  the same properties as in Theorem \ref{thm:global2}. 
\end{thm}
The general strategy is basically the same as for the two-dimensional case, 
except that the smallness condition spares our using the logarithmic
interpolation inequality (that does not hold for $d=3$). 
We just underline what has to be changed in the main steps of the proof.

\subsection*{Step 1: Sobolev estimates for the velocity} The counterpart of Proposition 
\ref{p:H1a} reads:
\begin{prop}\label{p:H1b} Let $(\rho,v)$ be a smooth solution of \eqref{CNS} 
on $[0,T]\times\T^3,$ fulfilling    \eqref{eq:normalization} and \eqref{eq:bounded}.
Assume that $P$ satisfies \eqref{eq:condP} and denote $P^*:=\|P\|_\infty.$ 
Under condition \eqref{eq:smalld3} and for large enough $\nu,$
there exists a constant $C$ such that for all $t\in[0,T],$ we have
$$\displaylines{
\mu \|\nabla\cP v(t)\|_2^2+\frac{1}{\nu}\bigl(\|\wt G(t)\|_2^2+\|\wt P(t)\|_2^2\bigr)
  +\frac{P^*}{\nu}\|(\sqrt\rho\,v)(t)\|_2^2+\frac{P^*}{\nu}\|e(t)\|_1\hfill\cr\hfill
  +\int_0^t\biggl(\|\sqrt\rho\,\dot v\|_{2}^2+\frac{\mu^2}{\rho^*}\|\nabla^2\cP v\|_{2}^2+\frac1{\rho^*}\|\nabla G\|_{2}^2+\nu\|\div v\|_2^2+\frac{\mu P^*}{\nu}\|\nabla v\|_{2}^2\biggr)d\tau
 \hfill\cr\hfill \leq C\Bigl(\mu\|\nabla \cP v_0\|_2^2+\frac1\nu\|\wt P_0\|_2^2
+\nu\|\div v_0\|_2^2\Bigr)\cdotp}
  $$
\end{prop}
\begin{proof}  In order to be able to consider general initial data \emph{with large energy},
 it is suitable  to  modify the definition of $\cE$ as follows:
 \begin{equation}\label{eq:E3def}
 \cE:=\wt\cE+\frac{P^*}\nu E+\frac{P(1)}{2\nu}(P^*-P(1)),
 \end{equation}
where $\wt\cE$ is still defined by \eqref{eq:tE}. Owing to \eqref{eq:wtcE} that is also valid for $d=3$, we have
  \begin{equation}\label{eq:E3ineq}
  \cE\geq\frac\mu2 \|\cP v\|_2^2+\frac{1}{2\nu}\bigl(\|\wt G\|_2^2+\|\wt P\|_2^2\bigr)
  +\frac{P^*}{2\nu}\|\sqrt\rho\,v\|_2^2+\frac{P^*}{2\nu}\|e\|_1.
  \end{equation}
  Then, we start from Inequality \eqref{eq:e2b} that is valid in any dimension and,  instead of \eqref{eq:interpo-des},  we   use that
  \begin{equation}\label{eq:L4}
 \biggl(\int_{\T^3}\rho|v|^4\,dx\biggr)^{1/2}\leq  \biggl(\int_{\T^3}\rho|v|^2\,dx\biggr)^{1/4} 
 \biggl(\int_{\T^3}\rho|v|^6\,dx\biggr)^{1/4}\lesssim(\rho^*)^{1/4}\|\sqrt\rho\,v\|_2^{1/2}\|\nabla v\|_2^{3/2}.
 \end{equation}
 One  can  thus bound the  right-hand side of \eqref{eq:e2c} as follows: 
 $$\begin{aligned}
3\Int_{\T^3} \rho|v\cdot\nabla \cP v|^2\,dx&\leq 3\sqrt{\rho^*}
\|\rho^{1/4}v\|_4^2\|\nabla\cP v\|_4^2\\
&\leq C{\rho^*}^{3/4}\|\sqrt\rho v\|_{2}^{1/2}\|\nabla v\|_{2}^{3/2}\|\nabla\cP v\|_{2}^{1/2}
\|\nabla^2\cP v\|_2^{3/2}\\&
\leq\!\Frac{\mu^2}{4\rho^*}\|\Delta\cP v\|_{2}^2 
+C\biggl(\frac{\rho^*}{\mu}\biggr)^6 \|\sqrt\rho v\|_{2}^2\|\nabla v\|_{2}^6\|\nabla \cP v\|_{2}^2,
\end{aligned}$$
$$\begin{aligned}
\frac{3}{\nu^2}\Int_{\T^3}\rho\, \Bigl|\,v\cdot\left[ \nabla^2(-\Delta)^{-1}\wt G\right]\Bigr|^2\,dx 
&\leq C(\rho^*)^{1/2}\nu^{-2}\|\rho^{1/4}v\|_4^2\|\wt G\|_4^2\\
&\leq C(\rho^*)^{3/4}\nu^{-2} \|\sqrt\rho v\|_{2}^{1/2}\|\nabla v\|_{2}^{3/2}\|\wt G\|_{2}^{1/2}\|\nabla G\|_2^{3/2}\\
&\leq\Frac{1}{8\rho^*}\|\nabla G\|_{2}^2+ C\frac {(\rho^*)^6}{\nu^8} \|\sqrt\rho v\|_{2}^2\|\nabla v\|_{2}^6\|\wt G\|_{2}^2,
\end{aligned}$$
and also, thanks to Inequality \eqref{eq:poincarep}, 
$$\begin{aligned}
\frac{3}{\nu^2}\Int_{\T^3}\rho |v\cdot\nabla^2\Delta^{-1}\wt P|^2\,dx&\leq C\frac{\rho^*}{\nu^{2}}\|v\|_{2}^{1/2}\|\nabla v\|_{2}^{3/2}\|\wt P\|_4^2\\&\leq C\frac{(\rho^*)^{3/2}}{\nu^2}\|\nabla v\|_{2}^{2}\|\wt P\|_2\|\wt P\|_\infty\\
&\leq \frac{\mu P^*}{4\nu}\|\nabla v\|_2^2 +C\frac{(\rho^*)^{3}}{\mu\nu^3}
\frac{\|\wt P\|_\infty^2}{P^*}\|\nabla v\|_2^2\|\wt P\|_2^2.
\end{aligned}
$$
 Next, instead of \eqref{eq:e4e}, we  write that,  in light of Inequality \eqref{eq:poincarep} with $p=2$,
 we have 
 $$\begin{aligned}
 -\frac1\nu\int_{\T^3}Pv\cdot\nabla G&\leq\frac1\nu\|P\|_\infty(1+\|\wt\rho\|_2)\|\nabla v\|_2\|\nabla G\|_2\\
 &\leq\Frac{1}{8\rho^*}\|\nabla G\|_{2}^2+2\frac{\rho^*}{\nu^2}\|P\|_\infty^2(1+\|\wt\rho\|_2)^2
 \|\nabla v\|_2^2.
 \end{aligned}
 $$
 Therefore, the right-hand side of Inequality \eqref{eq:e8} becomes
 $$\displaylines{
 2\frac{(\rho^*)^3}{\nu^2}\|P\|_\infty^2 \|\nabla v\|_2^2+ 
 \frac2{\nu^2}\int_{\T^3} \wt G^2\,h\,dx\hfill\cr\hfill
 +C\|\sqrt \rho v\|_2^2\|\nabla v\|_2^6\biggl( \biggl(\frac{\rho^*}{\mu}\biggr)^6 \|\nabla \cP v\|_{2}^2
 +\frac {(\rho^*)^6}{\nu^8}\|\wt G\|_{2}^2\biggr) +C\frac{(\rho^*)^{3}}{\mu\nu^3}
\frac{\|\wt P\|_\infty^2}{P^*}\|\nabla v\|_2^2\|\wt P\|_2^2.}
 $$
Then, following the computations leading to \eqref{eq:e9} and assuming that $\nu$ satisfies
  \begin{equation}\label{eq:condnu3}
 \nu\geq 8(\rho^*)^3 \frac{P^*}\mu\andf   \nu^2\geq 8\|h\|_\infty\rho^*,
\end{equation}
 we get,
$$
 \frac d{dt}\cE+\cD\leq C\biggl(\|\sqrt\rho\,v\|_2^2\|\nabla v\|_2^6\biggl( \biggl(\frac{\rho^*}{\mu}\biggr)^6 \|\nabla \cP v\|_{2}^2
 +\frac {(\rho^*)^6}{\nu^8}\|\wt G\|_{2}^2\biggr) +\frac{(\rho^*)^{3}}{\mu\nu^3}
\frac{\|\wt P\|_\infty^2}{P^*}\|\nabla v\|_2^2\|\wt P\|_2^2\biggr)
 $$
 with 
$$\cD:= \frac14\|\sqrt\rho\,\dot v\|_{2}^2+\frac{\mu^2}{4\rho^*}\|\nabla^2\cP v\|_{2}^2+\frac1{8\rho^*}\|\nabla G\|_{2}^2+\frac12\int_{\T^2}(\div v)^2(\nu+h)\,dx+\frac{\mu P^*}{2\nu}\|\nabla v\|_{2}^2.$$
 Hence, remembering Inequality \eqref{eq:E3ineq},  one can conclude that  if 
  $\nu$ satisfies \eqref{eq:condnu3}, then we have 
the  differential inequality 
 \begin{equation}\label{eq:E3}
 \frac d{dt}\cE+\cD\leq C\|\nabla v\|_2^2\,\cE \biggl(\frac{(\rho^*)^{3/2}\|\wt P\|_\infty^2}{\nu^3\mu^{1/2}}
 +\frac{(\rho^*)^6}{\mu^9}E_0\,\cE^2\biggr)\cdotp
 \end{equation}
 Setting 
 $$X(t)=\cE(t)+\int_0^t\cD\,d\tau,\quad
 f(t):=  C\|\nabla v(t)\|_2^2,\quad A:= \frac{(\rho^*)^{3}P^*}{\mu\nu^2}\andf  B:= \frac{(\rho^*)^6}{\mu^9}E_0,$$
 Inequality \eqref{eq:E3} rewrites
 $$
 \frac d{dt} X\leq (AX + BX^3) f(t).
 $$
 This  may be integrated  into
 $$
 \frac{X(t)}{\sqrt{1 +c X^2(t)}}\leq  \frac{X(0)}{\sqrt{1 +c X^2(0)}}\, e^{A\int_0^t f(\tau)\,d\tau}\with
 c:=\frac BA\cdotp
 $$
 Bounding $f$ according to \eqref{eq:energy}, we see that  under the smallness condition 
\begin{equation}\label{eq:smallness3}
\cE^2_0<\frac AB\,\frac1{e^{\frac{2CAE_0}\mu}-1},
\end{equation}
we have 
 \begin{equation}\label{eq:global3D}
 %\frac{cX^2(0)}{1+cX^2(0)}e^{\frac{2CAE_0}\mu}<1\quad\hbox{then}\quad
X^2(t)\leq \frac{X^2(0)}{1+cX^2(0)}\biggl( \frac{e^{\frac{2CAE_0}\mu}}{1- \frac{cX^2(0)}{1+cX^2(0)}e^{\frac{2CAE_0}\mu}}\biggr)\quad\hbox{for all }\ 
t\geq0.
\end{equation}
%\begin{equation}\label{eq:smallness3}
%\cE^2(0)\,e^{\frac{2CAE_0}\mu}\leq \frac{3A}{4B}=\frac{3\mu^{\frac{17}2}\|\wt P\|_\infty^2}{4(\rho^*)^{\frac 92}E_0\nu^3},\end{equation}
\medbreak
Note that the largeness condition \eqref{eq:condnu3} on $\nu$ guarantees that 
the argument of the exponential function  above is very small.  Therefore,  the smallness condition \eqref{eq:smallness3}  
may be simplified into 
$$\cE_0\ll \frac{\mu^5}{(\rho^*)^3 E_0}\cdotp$$
%$$\cE(t)+\int_0^t\cD\,d\tau \leq 2\cE(0)\quad\hbox{for all }\ t\geq0.$$
 If  \eqref{eq:smalld3} holds true, then  that latter condition is fulfilled for  $\nu$ large compared
to $E_0^2.$ 
 \end{proof}
\begin{rmk}
Note that the smallness condition means that
 one can take the initial energy as large as we want provided that $\nu$ is large enough, 
but that $\div v_0$ must be $\cO(\nu^{-1/2}).$ At the same time, there is no smallness condition on 
$\rho_0-1$ whatsoever. 
\end{rmk}

%%%%%%%%%%%%%%%%%%%%%%%%

\subsection*{Step 2: Upper bound for the density}

In order to adapt Proposition \ref{p:boundrho2} to  the case  $d=3,$ the only changes are
in \eqref{eq:F1} and \eqref{eq:F2}.  As regards \eqref{eq:F1}, one may still start from \eqref{eq:d=2} then
combine with \eqref{eq:L4}  in order to get
\begin{equation}\label{eq:F5}
\|(-\Delta)^{-1}(\rho v)\|_\infty\leq C(\rho^*)^{\frac78}\|\sqrt\rho\,v\|_2^{\frac14}\|\nabla v\|_2^{\frac34}
\leq C(\rho^*)^{\frac78}E_0^{\frac18}\|\nabla v\|_2^{\frac34}.
\end{equation}
Next, instead of \eqref{eq:CLMS}, in order to bound the commutator term, we write that
\begin{equation}\label{eq:CLMS2}\|[v^j,(-\Delta)^{-1}\d_i\d_j]\rho v^i\|_\infty\lesssim
\|[\wt v^j,(-\Delta)^{-1}\d_i\d_j]\rho v^i\|_{W^{1,\frac{24}7}}
\lesssim\|\nabla v\|_{6}\|\rho v\|_8.
\end{equation}
Now, combining H\"older inequalities, Sobolev embedding and interpolation inequalities yields
$$
\begin{aligned}
\|\rho v\|_8&\leq (\rho^*)^{\frac{19}{20}}\|\sqrt \rho\, v\|_2^{\frac1{10}}\|v\|_{12}^{\frac9{10}}\\
&\leq C(\rho^*)^{\frac{19}{20}}\|\sqrt \rho\, v\|_2^{\frac1{10}}\|\nabla v\|_{\frac{12}5}^{\frac9{10}}\\
&\leq C(\rho^*)^{\frac{19}{20}}\|\sqrt \rho\, v\|_2^{\frac1{10}}\|\nabla v\|_{2}^{\frac{27}{40}}
\|\nabla v\|_6^{\frac9{40}}.
\end{aligned}
$$
Therefore, in $\T^3,$ Inequality \eqref{eq:CLMS} becomes
$$
\|[v^j,(-\Delta)^{-1}\d_i\d_j]\rho v^i\|_\infty\leq C(\rho^*)^{\frac{19}{20}}\|\sqrt \rho\, v\|_2^{\frac1{10}}\|\nabla v\|_{2}^{\frac{27}{40}}
\|\nabla v\|_6^{\frac{49}{40}}.$$
In order to bound the last term,  we use that
$$\begin{aligned} \|\nabla v\|_{6}&\lesssim \|\nabla\cP v\|_{6}+\nu^{-1}\bigl(\|\wt G\|_{6}+\|\wt P\|_{6}\bigr)\\
&\lesssim \|\nabla^2\cP v\|_2 + \nu^{-1}\bigl( \|\nabla G\|_{2}+\|\wt P\|_\infty\bigr)\cdotp\end{aligned}
 $$
Hence, using the energy conservation \eqref{eq:energy} and the definition of 
$\cE$ and $\cD,$
\begin{equation}\label{eq:F6}
\|[v^j,(-\Delta)^{-1}\d_i\d_j]\rho v^i\|_\infty\lesssim (\rho^*)^{\frac{19}{20}}E_0^{\frac1{20}}
\|\nabla v\|_{2}^{\frac{27}{40}}\biggl(\biggl(\frac{\sqrt{\rho^*}\cD}{\mu}\biggr)^{\frac{49}{80}}
+\biggl(\frac{\|\wt P\|_\infty}{\nu}\biggr)^{\frac{49}{40}}\biggr)\cdotp
\end{equation}
Plugging inequalities \eqref{eq:F5} and \eqref{eq:F6} in \eqref{eq:F+}, we get
\begin{multline}\label{eq:F7}
\|F^+(t)\|_\infty\leq\|F^+(0)\|_\infty  + \frac{\gamma-1}\gamma E_0
+ C\frac\gamma{\nu^2}(\rho^*)^{\frac78}E_0^{\frac18}
\int_0^te^{-\frac\gamma\nu(t-\tau)}\|\nabla v(\tau)\|_2^{\frac34}\,d\tau
\\+C\frac{ (\rho^*)^{\frac{19}{20}}}\nu E_0^{\frac1{20}}
\int_0^te^{-\frac\gamma\nu(t-\tau)}
\|\nabla v\|_{2}^{\frac{27}{40}}\biggl(\biggl(\frac{\sqrt{\rho^*}\cD(\tau)}{\mu}\biggr)^{\frac{49}{80}}
+\biggl(\frac{\|\wt P\|_\infty}{\nu}\biggr)^{\frac{49}{40}}\biggr)d\tau.
\end{multline}
 Since the  integrals in the right-hand side may be bounded in terms of the data
  according to the basic energy inequality \eqref{eq:energy} and to \eqref{eq:DE}, 
 we eventually get if $\nu$ is large enough:
 $$ \|F^+(t)\|_\infty\leq\|F^+(0)\|_\infty + C_0\nu^{-\frac{27}{80}} + \frac{\gamma-1}\gamma E_0$$
with $C_0$ depending only on $E_0,$ $\cE_0,$ $\|\rho_0\|_\infty,$
$\mu$ and $\gamma.$
From this point, one can conclude as in the two-dimensional case that 
\eqref{eq:rhomax} is fulfilled if $\nu$ is large enough.

%%%%%%%%%%%%%%

\subsection*{Step 3: Time weighted estimates}
As in the 2D case, the starting point is Identity \eqref{mom-d}. 
However, Inequality \eqref{w22} that has been used all the time has to be replaced with 
an estimate  for $t^{1/8}\nabla v$ in  $L_4(0,T\times\T^3)$:
 we write that the previous steps and \eqref{w23} imply that
\begin{eqnarray}\label{w22a3}
\|t^{1/8}\nabla\cP v\|_{L_4(0,T\times\T^3)}&\!\!\!\leq\!\!\!& \|t^{1/4}\nabla\cP v\|_{L_\infty(0,T;L_3)}^{1/2}
\|\nabla\cP v\|_{L_2(0,T;L_6)}^{1/2}\nonumber\\
&\!\!\!\lesssim\!\!\!& \|\nabla\cP v\|_{L_\infty(0,T;L_2)}^{1/4} \|\sqrt t\,\nabla^2\cP v\|_{L_\infty(0,T;L_2)}^{1/4}
\|\nabla^2\cP v\|_{L_2(0,T\times\T^3)}^{1/2}\nonumber\\
&\!\!\!\leq\!\!\!& C_0\|\sqrt{\rho t}\,\dot v\|_{L_\infty(0,T;L_2)}^{1/4},
\end{eqnarray}
where the meaning of $C_0$ is the same as in the two-dimensional case.
\medbreak
Similarly, we have
\begin{eqnarray}\label{w22b3}
\|t^{1/8} \nabla^2(-\Delta)^{-1}\wt G\|_{L_4(0,T\times \T^3)}&\!\!\!\lesssim\!\!\!&  \|\wt G\|_{L_4(0,T\times \T^3)}\nonumber\\
&\!\!\!\lesssim\!\!\!& \|\wt G\|_{L_\infty(0,T;L_2)}^{1/4} \|\sqrt t\nabla G\|_{L_\infty(0,T;L_2)}^{1/4}
\|\nabla G\|_{L_2(0,T\times\T^3)}^{1/2}\nonumber\\
&\!\!\!\leq\!\!\!& C_0\nu^{1/8} \|\sqrt{\rho t}\,\dot v\|_{L_\infty(0,T;L_2)}^{1/4}.\end{eqnarray}
Since \eqref{w22b} is valid in any dimension, one can conclude that
\begin{equation}\label{w22-3}
 \|t^{1/8}\nabla v\|_{L_4(0,T\times \T^3)} \leq C_0\bigl(\|\sqrt{\rho t}\,\dot v\|_{L_\infty(0,T;L_2)}^{1/4} +\nu^{-3/4}T^{1/8}\bigr)\cdotp
\end{equation}

\subsubsection*{Substep 1} Compared to $d=2,$ the only change lies in the estimate for 
$\int_{\T^3} \rho\,\div v\,|\dot v|^2t\,dx.$
Now, still using that $\div v=\nu^{-1}(\wt P+\wt G),$ we write that
  $$\begin{aligned}
 \int_{\T^3}\rho\,\div v|\dot v|^2t\,dx&\leq \nu^{-1}\int_{\T^2} (\wt P+\wt G)\rho|\dot v|^2t\,dx\\
 &\leq \nu^{-1}\bigl(\|\wt P\|_\infty\|\sqrt{\rho t}\,\dot v\|_2^2 + \sqrt{\rho^*}\|\wt G\|_3\|\sqrt t\,\dot v\|_6\|\sqrt{\rho t}\,\dot v\|_2\bigr)\\
 &\leq C_0\nu^{-1}\bigl(\|\sqrt{\rho t}\,\dot v\|_2^2+\|\wt G\|_2^{1/2}\|\nabla G\|_2^{1/2}\|\sqrt t\,\nabla\dot v\|_2 
 \|\sqrt{\rho t}\,\dot v\|_2\bigr)\cdotp\end{aligned}
 $$
  The first term may be treated as in the 2D case. As for the second one, we use the fact that \eqref{w20} ensures that
  $$
  \|\nabla G\|_2\leq\sqrt{\rho^*}\|\sqrt \rho\,\dot v\|_2.
  $$
  Hence, using Proposition \ref{p:H1a}  to bound $\|\wt G\|_2,$  we get  
  $$\begin{aligned}\int_0^T\!\!\!\! \int_{\T^2}\rho\,\div v|\dot v|^2t\,dx\,dt&\leq C_0\biggl(\nu^{-1}\!\!\int_0^T\|\sqrt{\rho t}\,\dot v\|_2^2
+\nu^{-3/4}\!\!\int_0^T\|\sqrt\rho\,\dot v\|_{2}^{1/2}\|\sqrt{\rho t}\,\dot v\|_2\|\sqrt t\,\nabla\dot v\|_2\,dt\biggr)\\
&\leq \frac{C_0}{\nu}\biggl(\int_0^T\!\!\|\sqrt{\rho t}\,\dot v\|_2^2
+\!\frac1{\sqrt\nu}\!\int_0^T\!\!\|\sqrt\rho\,\dot v\|_{2}\|\sqrt{\rho t}\,\dot v\|_2^2\,dt\biggr)
\!+\!\frac\mu2\int_0^T\!\!\|\sqrt t\,\nabla\dot v\|_2^2\,dt.\end{aligned}
$$
In the end, we thus obtain 
\begin{multline}\label{eq:step1-3}
\int_{\T^3} \frac{D}{Dt}(\rho \dot v) \cdot (t\dot v)\, dx \geq 
\frac 12 \frac{d}{dt} \int_{\T^3} \rho |\dot v|^2 t dx  - \frac12 \int_{\T^3} \rho|\dot v|^2 \,dx
\\-\frac{C_0}{\nu}\biggl(\int_0^T\!\!\|\sqrt{\rho t}\,\dot v\|_2^2
+\!\frac1{\sqrt\nu}\!\int_0^T\!\!\|\sqrt\rho\,\dot v\|_{2}\|\sqrt{\rho t}\,\dot v\|_2^2\,dt\biggr)
-\frac\mu2\int_0^T\!\!\|\sqrt t\,\nabla\dot v\|_2^2\,dt.
\end{multline}

\subsubsection*{Substep 2}

Instead of \eqref{WI1}, we use that, by virtue of  \eqref{w22-3}, 
\begin{equation}\label{WI1b}
\biggl| \int_0^T I_1(t)\,dt\biggr|
 \leq C_0 T^{1/4} \bigl(\|\sqrt{\rho t}\,\dot v\|_{L_\infty(0,T;L_2)}^{1/2}+\nu^{-3/2}T^{1/4}\bigr)\|\sqrt t\,\nabla \dot v \|_{L_2(0,T\times\T^3)}.
\end{equation}
For bounding $I_2,$ we decompose it into three parts, as in \eqref{I2}.  For $I_{2,1},$ we write that
$$\biggl|\int_0^T I_{21}(t)\,dt\biggr|\leq \|\sqrt t\,\nabla v\|_{L_\infty(0,T;L_3)}\|\nabla^2\cP v\|_{L_2(0,T\times\T^3)}
\|\sqrt t\,\dot v\|_{L_2(0,T;L_6)}.$$
Let us notice that
\begin{equation}\label{w30}
\|t^{1/4}\nabla  v\|_{L_\infty(0,T;L_3)}\leq C_0\bigl(\|\sqrt{\rho t}\,\dot v\|_{L_\infty(0,T;L_2)}^{1/2}
+\nu^{-2/3}T^{1/4}\bigr)\cdotp
\end{equation}
Indeed,  as already used for proving \eqref{w22a3} and \eqref{w22b3}, we have
\begin{eqnarray}\label{w30a}
\|t^{1/4}\nabla \cP v\|_{L_\infty(0,T;L_3)}
&\!\!\!\leq\!\!\!& C_0\|\sqrt{\rho t}\,\dot v\|_{L_\infty(0,T;L_2)}^{1/2},\\\label{w30b}
\|t^{1/4}\wt G\|_{L_\infty(0,T;L_3)}
&\!\!\!\leq\!\!\!& C_0\nu^{1/4}\|\sqrt{\rho t}\,\dot v\|_{L_\infty(0,T;L_2)}^{1/2},
\end{eqnarray}
and we have the obvious inequality
\begin{equation}\label{w30c}
\|\wt P\|_{L_\infty(0,T;L_3)}\leq\|\wt P\|_{L_\infty(0,T;L_2)}^{2/3}\|\wt P\|_{L_\infty(0,T\times\T^3)}^{1/3}\leq\nu^{1/3} C_0.
\end{equation}
Hence, combining  Sobolev embedding and \eqref{w30}, we obtain 
\begin{equation}\label{w31b}
\biggl|\int_0^TI_{21}\,dt\biggr|\leq C_0T^{1/4}\bigl(\|\sqrt{\rho t}\,\dot v\|_{L_\infty(0,T;L_2)}^{1/2}
+\nu^{-2/3}T^{1/4}\bigr)\|\sqrt t\,\nabla\dot v\|_{L_2(0,T\times\T^3)}.
\end{equation}
In order to bound $I_{22},$ we now write that 
$$\biggl|\int_0^TI_{22}\,dt\biggr|\lesssim \nu^{-1}\|\sqrt t\,\nabla v\|_{L_\infty(0,T;L_3)}
\|\nabla G\|_{L_2(0,T\times\T^3)}\|\sqrt t\, \dot v\|_{L_2(0,T;L_6)}.$$
Therefore, using the previous section and \eqref{w30}, we get
\begin{equation}\label{w32}
\biggl|\int_0^TI_{22}\,dt\biggr|\leq C_0\nu^{-1/2}T^{1/4} \bigl(\|\sqrt{\rho t}\,\dot v\|_{L_\infty(0,T;L_2)}^{1/2}
+\nu^{-2/3}T^{1/4}\bigr)\|\sqrt t\,\nabla\dot v\|_{L_2(0,T\times\T^3)}.
\end{equation}
To bound $I_{23},$ we use \eqref{w26} as in the two-dimensional case.
The first two terms of the decomposition may be bounded as before. For the third one, we use the fact that
$$\begin{aligned}
\biggl|\int_0^T\int_{\T^3}\d_k\wt G\,\d^2_{ik}\psi\,\dot v^i\,t\,dx\biggr|
&\lesssim \sqrt T\,\|\nabla G\|_{L_2(0,T\times\T^3)}\|\wt P\|_{L_\infty(0,T;L_3)}\|\sqrt t\,\dot v\|_{L_2(0,T;L_6)}\\
&\leq C_0\sqrt T\, \|\sqrt t\,\nabla \dot v\|_{L_2(0,T\times\T^3)}\end{aligned}
$$
and that
 $$\begin{aligned}\biggl|\int_{\T^3}\d_kv^j\,\d^2_{ik}\psi\,\d_j\dot v^i\,t\,dx\biggr|&\leq T^{3/8}\|t^{1/8}\nabla v\|_{L_4(0,T\times\T^3)}
 \|\wt P\|_{L_4(0,T\times\T^3)} \|\sqrt t\,\nabla \dot v\|_{L_2(0,T\times\T^3)}\\
 &\leq C_0\nu^{1/4}T^{3/8}\bigl(\|\sqrt{\rho t}\,\dot v\|_{L_\infty(0,T;L_2)}^{1/4}+\nu^{-3/4}T^{1/8}\bigr)
  \|\sqrt t\,\nabla \dot v\|_{L_2(0,T\times\T^3)}.\end{aligned}
$$
Hence, one can conclude that 
\begin{equation}\label{w33b}
\biggl|\int_0^TI_{23}\,dt\biggr|\leq C_0\biggl(\frac{\sqrt T}{\nu^{3/2}}+\frac{T^{3/8}}{\nu^{3/4}}\|\sqrt{\rho t}\,\dot v\|_{L_\infty(0,T;L_2)}^{1/4}\biggr)
  \|\sqrt t\,\nabla \dot v\|_{L_2(0,T\times\T^3)}.
\end{equation}
Putting together all the estimates of the second substep, we get 
\begin{multline}\label{eq:step2-3}
-\mu\int_0^T\!\!\! \int_{\T^3} \biggl(\frac{D}{Dt} \Delta v\biggr) \cdot \dot v \, t \, dx \,dt \geq 
 \mu\int_0^T \!\!\!\int_{\T^3} |\nabla \dot v|^2t \, dx \,dt\\
 -C_0T^{1/4}\bigl(\|\sqrt{\rho t}\,\dot v\|_{L_\infty(0,T;L_2)}^{1/2}
+\nu^{-2/3}T^{1/4}\bigr)\|\sqrt t\,\nabla\dot v\|_{L_2(0,T\times\T^3)}.
 \end{multline}

\subsubsection*{Substep 3}

To bound $K_1$ (defined in \eqref{w31}),  we write that
$$
\nu\left|\int_0^T K_1  \,dt\right| \leq CT^{1/4} \|t^{1/8}\nabla v\|_{L_4(0,T;L_4)}\|t^{1/8}(\wt P 
+ \wt G)\|_{L_4(0,T;L_4)} \|\sqrt t\,\div \dot v\|_{L_2(0,T;L_2)},$$
whence
\begin{multline}\label{eq:K1-3}
 \nu\left|\int_0^T K_1  \,dt\right|\leq C_0T^{1/4}\bigl(\|\sqrt{\rho t}\,\dot v\|_{L_\infty(0,T;L_2)}^{1/4}+\nu^{-3/4}T^{1/8} \bigr) \\\times\bigl(\nu^{1/8}\|\sqrt{\rho t}\,\dot v\|_{L_\infty(0,T;L_2)}^{1/4}
+\nu^{1/4}T^{1/8}\bigr) \|\sqrt t\,\div \dot v \|_{L_2(0,T;L_2)}.
\end{multline}
  We decompose $K_2$ as in the case $d=2.$ 
To bound $K_{2,1},$ we write that
\begin{eqnarray}\label{eq:K21-3}
\nu\biggl|\int_0^T K_{2,1}\,dt\biggr|&\!\!\!\leq\!\!\!&\sqrt T\|\nabla \cP v\|_{L_2(0,T\times\T^3)}
 \|\wt P\|_{L_\infty(0,T\times\T^3)} \|\sqrt t\,\nabla \dot v\|_{L_2(0,T\times\T^3)}\nonumber\\
&&\qquad\qquad\qquad +\|\sqrt t\nabla\cP v\|_{L_\infty(0,T;L_3)}\|\nabla G\|_{L_2(0,T\times\T^3)}
 \|\sqrt t\,\dot v\|_{L_2(0,T;L_6)}\nonumber\\
 &\!\!\!\leq\!\!\!& C_0\bigl(\sqrt T +T^{1/4}\|\sqrt{\rho t}\,\dot v\|_{L_\infty(0,T;L_2)}^{1/2}\bigr) 
  \|\sqrt t\,\nabla \dot v\|_{L_2(0,T\times\T^3)}.
\end{eqnarray}
For $K_{2,2},$ we have 
\begin{eqnarray}\label{eq:K22-3}
\nu\biggl|\int_0^T K_{2,2}\,dt\biggr|&\!\!\!\leq\!\!\!& \|\sqrt t\,\div v\|_{L_\infty(0,T;L_3)}\|\sqrt t\,\dot v\|_{L_2(0,T;L_6)}
\|\nabla G\|_{L_2(0,T\times\T^3)}\nonumber\\
&&\!\!\!\!\!\!\!\!\!\!\!\!\!\!+\nu^{-1}T^{1/4}\|t^{1/8}\wt G\|_{L_4(0,T;L_4)}\|\sqrt t\nabla\dot v\|_{L_2(0,T;L_2)}\bigl(
\|t^{1/8}\wt G\|_{L_4(0,T;L_4)}+\|t^{1/8}\wt P\|_{L_4(0,T;L_4)}\bigr)\nonumber\\
&\!\!\!\leq\!\!\!&C_0T^{1/4}\bigl(T^{1/8}\nu^{-3/4}+\nu^{-3/4}\|\sqrt{\rho t}\, \dot v\|_{L_\infty(0,T;L_2)}^{1/2}\bigr)
\|\sqrt t\,\nabla\dot v\|_{L_2(0,T;L_2)}.
\end{eqnarray}
Finally, we have
\begin{eqnarray}\label{eq:K23-3}
\nu^2\biggl|\int_0^T K_{2,3}\,dt\biggr|&\!\!\!\lesssim\!\!\!& \sqrt T\,\|\wt P\|_{L_4(0,T;L_4)}^2\|\sqrt t\nabla\dot v\|_{L_2(0,T;L_2)}\nonumber\\
&&\qquad\qquad\qquad+\sqrt T\,\|\wt P\|_{L_\infty(0,T;L_3)}\|\nabla G\|_{L_2(0,T\times\T^3)}\|\sqrt t\,\dot v\|_{L_2(0,T;L_6)}\nonumber\\
&\!\!\!\leq\!\!\!& C_0\sqrt{\nu T}\|\sqrt t\nabla\dot v\|_{L_2(0,T;L_2)}.
\end{eqnarray}
Plugging  \eqref{eq:H1}, \eqref{eq:vL4}, \eqref{w22}, \eqref{w22a} and \eqref{w22b}
in \eqref{eq:K21-3},   \eqref{eq:K22-3} and  \eqref{eq:K23-3} yields 
$$\nu \biggl|\int_0^T K_2 \,dt\biggr| \leq C_0\bigl(\sqrt T
+T^{1/4}\|\sqrt{\rho t}\,\dot v\|_{L_\infty(0,T;L_2)}^{1/2}\bigr) 
  \|\sqrt t\,\nabla \dot v\|_{L_2(0,T\times\T^3)}.$$
The conclusion of this step   is that, if $\nu$ is large enough then
\begin{multline}\label{eq:step3-3}
-(\nu\!-\!\mu)\int_0^T \!\!\!\int_{\T^3} \frac{D}{Dt} \nabla \div v\cdot\dot v\,t\, dx\, dt \geq  (\nu\!-\!\mu)\! \int_0^T\!\!\! \int_{\T^3} (\div\dot v)^2 t \,dx\,dt \\
- C_0\bigl(\nu^{1/4}\sqrt T+T^{1/4}\|\sqrt{\rho t}\, \dot v\|_{L_\infty(0,T;L_2)}^{1/2}\bigr)\|\sqrt t\,\nabla\dot v\|_{L_2(0,T;L_2)}\\
-C_0\bigl(\nu^{1/8}T^{1/4}\|\sqrt{\rho t}\,\dot v\|_{L_\infty(0,T;L_2)}^{1/2}+\nu^{3/8}\sqrt T\bigr)\|\sqrt t\, \div \dot v \|_{L_2(0,T;L_2)}.
\end{multline}

\subsubsection*{Substep 4}   

Term $L_1$ may still be bounded according to Inequality \eqref{w36}.
As for $L_2,$ we have
$$\begin{aligned}
\biggl|\int_0^TL_2(t)\,dt\biggr|&\leq \frac{1}{2\nu}\|\wt P\|_{L_4(0,T;L_4)}^2\|t\,\div\dot v\|_{L_2(0,T;L_2)}\\
&\qquad\qquad+\frac{1}\nu \|\wt P\|_{L_\infty(0,T;L_3)}\|\nabla G\|_{L_2(0,T;L_2)}\| t\,\dot v\|_{L_2(0,T;L_6)}\\
&\qquad\qquad\qquad\qquad+\|\wt P\|_{L_\infty(0,T;L_\infty)}\|\nabla v\|_{L_2(0,T;L_2)}\|t\,\nabla\dot v\|_{L_2(0,T;L_2)}\\
&\leq C_0\sqrt T\,\|\sqrt t\,\nabla\dot v\|_{L_2(0,T;L_2)}.
\end{aligned}
$$
So this step  gives
\begin{multline}\label{eq:step4-3}
\int_{\T^3} \frac{D}{Dt} \nabla P\cdot \dot v\,t\,dx \geq - \frac\nu4 \int_0^T \!\!\!\int_{\T^3} (\div\dot v)^2 \, t \,dx\,dt 
\\- {\|h\|_\infty T}\nu^{-1}\|\div v\|_{L_2(0,T;L_2)}^2- C_0\sqrt T\|\sqrt t\,\nabla\dot v\|_{L_2(0,T;L_2)}.
\end{multline}

\subsubsection*{Susbstep 5}  Combining Inequalities \eqref{eq:step1}, \eqref{eq:step2-3}, \eqref{eq:step3-3} and \eqref{eq:step4-3} 
yields  for large  $\nu,$
$$\displaylines{
 \|\sqrt{\rho t} \, \dot v\|_{L_\infty(0,T;L_2)}^2 \!+\! 2\mu
 \int_0^T \|\sqrt t\, \nabla\cP \dot v\|_2^2 \,dt \!+\!\frac{3\nu}2\int_0^T \|\sqrt t\,\div \dot v\|_2^2\,dt
 \leq  
2\! \int_0^T\! \|\div v\|_\infty   \|\sqrt{\rho t} \, \dot v\|_{2}^2\,dt  \hfill\cr\hfill
 +\|\sqrt\rho\,\dot v\|_{L_2(0,T;L_2)}^2 + {\|h\|_\infty T}\nu^{-1}\|\div v\|_{L_2(0,T\times\T^2)}^2
\! +C_0T^{1/4}\bigl(\nu^{1/8}\|\sqrt{\rho t}\,\dot v\|_{L_\infty(0,T;L_2)}^{1/2}\hfill\cr\hfill+\nu^{3/8}T^{1/4}\bigr)\|\sqrt t\, \div \dot v \|_{L_2(0,T;L_2)}
+C_0T^{1/4}\bigl(\|\sqrt{\rho t}\,\dot v\|_{L_\infty(0,T;L_2)}^{1/2}+T^{1/4}\bigr)\|\sqrt t\,\nabla\dot v\|_{L_2(0,T;L_2)}.}$$
Playing with Young inequality and Gronwall Lemma yields Prop. 
\ref{p:weight} for~$d=3.$

%\begin{prop}\label{p:weight3}
%Let $(\rho,v)$ be a smooth solution with bounded density of \eqref{CNS} on $[0,T]\times\T^3.$
%If $\nu$ is large enough and Condition  \eqref{eq:smallness3} is fulfilled then 
%the estimate of Proposition \ref{p:weight} holds true.
%\end{prop}
\medbreak 
It is now easy to adapt  Corollary \ref{c:c} to  the  3D case: we start from 
$$\|\wt G\|_\infty\lesssim \|\wt G\|_2^{1/4}\|\nabla G\|_6^{3/4}.$$
Hence,  remembering \eqref{w20} and using the embedding $\dot H^1(\T^3)\hookrightarrow L_6(\T^3),$
$$
\|\wt G\|_\infty\lesssim (\rho^*)^{3/4} \|\wt G\|_2^{1/4} t^{-3/8}\|\sqrt t\,\nabla\dot v\|_2^{3/4}.
$$
Therefore, as in the 2D case,
$$\|\wt G\|_{L_{1+\varepsilon}(0,T;L_\infty)}^{1+\varepsilon}\lesssim(\rho^*)^{\frac{3(1+\varepsilon)}{4}}\nu^{\frac{1+\varepsilon}8}
  (\nu^{-\frac12}\|\wt G\|_{L_\infty(0,T;L_2)})^{\frac{1+\varepsilon}4}
    \|\sqrt t\,\nabla\dot v\|_{L_2(0,T;L_2)}^{\frac{3(1+\varepsilon)}4}\biggl(\int_0^T t^{-\frac{3+3\varepsilon}{5-3\varepsilon}}\,dt\biggr)^{\frac{5-3\varepsilon}8},$$
and one can thus conclude that  $\div v$ is in $L_{1+\varepsilon}(0,T;L_\infty)$ provided that $\varepsilon <1/3.$ 
Bounding $\cP v$  is left to the reader.
\end{appendix}

\bigskip

\noindent{\bf Acknowledgments.} This work was partially supported by   ANR-15-CE40-0011.
The second author (P.B.M.) has been partly supported by the Polish National Science Centre's grant No 2018/29/B/ST1/00339.

\bigskip

%%%%%%%%%%%%%%%%%%%%%%%%%%%%%%%%%

\end{document}